\documentclass[11pt]{amsart}

\usepackage[english]{babel}
\usepackage[utf8]{inputenc}
\usepackage[margin=3cm]{geometry}
\usepackage{hyperref}
\usepackage{xcolor}
\hypersetup{
	colorlinks=true,
	linkcolor=blue,  
	urlcolor=cyan,
	pdftitle={},
	pdfauthor={},
}
\usepackage{graphicx}
\usepackage{amsmath, amsthm, amsfonts,stmaryrd,amssymb}
\usepackage{mathtools}
\usepackage{thmtools}
\usepackage[ruled,vlined]{algorithm2e}
\usepackage{enumerate}
\usepackage{todonotes}
\usepackage{enumitem}
\usepackage{subfig}
\usepackage[capitalize]{cleveref}
\usepackage{rotating}

\theoremstyle{plain}
\newtheorem{theorem}{Theorem}[section]
\newtheorem{remark}[theorem]{Remark}

\newtheorem{proposition}[theorem]{Proposition}

\theoremstyle{definition}

\newtheorem{example}[theorem]{Example}

\def\Om{\Omega}
\def\intO{\int_\Om}

\def\d{\text{\normalfont d}}
\def\x{\mathbf{x}}

\def\Id{\mathrm{Id}}
\def\prox{\text{prox}}

\def\Mc{\mathcal{M}}
\def\Pc{\mathcal{P}}
\def\Xc{\mathcal{X}}

\def\Bc{\mathcal{B}}
\def\Tc{\mathcal{T}}

\def\Gc{\mathcal{G}}
\def\Fc{\mathcal{F}}

\def\R{\mathbb{R}}
\def\Rn{\R^n}
\def\Rd{\R^d}

\def\C{\mathcal{C}} 
\newcommand\V{V}
\newcommand\Vd{{V^d}}

\newcommand\Co{C}

\def\vp{\varphi}
\def\<{\langle} \def\>{\rangle}
\DeclareMathOperator\Lip{Lip}

\DeclareMathOperator\FLip{Fr-Lip}
\DeclareMathOperator\interior{int}
\DeclareMathOperator\diam{diam}
\def\bN{\mathbb{N}}
\DeclareMathOperator\Tr{Tr}
\DeclareMathOperator\diag{diag}

\def\sign{\text{sign}}

\DeclareMathOperator*{\argmin}{argmin}

\DeclareMathOperator\KR{KR}

\let\div\underfined
\DeclareMathOperator\div{div}

\usepackage{array,multirow,makecell}
\newcommand{\PreserveBackslash}[1]{\let\temp=\\#1\let\\=\temp}
\newcolumntype{R}[1]{>{\raggedleft\arraybackslash }b{#1}}
\newcolumntype{L}[1]{>{\raggedright\arraybackslash }b{#1}}
\newcolumntype{C}[1]{>{\centering\arraybackslash }m{#1}}
\numberwithin{equation}{section}

\title[L1 optimal transport for FWI]{Unbalanced L1 optimal transport for vector valued measures and application to Full Waveform Inversion}

\begin{document}
	
	\date{\today}
	
	\author{Gabriele Todeschi}
	\address{Gabriele Todeschi (\href{mailto:gabriele.todeschi@univ-eiffel.fr}{\tt gabriele.todeschi@univ-eiffel.fr}), Univ. Grenoble Alpes, ISTerre, F-38000 Grenoble, France; now LIGM, Univ. Gustave Eiffel, CNRS, F-77454 Marne-la-Vallée, France
	} 

	\author{Ludovic Métivier}
	\address{Ludovic Métivier,
	(\href{mailto:ludovic.metivier@univ-grenoble-alpes.fr}{\tt ludovic.metivier@univ-grenoble-alpes.fr})
	Univ. Grenoble Alpes, CNRS, LJK, ISTerre, F-38000 Grenoble, France}
	
	\author{Jean-Marie Mirebeau}
	\address{Jean-Marie Mirebeau,
	(\href{mailto:jean-marie.mirebeau@ens-paris-saclay.fr}{\tt jean-marie.mirebeau@ens-paris-saclay.fr})
	CNRS, Centre Borelli, ENS Paris-Saclay, University Paris-Saclay, 91190, Gif-sur-Yvette, France}

	\begin{abstract}
		Optimal transport has recently started to be successfully employed to define misfit or loss functions in inverse problems. However, it is a problem intrinsically defined for positive (probability) measures and therefore strategies are needed for its applications in more general settings of interest. In this paper we introduce an unbalanced optimal transport problem for vector valued measures starting from the $L^1$ optimal transport. By lifting data in a self-dual cone of a higher dimensional vector space, we show that one can recover a meaningful transport problem. 
		We show that the favorable computational complexity of the $L^1$ problem, an advantage compared to other formulations of optimal transport, is inherited by our vector extension.
		We consider both a one-homogeneous and a two-homogeneous penalization for the imbalance of mass, the latter being potentially relevant for applications to physics based problems. In particular, we demonstrate the potential of our strategy for full waveform inversion, an inverse problem for high resolution seismic imaging.
	\end{abstract}

	\maketitle
	
	
	\section{Introduction}
	
	Optimal transport is the problem of finding the optimal way of reallocating distributions of mass. Among its many interesting features, it allows to define, in a natural and meaningful way, distances between probability measures. For this reason, it has started recently to be considered for applications to inverse problems to propose misfits between real and synthetic data. Inverse problems are usually highly non-linear and non-convex variational problems and a fundamental role is played by the misfit or loss function. The use of optimal transport based distances has revealed to be helpful in attenuating the lack of well-posedness of the problem and increasing the robustness of the inversion approach. Full Waveform Inversion (FWI), a seismic inversion problem, is one particular example, with a lot of recent literature (see for example the review paper \cite{metivier2022review}).

	\subsection{Optimal transport}

	Consider $\Om\subset\Rd$, compact and convex with non-empty interior, and two positive probability measures $\mu,\nu\in\Pc(\Om)$, that is non-negative measures with the same total mass, $\intO\mu=\intO\nu$. 
	Given a scalar non-negative real valued continuous function $c:\Omega\times\Omega\rightarrow\R_+$, the optimal transport problem between $\mu$ and $\nu$ is
	\begin{equation}\label{eq:OT}
		\text{OT}_c(\mu,\nu) \coloneqq\min_{\gamma\in\Pi(\mu,\nu)} \int_{\Om\times\Om} c(x,y) \d \gamma \,,
	\end{equation}
	where the set of admissible transport plans $\gamma$ is defined as
	\begin{equation}
	\label{eq:positive_couplings}
		\Pi(\mu,\nu)=\{\gamma\in\Pc(\Om\times\Om)\,, (\pi_1)_\#\gamma=\mu, (\pi_2)_\#\gamma=\nu\}\,,	
	\end{equation}
	i.e.\ the subset of probability measures on the product space $\Om\times\Om$ that have first and second marginals $\mu$ and $\nu$.
	The ground cost $c(x,y)$ is the cost for transporting one unit of mass from $x$ to $y$. An interesting case is provided by the choice $c(x,y)=|x-y|^p$. In this case \eqref{eq:OT} is called $L^p$ optimal transport problem and can be shown to provide a ($p$-th power of a) distance on the space of probability measures which is referred to as the $p$-Wasserstein distance, $W_p$. This distance translates in a natural way to the measure setting properties of the ground cost function. Importantly, it characterizes weak convergence of measures. We suggest for example the monographs \cite{santambrogio2015optimal,villani2003topics} for an introduction to the subject.

	In this paper we deal with the specific case $p=1$, namely the $L^1$ optimal transport problem.
	In this case problem \eqref{eq:OT} enjoys three equivalent formulations:
	\begin{align}
		W_1(\mu,\nu)\coloneqq &\min_{\gamma\in\Pi(\mu,\nu)} \intO |y-x| \d \gamma  \label{eq:OT1} \\
		=&\max_{\phi\in\C(\Omega)} \left\{\intO \phi (\mu-\nu), \; |\nabla \phi|_\infty \le 1  \right\} \label{eq:dual} \\
		=&\min_{\sigma\in\Mc(\Omega;\Rd)} \left\{ \int_\Om |\sigma| \,, \; \div (\sigma)=\mu-\nu \right\}. \label{eq:BP}
	\end{align}
	The maximization problem \eqref{eq:dual} is known as the dual formulation of the $L^1$ transport whereas \eqref{eq:BP} is known as Beckmann's minimal flow problem.
	These two formulations are particularly interesting for at least two reasons. Firstly, they appear as simpler problems as they present only local constraints differently from \eqref{eq:OT1}. In particular, notice that given a solution $\gamma$ it is possible to recover solutions $\sigma$ and $\phi$ (see \cite[Chapter 4]{santambrogio2015optimal}) but the opposite is not true. For this reason they are much more amenable for numerical computations. Furthermore, they can be generalized straightforwardly to more general signed measures, contrary to the original form of the problem.

	\subsection{Applications to inverse problems}
	
	Two major issues arise for the application of optimal transport to inverse problems. The first one is the computational cost. Problem \eqref{eq:OT} is evidently more expensive than evaluating simple $L^p$ distances which are usually considered. However, there has been a significant improvement on this side in the past years (see for example \cite{peyre2020computational,merigot2021optimal}), which mostly contributed to the recent development of optimal transport. We highlight also that for the specific applications to inverse problems computing an optimizer $\gamma$ is not strictly necessary: the only requirement is to be able to evaluate the misfit function, i.e.\ compute the transport cost, and compute its gradient with respect to the measures. This may allow to devise a cheaper numerical strategy. In this regard, the $L^1$ problem offers a cheaper alternative as previously remarked. 
	
	The second issue is that problem \eqref{eq:OT} is defined for probability measures, i.e.\ non-negative measures with fixed total mass. Generalizations have been proposed to deal with non-negative Radon measures \cite{piccoli2014generalized,chizat2018interpolating}. However, inverse problems may have to deal in general with signed data which therefore need to be handled properly.
	Trivial transformation such as transporting separately positive and negative parts \cite{engquist2014application}, taking the exponential or powers of the data \cite{qiu2017full}, adding positive mass everywhere \cite{yang2018analysis}, may not result effective. These transformations tend to destroy the structure of the data. Two effective and general ideas have been proposed in \cite{metivier2016measuring,metivier2018optimal}.
	
	In \cite{metivier2016measuring} the authors started from the dual formulation \eqref{eq:dual} of the $L^1$ problem to propose an optimal transport misfit for signed data. As previously remarked, this optimal transport problem has the advantage that it can be extended directly. They considered then the misfit provided by: 
	\begin{equation}\label{eq:KRscal}
		\KR(\mu,\nu)=\max_{\phi\in\C(\Omega)} \left\{\intO \phi (\mu-\nu), \; |\nabla \phi|_\infty\le1, \, |\phi|_{\infty}\le \lambda  \right\}
	\end{equation}
	for some parameter $\lambda\in\R_+$. The $L^\infty$ bound on the potential is added in order to be able to deal with measures that do not have the same total mass. Problem \eqref{eq:KRscal} is known as the Kantorovich-Rubenstein norm, or bounded-Lipschitz metric, and is a distance on the space of signed measures. This misfit function provides good results and the approach is also computationally appealing, since the problem can be solved effectively up to high dimensions.
	However, as problem \eqref{eq:KRscal} is not framed on positive measures it does not ensure sensitivity to horizontal shifts of the data (see Example \ref{ex:KRshift}). The growth of the optimal transport cost with respect to shifts of the measures is a key feature for its application to inverse problems.
	
	In \cite{metivier2018optimal}, the authors proposed to lift signed data to a higher dimensional space, where it can be regarded as a positive measure on its trace.
	Then, the positive lifted measures can be discretized and transported according to \eqref{eq:OT}.
	Since the value assigned to the measure is arbitrary, the ground cost $c(x,y)$ needs to be tuned properly in order to encode transport along the time direction and not along the amplitude one. The approach in the end is effective and leads to a robust inversion problem. On the other hand, it is computationally expensive and not suited to scale to high dimensions.

	\subsection{Contribution and structure of the paper}
	
	In this paper we propose a generalization of the $L^1$ optimal transport problem \eqref{eq:OT1} to vector valued measures. 
	Assuming to be able to lift, without loss of information, signed data into a (self-dual) cone $\Co$ of a general vector space $\V$, this generalization can serve as a meaningful misfit function in inverse problems. In this sense, our approach can be seen as an extension of the Kantorovich-Rubenstein approach \eqref{eq:KRscal} proposed in \cite{metivier2016measuring}, aiming at recovering a transport problem between positive measures. By extending the $L^1$ transport we take advantage of the lower computational complexity this problem offers.
	We consider as a specific application the FWI problem, where one can consider multi-component signals as data and naturally lift them in the space of symmetric positive semi-definite (PSD) valued measures. Some good preliminary results we obtained indicate the effectiveness of the approach.

	Extensions of the optimal transport problem to vector valued measures (and more specifically PSD valued measures) have already been proposed recently, starting from the seminal work \cite{ning2013MatrixvaluedMO}, where the authors proposed a generalization of the formulation \eqref{eq:OT} inspired by quantum mechanics arguments. Later on, several other works appeared, considering either a static formulation \cite{peyre2016quantum} or extensions of the dynamical formulation for the $L^2$ transport problem \cite{chen2017matrix,carlen2017gradient,mittnenzweig2017entropic,brenier2020optimal,chen2018vector}.
	The most natural way of extending optimal transport to a vector setting is probably to consider the dual formulation \eqref{eq:dual} (or the Beckmann's one \eqref{eq:BP}) of the $L^1$ transport.
	This has been originally done in \cite{ning2014metrics,chen2017matricial,ryu2018vector}, where the authors proposed an extension of $L^1$ transport to the setting of PSD valued measures inspired again from quantum mechanics. More recently, in \cite{ciosmak2021optimal,ciosmak2021matrix} a similar extension has been considered for functional analysis applications.
	
	In this work we propose to use an SDMM algorithm \cite{combettes2011proximal}, a primal-dual proximal splitting optimization approach, for solving our discrete vector valued optimal transport problem. SDMM is an instance of ADMM algorithm,
	a popular approach for solving optimal transport problems in the scalar setting \cite{benamou2000computational,benamou2016augmented,lavenant2018dynamical,carrillo2022primal,natale2022mixed}. It enjoys nice convergence guarantees, with a sub-linear convergence rate, and is effective as long as high accuracy is not requested. A similar primal-dual approach has been considered in \cite{ryu2018vector} for solving the vector valued extension of optimal transport introduced in \cite{chen2018vector}.
	Importantly, thanks to the characteristics of the $L^1$ problem and of the SDMM algorithm, the vectorial extension we propose only imply a linear increase in the computational complexity with respect to scalar optimal transport
	(see Remark \ref{rmk:complexity}).
	Another popular approach for solving discrete optimal transport is entropic regularization \cite{cuturi2013sinkhorn,peyre2020computational}.
	This regularization technique allows for the use of the efficient Sinkhorn algorithm, which enjoys a linear convergence rate and an $N\log N$ complexity per iteration when implemented on cartesian grids, making it suitable to scale to large dimensions.
	In \cite{peyre2016quantum} the authors developed an entropic approach for an optimal transport model introduced therein for PSD valued measures.
	With respect to the scalar case, this requires the additional cost of performing eigendecompositions of matrices of size $n'\times n'$ (the size of the PSD valued extension). The increase in complexity is more than linear with respect to $n'$, but at least for small values one can use closed form solutions and the additional cost is moderate.
	Nevertheless, the blurring effect of entropic regularization may require low values of the temperature parameter to provide meaningful descent directions in the inverse problem, vanishing the favorable computational complexity.
	The constant of linear convergence rapidly degrades in fact for smaller and smaller values of $\varepsilon$ (converging exponentially fast to one in the worst case scenario).
	A numerical approach based on Sequential Quadratic Programming (SQP) is proposed in \cite{chen2018efficient} for the extension of dynamical optimal transport to the vector valued setting introduced in \cite{chen2017matrix}.
	Dynamical formulations of optimal transport demand a higher computational cost by computing full trajectories, making these in general less attractive for inverse problems.
	Moreover, the SQP approach requires the solution of large matrices, obtained via Newton or quasi-Newton methods, with size given by the total number of variables. For this reason, these approaches may not be suited to scale to large dimensions. Importantly, as the total number of variables is proportional to the dimension of the vector extension, the computational cost increases more than linearly with respect to this latter.
	
	After setting our notation in Section \ref{sec:notation}, we present the unbalanced $L^1$ optimal transport problem for vector valued measures in Section \ref{sec:vectorL1}.
	We first introduce the straightforward extension of the dual \eqref{eq:dual} and the Beckmann's \eqref{eq:BP} formulations, in the balanced setting.
	These coincide with the ones already considered in \cite{ciosmak2021matrix}.
	Differently from these and the other works we mentioned just above, we consider however a general setting and justify our results via convex analysis arguments.
	In order to deal with general measures, we propose an unbalanced problem similarly to the (scalar) unbalanced $L^1$ problem introduced in \cite{piccoli2014generalized}. We show its equivalence to the vector extension of the Kantorovich-Rubenstein norm \eqref{eq:KRscal} (analogously to what has been done in \cite{piccoli2016properties} in the scalar case) and also to a relaxed version of the Beckmann's formulation.
	We then consider a different way of penalizing transport in the unbalanced model, namely an $L^2$ penalization instead of an $L^1$, which can be beneficial for applications to inverse problems.
	In Section \ref{sec:discrete_model} we present our finite difference discretization and the implementation of the SDMM algorithm.
	Finally, in Section \ref{sec:numerics} we present some numerical results and the application to the inverse problem of FWI.

	\section{Notations}\label{sec:notation}
	
	The standard Euclidean scalar product and norm are denoted by $x\cdot y$ and $|\cdot|$, for $x,y\in\R^d$, where $d$ is the dimension of the compact optimal transport domain $\Omega\subset \R^d$.
	We denote by $\V$ another real vector space of dimension $n$ endowed with its scalar product $\langle \cdot, \cdot \rangle_{\V}$ and the respective norm $|\cdot|_{\V}$. In $\V$ we consider a convex cone $\Co$.
	We recall that the dual cone of $\Co$ is defined as
	\[
	\Co^*=\{u\in\V: \langle u, v\rangle_\V\ge 0, \, \forall v\in\V \},
	\]
	and we restrict to considering cones such that $\Co\subseteq\Co^*$. In particular, this implies that $\langle v_1,v_2 \rangle_\V\ge 0$, $\forall v_1,v_2\in\Co$.
	Possible examples are $\V=\Rn$ and $\Co$ equal to the non-negative orthant $[\R_+]^n$, or $\V=\mathrm{S}^{n'}$ the space of symmetric matrices and $\Co=\mathrm{S}^{n'}_+$ the subspace of positive semi-definite matrices (where $n = n' (n'+1)/2$). For these two cases in particular it holds $\Co=\Co^*$ and the cones are called self-dual.
	We further introduce the vector space $\V^d$ endowed with the scalar product $\langle \cdot, \cdot \rangle_{\Vd}=\sum_{l=1}^d \langle \cdot, \cdot \rangle_{\V}$ and the respective norm $|\cdot|_{\Vd}$.
	One has a natural isomorphism of vector spaces $V^d \simeq L(\R^d;V)$, and as a result $V^d$ is also equipped with the operator norm $|\cdot|_{L(\R^d;V)}$. 
	The two norms are comparable, namely
	\begin{equation}
	\label{eq:Fr_op_norms}
	|\cdot|_{L(\R^d;V)} \leq |\cdot|_{V^d} \leq \sqrt d |\cdot|_{L(\R^d;V)},
	\end{equation}
	which is easily established using the singular value decomposition.
	However in dimension $d>1$ and for $n>1$ these norms are distinct, and $|\cdot|_{L(\R^d;V)}$ is not related to an inner product. For $A,B \in L(\R^d;V)$, $\<A,B\>_{\Vd} = \mathrm{Tr}(A^\top B)$ is known as the Frobenius inner product.
	
	We denote $\C(\Om;\V)$ and $\C(\Om;\Vd)$ the spaces of continuous functions defined on the compact domain $\Om \subset \R^d$ and valued in $\V$ and $\Vd$. These spaces are endowed respectively with the sup norms
	\[
	|\phi|_{\V,\infty}=\sup_{x\in\Omega} |\phi(x)|_\V \,, \quad |\psi|_{\Vd,\infty}=\sup_{x\in\Omega} |\psi(x)|_\Vd \,, \quad \text{for $\phi\in\C(\Om;\V)$, $\psi\in\C(\Om;\Vd)$}.
	\]
	The space of finite vectorial measures valued in $\V$ is $\Mc(\Omega;\V)$ and is the dual space of $\C(\Om;\V)$ via the duality product $(\phi,\mu)\mapsto \intO \langle\phi,\mu\rangle_{\V}$.
	The convex subspace of measures which take value in the cone $\Co$ is $\Mc(\Omega;\Co)$.
	A measure $\mu \in \Mc(\Omega;\V)$ belongs to $\Mc(\Omega;\Co)$ if and only if $\int_\Omega \<\varphi,\mu\>_V \geq 0$ for all $\vp \in C^0(\Omega;\Co^*)$.
	To any measure $\mu$ we can associate a positive scalar measure $|\mu|_\V\in\Mc(\Omega;\R_+)$ defined as
	\begin{equation}\label{eq:totmeasure}
	|\mu|_\V (K)=\sup_{\substack{\varphi\in\C(K;\V)\\|\phi|_{\V,\infty}\le1}} \int_K \langle\varphi,\mu\rangle_{\V}\,,\quad
	\text{for all Lebesgue measurable subsets $K \subset \Omega$}\,.
	\end{equation}
	The map 
	\[
	\mu\in\Mc(\Omega;\V)\longmapsto |\mu|_{\Mc(\Omega;\V)}\coloneqq|\mu|_\V(\Omega)
	\]
	is a norm on $\Mc(\Omega;\V)$ and is called the total variation norm.
	On $\Mc(\Omega;\V)$ we consider also the weak-* topology.
	Importantly, closed balls of $\Mc(\Omega;\V)$ with respect to the strong topology are weakly-* compact.
	Furthermore, from definition \eqref{eq:totmeasure}, the total variation norm is lower semi-continuous with respect to the weak-* convergence.
	Finally, we define the space of finite vectorial measures valued in $\Vd$, $\Mc(\Omega;\Vd)$, in duality with $\C(\Om;\Vd)$ through $(\psi,\sigma)\mapsto \intO \langle \psi, \sigma \rangle_{\Vd}$. The space $\Mc(\Omega;\Vd)$ endowed with the total variation norm $|\sigma|_{\Mc(\Omega;\Vd)}$, defined in the same way as above, is again a normed vector space. In the following we will use the notation:
	\[
	|\mu|_{\Mc(\Omega;\V)}=\intO |\mu|_\V \,, \quad |\sigma|_{\Mc(\Omega;\Vd)}=\intO |\sigma|_\Vd\,, \quad \text{for $\mu\in\Mc(\Omega;\V)$ and $\sigma\in\Mc(\Omega;\Vd)$}.
	\]
	
	On $\V$ and $\V^d$ we can consider two orthonormal bases, respectively $(v_k)_{k=1}^n$ and $w_{k,l}= v_k \otimes e_l$, for $(e_l)_{l=1}^d$ canonical basis of $\R^d$.
	Using these bases, a function $\phi$ belongs to $\C(\Om;\V)$ if and only if there exists $\, (a_k)_{k=1}^n\subset\C(\Omega;\R)$ such that $\phi(x)=\sum_{k=1}^n a_k(x) v_k, \forall x\in\Omega$, and analogously for $\psi\in\C(\Om;\Vd)$. In the same way a measure $\mu\in\Mc(\Omega;\V)$ can be written as $\mu=\sum_{k=1}^n b_k v_k$ for a sequence of real valued measures $(b_k)_{k=1}^n\subset\Mc(\Omega;\R)$, and analogously for $\sigma\in\Mc(\Omega;\Vd)$.

	The gradient $\nabla' \phi \in C(\Omega,V^d)$ of a function $\phi \in C^1(\Omega,V)$ is defined componentwise as follows
	\[
	\nabla' \phi = \sum_k  v_k \otimes \nabla \phi_k \quad \text{for } \phi = \sum_k \phi_k v_k\,.
	\]
	Likewise, for $\psi \in C^1(\Omega,V^d)$, the divergence $\div' \psi\in C(\Omega,V)$ is defined as
	\[
	\div' \psi = \sum_k \div (\psi_k) v_k \quad 
	\text{for }\psi = \sum_k v_k \otimes \psi_k\,.
	\]
	In the following, we drop the superscript over $\nabla$ and $\div$ for the sake of notational simplicity.
	A function $\phi\in\C(\Om;\V)$ is 1-Lipschitz if
	\[
	|\phi(y)-\phi(x)|_\V\le |y-x| \,, \quad \forall x,y\in\R^d \,.
	\]
	Since $\Omega$ is supposed to be convex,
	\begin{equation}
	\label{eq:lip_iff_grad1}
		\text{$\phi$ is 1-Lipschitz} \;\iff\; |\nabla \phi|_{L(\R^d;V), \infty}\le 1\,,
	\end{equation}
	where $|\nabla \phi|_{L(\R^d;V)}$ is the operator norm.
	By convention, $|\nabla \phi|_{L(\R^d;V), \infty} = \infty$ if the distributional derivative $\nabla \phi$ does not admit an $L^\infty$ representative; the quantity $|\nabla \phi|_{L(\R^d;V), \infty}$ can also be expressed using duality, similarly to \eqref{eq:Fr_norm_dual} below.
	There is no counterpart of the equivalence \eqref{eq:lip_iff_grad1} involving the Frobenius $|\cdot|_\Vd$ norm, and thus we define instead the set
	\begin{equation}
	\label{eqdef:FrLip}
		\FLip(\Omega,V) \coloneqq 
		\{\phi\in \C(\Omega; V)\,, \  |\nabla\phi|_{V^d,\infty} \leq 1 \}.
	\end{equation}
	We refer to $|\nabla\phi|_{V^d,\infty}$ as the Fr-Lipschitz constant, and note again that 
	it is defined as $\infty$ by convention if the distributional derivative $\nabla \phi$ has no $L^\infty$ representative. 
	By definition of the distributional derivative, and since $L^\infty = (L^1)^*$, one has for any $\phi \in \C(\Omega; V)$
	\begin{equation}
	\label{eq:Fr_norm_dual}
		|\nabla\phi|_{V^d,\infty} \coloneqq \sup_{\psi \in \C^\infty_c(\interior(\Omega);V^d)} \left\{ \int_\Om \<\phi,\, \div \psi\>_V\,, \; |\psi|_{V^d,1}\le 1 \right\}, 
	\end{equation}
	where $\C^\infty_c$ denotes the space of smooth and compactly supported test functions, $\interior(\Omega)$ stands for the interior of the domain and $|\psi|_{V^d,1} := \int_\Omega |\psi|_{V^d}$.
	Notice that $1\text{-}\Lip(\Omega,V) \subset \FLip(\Omega,V) \subset \sqrt d\text{-}\Lip(\Omega,V)$. Furthermore, $\FLip(\Omega,V)$ is closed w.r.t.\ uniform convergence, since \eqref{eq:Fr_norm_dual} is lower-semi-continuous for this topology (as a supremum of continuous maps).

	\section{Unbalanced $L^1$ optimal transport problem for vector valued measures}\label{sec:vectorL1}
	
	In this section we introduce an unbalanced $L^1$ optimal transport problem for vector valued measures.
	In order to ease the presentation of the unbalanced model, we first consider the balanced case, i.e.\ the case of measures $\mu$ and $\nu$ with the same total mass, and present the vector extension of the dual and Beckmann's formulation of the $L^1$ transport, problems \eqref{eq:dual} and \eqref{eq:BP}. These extensions coincide with the ones introduced in \cite{ciosmak2021matrix}.
	We then extend these models to the unbalanced setting. We will consider two different penalizations for the mass imbalance, total variation and $L^2$, respectively in Sections \ref{ssec:vectorL1_unbal} and \ref{ssec:vectorL1_unbalq2}.
	
	For the sake of generality, we stick to considering $\mu,\nu\in\Mc(\Omega;\V)$ instead of $\mu,\nu\in\Mc(\Omega;\Co)$. As in the scalar case in fact, the vectorial extension of the $L^1$ problem does not require the measures to be positive. Of course, considering positive measures will be key for the application to inverse problems, see Section \ref{sec:numerics} and Example \ref{ex:deltas_KRvect}.
	
	\subsection{Balanced case}\label{ssec:vectorL1_bal}
	
	We consider two measures $\mu,\nu\in\Mc(\Omega;\V)$ with the same total mass, $\intO \mu = \intO \nu$, i.e.
	\begin{equation}\label{eq:equalmass}
	\intO \mu_k = \intO \nu_k\,, \quad \forall\, 1 \leq k \leq n\,, \quad \text{for $\mu=\sum_{k=1}^n \mu_k v_k$ and $\nu=\sum_{k=1}^n \nu_k v_k$}\,.
	\end{equation}
	We can define straightforwardly in this setting the dual problem \eqref{eq:dual} as:
	\begin{equation}\label{eq:dualvect}
		\sup_{\phi\in\C(\Om;\V)}\left\{ \int_\Om \langle\phi,\nu-\mu\rangle_\V\,, \; |\nabla \phi|_{\Vd,\infty}\le 1 \right\}.
	\end{equation}
	
	\begin{proposition}\label{prop:dualvect_existence}
		For all $\mu,\nu\in\Mc(\Omega;\V)$, with the same total mass, problem \eqref{eq:dualvect} admits a solution $\phi \in \FLip(\Omega;V)$.
	\end{proposition}
	\begin{proof}
		The proof is similar to the scalar case (see e.g. \cite[Proposition 1.11]{santambrogio2015optimal}). Thanks to the regularity constraint, a maximizing sequence $\phi_n$ is equi-continuous. Furthermore, since the problem is independent of the addition of a constant, we can fix an arbitrary one for each potential $\phi_n$ and derive a uniform bound using the equi-continuity and the boundedness of the domain $\Omega$. By Ascoli-Arzelà one can extract a subsequence converging uniformly on $\Omega$. 
		For the uniform topology, the functional is continuous, and the space $\FLip$ closed as observed below \eqref{eq:Fr_norm_dual}, hence the limit solves \eqref{eq:dualvect}.
	\end{proof}

	The Beckmann formulation \eqref{eq:BP} can also be extended in a straighforward way. For $\sigma\in\Mc(\Omega;\Vd)$ and $\delta\in\Mc(\Omega;\V)$, the continuity equation $\div(\sigma)=\delta$ with zero flux boundary conditions is defined in the following weak sense
	\[
	-\intO \langle \nabla \phi , \sigma \rangle_{\Vd} = \intO \langle \phi,\delta \rangle_{\V}\,, \quad \forall \phi\in\C^1(\Om;\V) \,.
	\]
	As a consequence of this definition, $(\sigma,\mu)$ can be a weak solution only if $\delta$ has zero total mass, i.e.\ $\intO \delta = 0$ according to \eqref{eq:equalmass}. Then we can write the problem:
	\begin{equation}\label{eq:BPvect}
		\inf_{\sigma\in\Mc(\Omega;\V^d)} \left\{ \int_\Om |\sigma|_{\Vd} \,, \; \div (\sigma)=\mu-\nu \right\}.
	\end{equation}
	As in the scalar case, \eqref{eq:dualvect} and \eqref{eq:BPvect} are in duality and the problems are equivalent, as it can be noticed via formal computations. Enforcing the divergence constraint in \eqref{eq:BPvect},
	\[
	\inf_{\sigma} \sup_{\phi} \intO |\sigma|_\Vd + \intO \langle\phi,\div(\sigma)-\mu+\nu\rangle_\V \,,
	\]
	exchanging inf and sup and integrating by parts, we obtain the following dual problem
	\[
	\sup_{\phi} \inf_{\sigma} \intO |\sigma|_\Vd - \intO \langle\nabla \phi,\sigma\rangle_\Vd + \intO \langle\phi,\nu-\mu\rangle_\V \,.
	\]
	The minimization in $\sigma$ provides then
	\[
	\inf_{\sigma} \intO |\sigma|_\Vd - \intO \langle\nabla \phi, \sigma\rangle_\Vd =
	\begin{cases}
		0 &\text{if $|\nabla \phi|_{\Vd,\infty}\le1$} \,, \\
		-\infty &\text{else} \,,
	\end{cases}
	\]
	so that we obtain problem \eqref{eq:dualvect}.
	
	In order to rigorously prove the equivalence of problem \eqref{eq:dualvect} with \eqref{eq:BPvect} we need to justify the swapping of inf and sup.
	We can establish this duality result adapting an approach proposed in \cite[Section 1.6.3]{santambrogio2015optimal} for proving the equivalence between the primal and dual Kantorovich problems (in the scalar setting). As a byproduct of this result, we also derive existence of a solution for the Beckmann formulation \eqref{eq:BPvect}.
	
	\begin{theorem}\label{thm:duality_BP-dual}
		For all $\mu,\nu\in\Mc(\Omega;\V)$, with the same total mass, problems \eqref{eq:dualvect} and \eqref{eq:BPvect} attain the same objective value. Furthermore, the latter admits a solution $\sigma\in\Mc(\Omega;\V^d)$.
	\end{theorem}
	\begin{proof}
		Let us introduce the functional $\Fc:\C(\Om;\Vd)\rightarrow \R$,
		\[
		\Fc(f)=- \sup_{\phi\in\C(\Om;\V)} \left\{ \intO \langle\phi,\nu-\mu\rangle_\V, \; |\nabla \phi-f|_{\Vd,\infty}\le1 \right\}.
		\]
		Note that $\nabla\phi - f$ is formally defined as a difference of distributions, and that it admits an $L^\infty$ representative if and only if the distributional derivative $\nabla \phi$ does, since $f$ is continuous.
		Following the same arguments of the proof of Proposition \ref{prop:dualvect_existence}, for every $f$ there exists a solution $\phi$ to the maximization problem (the equi-continuity is provided in this case by observing that $|\nabla \phi|_{\Vd,\infty} \leq 1+|f|_{\Vd,\infty}$) and for $f=0$ we obtain the opposite of problem \eqref{eq:dualvect}. We show that $\Fc$ is convex and lower semi-continuous. 
		
		Regarding convexity, take $f_1,f_2$, the corresponding optimal potentials $\phi_1,\phi_2$, and an interpolation parameter $t \in [0,1]$. The potential $t\phi_1+(1-t)\phi_2$ is admissible for the maximization problem with $t f_1+(1-t)f_2$ since
		\[
		|t\nabla \phi_1+(1-t)\nabla \phi_2-tf_1-(1-t)f_2|_\Vd \le t|\nabla \phi_1-f_1|+(1-t)|\nabla \phi_2-f_2|_\Vd\le1
		\]
		pointwise on $\Omega$. Therefore we obtain:
		\[
		\Fc(tf_1+(1-t)f_2) \le -\intO\langle t\phi_1+(1-t)\phi_2,\nu-\mu\rangle_\V=t\Fc(f_1)+(1-t)\Fc(f_2)\,.
		\]
		Regarding lower semi-continuity, consider a sequence $f_n$ uniformly converging to $f$, and which is therefore uniformly bounded. We extract a subsequence (without relabeling it) such that $\displaystyle \lim_{n\rightarrow \infty} \Fc(f_n)=\liminf_{n\rightarrow \infty}\Fc(f_n)$, still converging to $f$.
		If we take the corresponding sequence of optimal potentials $\phi_n$, then
		\[
		|\nabla \phi_n|_\Vd\le |\nabla \phi_n -f_n|_\Vd+|f_n|_\Vd\le C\,,
		\]
		pointwise on $\Omega$, hence the sequence $(\phi_n)_{n \in \bN}$ is uniformly equi-continuous. Up to adding a normalizing constant we may assume that $\int_\Omega \phi_n = 0$ for all $n\in \bN$, hence the sequence of potentials is uniformly equi-bounded, and therefore converges uniformly (up to extracting a subsequence) towards a potential $\phi$, by the Arzelà-Ascoli theorem. 
		This latter is an admissible potential for $f$.
		Indeed,	for any test function $\psi \in \C^\infty_c(\interior(\Omega),V^d)$ such that $|\psi|_{V^d,1} \leq 1$ we have 
		\begin{equation*}
			\int_\Om \<\phi_n,\div(\psi)\>_V + \<f_n,\psi\>_{V^d} \leq 1 \,,
		\end{equation*}
		which passes to the limit by uniform convergence. 
		This implies $|\nabla \phi-f|_{V^d,\infty} \leq 1$.
		In the end we get:
		\[
		\Fc(f)\le -\intO\langle \phi,\nu-\mu\rangle_\V=\liminf_n \Fc(f_n)\,.
		\]
		
		Let us compute the Legendre transform of $\Fc$:
		\[
		\begin{aligned}
			\Fc^*(\sigma)&=\sup_{f\in\C(\Om;\Vd)} \intO \langle\sigma,f\rangle_\Vd-\Fc(f) \\
			&=\sup_{\substack{f\in\C(\Om;\Vd)\\ \phi\in\C(\Om;\V)}} \left\{ \intO \langle\sigma,f\rangle_\Vd+\intO\langle \phi,\nu-\mu\rangle_\V \,, \; |\nabla \phi-f|_{\Vd,\infty}\le1 \right\},
			\\
			&=\sup_{\substack{f\in\C^\infty(\Om;\Vd)\\ \phi\in\C^\infty(\Om;\V)}} \left\{ \intO \langle\sigma,f\rangle_\Vd+\intO\langle \phi,\nu-\mu\rangle_\V \,, \; |\nabla \phi-f|_{\Vd,\infty}\le1 \right\}.
		\end{aligned}
		\]
		The last equality is justified by the following density argument. 
		Let $f\in\C(\Om;\Vd)$, $\phi\in\C(\Om;\V)$. Choose a non-negative mollifier $\rho \in C^\infty(\R^d)$ supported in the unit ball and with unit integral, and denote $\rho_\delta := \delta^{-d}\rho(\cdot/\delta)$. The convolutions $f_\delta := f \ast \rho_\delta$ and $\phi_\delta := \phi\ast \rho_\delta$, obey $|\nabla \phi_\delta - f_\delta|_{\Vd,\infty} = |\rho_\delta\ast (\nabla \phi - f)|_{\Vd,\infty} \leq 1$ and are defined on $\Omega_\delta := \{x \in \Omega,\, d(x,\partial\Omega) \geq \delta\}$. 
		We assume, up to translating the domain, that $\overline B(0,r) \subset \Omega$ for some $r>0$, and observe that $(1-\delta/r) \Omega \subset \Omega_\delta$ by convexity for all $0 < \delta < r$.
		Denoting $\tilde f_\delta(x) := f_\delta( (1-\delta/r) x)$ and $\tilde\phi_\delta(x) := \phi_\delta( (1-\delta/r) x) /(1-\delta/r)$ we have $\tilde f_\delta \in C^\infty(\Omega; \V^d)$, $\tilde \phi_\delta \in C^\infty(\Omega; V)$, $|\nabla \tilde \phi_\delta-\tilde f_\delta|_{\Vd,\infty} \leq 1$, and $\tilde f_\delta \to f$ and $\tilde \phi_\delta \to \phi$ uniformly on $\Omega$ as $\delta \to 0$, hence the equality. 
		
		Changing variables to $\hat{f}:=f-\nabla \phi \in C^\infty(\Omega;\Vd)$ we obtain, by density of smooth functions in the set of continuous functions
		\[
		\begin{aligned}
			\Fc^*(\sigma)
			&= \sup_{\substack{\hat f \in C^\infty(\Omega; \Vd)\\\phi\in\C^\infty(\Om;\V)}}
			\left\{ \intO \<\sigma, \hat f\> + \intO \langle \sigma,\nabla \phi\rangle_\Vd+\intO\langle\phi,\nu-\mu\rangle_\V,\, | \hat f |_{\Vd,\infty} \leq 1 \right\}
			\\
			&= \sup_{\phi\in\C^\infty(\Om;\V)} \left\{ \intO |\sigma|_\Vd + \intO \langle \sigma,\nabla \phi\rangle_\Vd+\intO\langle\phi,\nu-\mu\rangle_\V \right\} \\
			&=
			\begin{cases}
				\intO |\sigma|_\Vd &\text{if } \;\div(\sigma)=\mu-\nu \\
				\infty &\text{else,}
			\end{cases}
		\end{aligned}
		\]
		which corresponds to the objective functional of problem \eqref{eq:BPvect}.
		
		As $\Fc$ is convex and lower semi-continuous, it holds
		\[
		-\Fc(0)=-\Fc^{**}(0)=-\sup_{\sigma} \left\{\intO \langle\sigma,0\rangle_\Vd-\Fc^*(\sigma)\right\}=\inf_{\sigma} \Fc^*(\sigma) \,,
		\]
		which provides the equivalence between problems \eqref{eq:BPvect} and \eqref{eq:dualvect}.
		The existence of a solution $\sigma$ to the dual problem follows from the fact that $\Fc$ being convex and lower semi-continuous,
		\[
		\sigma\in\partial \Fc(0) \iff 0\in\partial \Fc^*(\sigma) \iff \Fc(0)+\Fc^*(\sigma)=0 \,.
		\]
		Since $0$ is in the interior of the domain of definition of $\Fc$, the subdifferential in this point is not empty.
	\end{proof}
	
	We introduce the functional $W_1^V:\Mc(\Omega;\V)\times\Mc(\Omega;\V)\rightarrow \R_+$,
	\begin{equation}\label{eq:W1V}
		W_1^V(\mu,\nu) \coloneqq \sup_{\phi\in\C(\Om;\V)}\left\{ \int_\Om \langle\phi,\nu-\mu\rangle_\V\,, \; |\nabla \phi|_{\Vd,\infty}\le 1 \right\}
	\end{equation}
	which associates to any two measures $\mu$ and $\nu$ the transport cost, if they have the same total mass, else the value $+\infty$.
	The following proposition is a simple consequence of the above duality result.
	
	\begin{proposition}
		$W_1^V(\cdot,\cdot)$ is jointly convex and lower semi-continuous. Furthermore, it provides a distance on any subset of $\Mc(\Omega;\V)$ of measures with the same total mass.
	\end{proposition}
	
	\subsection{Unbalanced case}\label{ssec:vectorL1_unbal}
	
	Requiring the two vector valued measures $\mu$ and $\nu$ to have the same mass may be in some cases too restrictive, as it requires the strong geometric constraint \eqref{eq:equalmass}.
	We introduce for this reason an unbalanced version of the problem.
	For $\mu,\nu\in\Mc(\Omega;\V)$ we define:
	\begin{equation}\label{eq:OT1vect_unbal}
		\Tc_{1,1}^\lambda(\mu,\nu) \coloneqq \inf_{\delta\in\Mc(\Omega;\V)} W_1^V(\mu+\delta,\nu)+\lambda \intO |\delta|_\V\,
	\end{equation}
	for a constant $\lambda>0$ (fixed in this section).
	The penalization of the total variation of $\delta$ determines the deformation price of the measures. Problem \eqref{eq:OT1vect_unbal} aims at finding a balance between transporting mass from $\mu$ to $\nu$ or deforming it and for $\lambda\rightarrow+\infty$ one recovers (formally) a pure transport if this exists.
	Problem \eqref{eq:OT1vect_unbal} is an extension to vector valued measures of the Piccoli-Rossi unbalanced $L^1$ distance \cite{piccoli2014generalized}.
	From the equivalent formulations \eqref{eq:dualvect}-\eqref{eq:BPvect} one deduces that $W_1^V(\mu,\nu)$ only depends on the difference of the two measures. The choice of $W_1^V(\mu+\delta,\nu)$ in the definition of problem \eqref{eq:OT1vect_unbal} is therefore not restrictive.

	\begin{theorem}
	\label{th:exists_unbalanced}
		For all $\mu$,$\nu\in\Mc(\Omega;\V)$, there exists a solution $\delta\in\Mc(\Omega;\V)$ to problem \eqref{eq:OT1vect_unbal}.
	\end{theorem}
	\begin{proof}
		The functional $W_1^V(\mu+\delta,\nu)$ is well-posed for any $\delta\in\Mc(\Omega;\V)$ such that $\mu+\delta$ has the same total mass as $\nu$, according to \eqref{eq:equalmass}. It is also lower semi-continuous for the weak-* topology, as a supremum of lower semi-continuous functionals \eqref{eq:W1V}, and non-negative.
		Since the total variation norm is lower semi-continuous as well, the existence of a solution follows by the direct method of calculus of variations: consider a minimizing sequence $(\delta_n)_{n \in \bN}$, note that $\int_\Omega |\delta_n|_V$ is uniformly bounded since $\lambda>0$ and $W_1^V(\mu+\delta_n,\nu) \geq 0$, hence there is a weak-* converging subsequence by compactness of the closed balls of $\Mc(\Omega,V)$, whose limit is a minimizer to \eqref{eq:OT1vect_unbal} by lower semi-continuity.
	\end{proof}

	In the rest of this section we show that \eqref{eq:OT1vect_unbal} admits analogous unbalanced versions of \eqref{eq:dualvect} and \eqref{eq:BPvect}. The latter is evident, it simply requires to use the Beckmann formulation of $W_1^V$ to obtain:
	\begin{equation}\label{eq:BPvect_unbal}
		\Tc^\lambda_{1,1}(\mu,\nu) = \min_{\substack{\sigma\in\Mc(\Omega;\V^d)\\\delta\in\Mc(\Omega;\V)}} \left\{ \int_\Om |\sigma|_{\Vd} + \lambda \int_\Om |\delta|_{\V}, \; \div (\sigma)=\mu-\nu+\delta \right\}.
	\end{equation}
	Due to \eqref{eq:BPvect_unbal}, it can be immediately recognized that $\Tc^\lambda_{1,1}(\cdot,\cdot)$ is again a distance on $\Mc(\Omega;\V)$. Using instead \eqref{eq:dualvect} to rewrite the $W_1^V$ distance, we obtain the saddle point problem
	\[
	\inf_{\delta\in\Mc(\Omega;\V)} \sup_{\phi\in\C(\Om;\V)}\left\{ \int_\Om \langle\phi,\nu-\mu-\delta\rangle_\V +\lambda \intO |\delta|_\V\,, \; |\nabla \phi|_{\Vd,\infty}\le 1 \right\}.
	\]
	Exchanging inf and sup and minimizing in $\delta$ provides
	\begin{equation}\label{eq:KRvect}
		\sup_{\phi\in\C(\Om;\V)}\left\{ \int_\Om \langle\phi,\nu-\mu\rangle_\V\,, \; |\nabla \phi|_{\Vd,\infty}\le 1, \, |\phi|_{\V,\infty} \le\lambda \right\},
	\end{equation}
	by the same computation used in the previous section. The dual problem \eqref{eq:KRvect} is a vector extension of the Kantorovich-Rubenstein norm \eqref{eq:KRscal}. The following theorem shows its equivalence with problem \eqref{eq:OT1vect_unbal}.
	
	\begin{theorem}\label{thm:duality_unbal}
		For all $\mu,\nu\in\Mc(\Omega;V)$, there exists a solution $\phi \in \FLip(\Omega;V)$ to problem \eqref{eq:KRvect} and it holds:
		\[
		\Tc^\lambda_{1,1}(\mu,\nu)=\max_{\phi\in\C(\Om;\V)}\left\{ \int_\Om \langle\phi,\nu-\mu\rangle_\V\,, \; |\nabla \phi|_{\Vd,\infty}\le 1, \, |\phi|_{\V,\infty} \le\lambda \right\}.
		\]
	\end{theorem}
	\begin{proof}
		The proof of existence of a solution is similar to that of problem \eqref{eq:dualvect} in \cref{prop:dualvect_existence}, the explicit uniform bound on the potential providing in this case directly the equi-boundedness.
		
		For $\phi \in \C(\Omega,V)$, we introduce the functionals 
		\[
		\Fc(\phi)=\int_\Om \langle\phi,\mu-\nu\rangle_\V + i_{|\nabla \phi|_{\Vd,\infty}\le 1}\,, \quad \Gc(\phi)=i_{|\phi|_{\V,\infty}\le\lambda}\,,
		\]
		where we denote by 
		\begin{equation}\label{eq:indicator}
		i_A(\phi) \coloneqq \begin{cases}
			0 & \text{if } \phi \in A,\\
			\infty & \text{otherwise,}
		\end{cases}
		\end{equation}
		the indicator function of a set $A\subset\C(\Omega,V)$. If $A$ is convex and closed (for the strong topology), as it is the case in the two functionals above, then $i_A$ is convex and (weakly-*) lower semi-continuous. Therefore, both $\Fc$ and $\Gc$ are convex and (weakly-*) lower semi-continuous. Furthermore, for $\phi=0$ both $\Fc$ and $\Gc$ are bounded and $\Gc$ is also continuous in this point. Then, from the Fenchel-Rockafellar duality theorem \cite[Theorem 1.12]{brezis2010functional},
		\[
		\max_{\phi\in \C(\Om;V)}-\Fc(\phi)-\Gc(\phi)=\min_{\delta\in\Mc(\Om;\V)} \Fc^*(-\delta)+\Gc^*(\delta)\,,
		\]
		where
		\[
		\Fc^*(-\delta) =\sup_{\phi\in\C(\Om;\V)} \int_{\Om}\langle\phi,\nu-\mu-\delta\rangle_\V-i_{|\nabla \phi|_{\Vd,\infty}\le 1}=W_1^V(\mu+\delta,\nu)\,,
		\]
		by definition of $W_1^V$, and
		\[
		\Gc^*(\delta)=\sup_{\phi\in\C(\Om;\V)} \int_{\Om}\langle\phi,\delta\rangle_\V-i_{|\phi|_{\V,\infty}\le\lambda}=\lambda \intO |\delta|_\V\,,
		\]
		showing the equivalence between problems \eqref{eq:OT1vect_unbal} and \eqref{eq:KRvect}.
	\end{proof}
	
	The following example clarifies the role played by the penalization term and is helpful in understanding how the parameter $\lambda$ can be tuned in practice.
	
	\begin{example}[Exact solution for two delta measures]\label{ex:deltas_KRvect}
		Let us consider the two measures $\mu=M_1 \delta_{x_1},\nu=M_2 \delta_{x_2}$ for $M_1,M_2\in\V$ and two distinct points $x_1,x_2\in \Om$. The Kantorovich-Rubinstein norm \eqref{eq:KRvect} can be written in this case as:
		\begin{equation}\label{eq:Krdeltas}
			|x_1-x_2| \max_{A,B\in \V} \left\{ \langle A, M_1 	\rangle_\V-\langle B, M_2 \rangle_\V \,, |A|_\V\le \gamma, |B|_\V\le \gamma, |A-B|_\V\le 1 \right\},
		\end{equation}
		where $\gamma=\frac{\lambda}{|x_1-x_2|}$.
		Indeed, let $\phi$ be such that $|\nabla \phi|_{V^d} \leq 1$ pointwise on the convex domain $\Omega$, which implies $|\nabla \phi|_{L(\R^d;V)} \leq 1$ since the Frobenius norm is larger than the operator norm. Thus $\phi$ is $1$-Lipschitz, and therefore $A \coloneqq \phi(x_1)$ and $B \coloneqq \phi(x_2)$ satisfy $|A-B|_{V} \leq |x_1-x_2|$.
		Conversely, assume that $|A-B|_{V} \leq |x_1-x_2|$, and define for any $z\in \Omega$
		\begin{align*}
			\phi(z) &= 
			\begin{cases}
				(1-t) A+t B & \text{if } t \in [0,1]\,,\\
				A & \text{if } t \leq 0\,,\\
				B & \text{if } t \geq 1\,,
			\end{cases}&
			\text{where }
			t \coloneqq (z-x_1)\cdot n\,,
		\end{align*}
		and $n \coloneqq (x_2-x_1) / |x_2 - x_1|^2$. Then $\phi$ is piecewise linear hence Lipschitz, and when $0<t<1$ the gradient $\nabla \phi(z) = (B-A) \otimes n$ is a rank one matrix which satisfies $|\nabla \phi(z)|_{V^d} = |\nabla \phi(z)|_{L(\R^d;V)} = |B-A|_\V |n| \leq 1$ as required, since the Frobenius and operator norms coincide on rank one matrices.	
	
		We can rewrite \eqref{eq:Krdeltas} by means of an infsup formulation as
		\[
		\max_{A,B} \min_{C_1,C_2,C_3}\langle A, M_1 \rangle_\V-\langle B, M_2 \rangle_\V + |C_1|_\V-\frac{1}{\gamma} \langle A,C_1 \rangle_\V + |C_2|_\V- \frac{1}{\gamma}\langle B,C_2 \rangle_\V + |C_3|_\V-\langle A-B, C_3 \rangle_\V.
		\]
		Swapping min and max by strong duality\footnote{
		The Fenchel-Rockafellar duality theorem can be applied again for this purpose, as in Theorem \ref{thm:duality_unbal}, for $f(A,B,D)\coloneqq\gamma \<A,M_1\>_\V-\gamma \<B,M_2\>_\V+i_{D=\gamma(A-B)}$ and $g(A,B,D)\coloneqq i_{|A|\leq 1} + i_{|B|\leq 1} + i_{|D|\leq 1}$ (up to a rescaling of $A$ and $B$ by $\gamma$), which are both finite for $(0,0,0)$ and the latter is furthermore continuous at this point.
		} and optimizing in $A$ and $B$ we obtain
		\[
		M_1-\frac{1}{\gamma} C_1-C_3=0\,, \quad -M_2-\frac{1}{\gamma} C_2+C_3=0\,,
		\]
		and the problem reduces to
		\[
		\min_{C\in\V} \gamma|C-M_1|_\V + \gamma |C-M_2|_\V+ |C|_\V\,.
		\]
		This is the generalized Fermat-Torricelli problem: finding the point that minimizes the weighted sum of the distances from the vertices of a triangle (which lies on the plane passing by the three points $0,M_1,M_2$). The analytical solution is available for this problem for example in \cite{uteshev2014analytical}.
		The total cost is then:
		\begin{equation}\label{eq:costdeltas}
		\left\{
		\begin{aligned}
			&\lambda (|M_1|_\V+|M_2|_\V) &&\text{if $\frac{\langle M_1,M_2\rangle_\V}{|M_1|_\V|M_2|_\V}\le \frac{1}{2\gamma^2}-1$} \\
			&|x_1-x_2||M_1|_\V+\lambda |M_1-M_2|_\V  &&\text{if $\frac{\langle M_1,M_2-M_1\rangle_\V}{|M_1|_\V|M_2-M_1|_\V}\ge \frac{1}{2\gamma^2}$} \\
			&|x_1-x_2| |M_2|_\V + \lambda |M_1-M_2|_\V  &&\text{if $\frac{\langle M_2,M_1-M_2\rangle_\V}{|M_2|_\V|M_1-M_2|_\V}\ge \frac{1}{2\gamma^2}$} \\
			&\begin{split}
				|x_1-x_2| \Big( (\gamma^2-\frac12)|M_1-M_2|_\V^2+\frac12(|M_1|_\V^2+|M_2|^2_\V)+\\
				(4\gamma^2-1)^\frac12(|M_1|_\V^2|M_2|_\V^2-\langle M_1,M_2\rangle_\V^2)^\frac12\Big)^\frac12
			\end{split} &&\text{else,}
		\end{aligned}
		\right.
		\end{equation}
		with the three degenerate cases
		\[
		\left\{
		\begin{aligned}
			&\lambda |M_1|_\V  &&\text{if $M_2=0 $} \\
			&\lambda |M_2|_\V  &&\text{if $M_1=0 $} \\
			& \min(2\lambda,|x_1-x_2|) |M|_\V  &&\text{if $M_1=M_2=M$}.
		\end{aligned}
		\right.
		\]
		It would be furthermore possible to explicitly write the solution $C$ and recover the optimal potential $A,B$.
		
		In the scalar case $n=1$, for $\mu=m_1 \delta_{x_1},\nu=m_2 \delta_{x_2}$, with $m_1,m_2\in\R$, the cost reduces to the already known (e.g. \cite[Section 5.1]{chizat2018unbalanced}):
		\begin{equation}\label{eq:costdeltas_n1}
		\left\{
		\begin{aligned}
			&\lambda (|m_1|+|m_2|) &&\text{if $\gamma\le\frac12$ or $\sign(m_1)\neq\sign(m_2)$} \\
			&\lambda (|m_2|-|m_1|)+ |x_1-x_2||m_1|  &&\text{if $m_2\ge m_1\ge0$ or $0\ge m_1 \ge m_2$} \\
			&\lambda (|m_1|-|m_2|)+ |x_1-x_2||m_2|  &&\text{if $m_1\ge m_2\ge 0$ or $0\ge m_2 \ge m_1$. }
		\end{aligned}
		\right.
		\end{equation}
		
		The first condition in \eqref{eq:costdeltas} discriminates between transporting or not the mass, depending on the parameter $\lambda$. If it is satisfied, then the mass is not transported (and the cost corresponds to its destruction at $x_1$ and its reconstruction at $x_2$). If the product $\langle M_1,M_2\rangle_\V$ is non-negative, then there is never transport for $\lambda\le\frac{|x_1-x_2|}{2}$, whereas transport occurs for $\lambda\ge\frac{|x_1-x_2|}{\sqrt{2}}$, with an intermediate regime in between depending on the angle.
		If on the contrary $\langle M_1,M_2\rangle_\V$ is negative, $\lambda$ needs to increase further to ensure transport, depending on the angle between $M_1$ and $M_2$, tending to $+\infty$ in the limiting case they have opposite directions, in which case transport never occurs. 
		This latter phenomenon is the equivalent of the scalar case with a positive dirac and a negative dirac \eqref{eq:costdeltas_n1}.
		This motivates our choice of considering measures valued in a cone $\Co\subseteq\Co^*$, in order to have a meaningful transport problem. Indeed, in this case $\langle M_1,M_2\rangle_\V\ge0$ for any $M_1,M_2\in \Co$.
		
	\end{example}

	The previous example shows that transport may not occur if the data are not signed. The next example shows a more interesting case where this can lead to a somewhat pathological behavior.

	\begin{example}[Insensitivity of $\KR$ to horizontal shifts for signed measures]\label{ex:KRshift}
		Consider for simplicity $\V=\R$, in which case \eqref{eq:KRscal} and \eqref{eq:KRvect} coincide, and two measures $\mu,\nu\in\Mc(\Omega;\V)$. Notice that \eqref{eq:KRscal} only depends on the difference of $\mu$ and $\nu$ and can be therefore formulated between the two positive measures $(\mu-\nu)_+$ and $(\mu-\nu)_-$.
		In terms of optimal transport, this means that positive mass can be transported on negative one and viceversa.
		Consider the two measures $\mu=\sin_{2\pi}(x)\coloneqq \sin(\max\{0,\min\{2\pi,x\}\})$, and $\nu=\mu(\cdot-t)$ for $t\in\R_+, t\ge 2\pi$ (see Figure \ref{fig:KRexample}). Since $\intO \mu=\intO\nu=0$, from the previous example we can deduce that $\KR(\mu,\nu)=\KR((\mu-\nu)_+,(\mu-\nu)_-)=W_1((\mu-\nu)_+,(\mu-\nu)_-)=2\pi$, for a sufficiently big value of the parameter $\lambda$. The transport cost is independent of the parameter $t$, i.e.\ the cost is independent of the shift between the two measures.
		We highlight that this is true since $\mu$ has zero mean. If this is not the case, one recovers sensitivity to the time shift, since an amount of mass needs to be transported between $\mu_+$ and $\nu_+$ (or between $\mu_-$ and $\nu_-$).
	\end{example}

	\begin{figure}
		\centering
		\begin{tikzpicture}[scale=2]
			\usetikzlibrary {arrows.meta} 
			% Axes
			\draw[->] (-0.5,0) -- (5,0);
			\draw (5,-0.05) node[below] {$x$} ;
			\draw[->] (0,-0.5) -- (0,1.5);
			
			% Grid lines (optional)
			\draw (1,-0.05) -- (1,0.05);
			\draw (1,-0.05) node[below] {$\pi$};
			\draw (2,-0.05) -- (2,0.05);
			\draw (2,-0.05) node[below] {$2\pi$};
			
			% truncated sine function and its translation
			\draw[scale=1, domain=0:2, smooth, variable=\x, black, line width=.2pt] plot ({\x}, {sin(deg(pi*\x))});
			\draw[scale=1, domain=0:2, smooth, variable=\x, black, line width=.2pt] plot ({\x+2.5}, {sin(deg(pi*\x))});
			%\draw[] (1.7,-1) node[right] {\large{$\sin_{2\pi}(x)$}};
			%\draw[] (4.2,-1) node[right] {\large{$\sin_{2\pi}(x-t)$}};
			\draw[] (1.7,-1) node[right] {\large{$\mu$}};
			\draw[] (4.2,-1) node[right] {\large{$\nu=\mu(\cdot-t)$}};
			
			% draw (mu-nu)_+
			\draw[scale=1, domain=0:1, smooth, variable=\x, blue, dashed, line width=1pt] plot ({\x}, {sin(deg(pi*\x))});
			\draw[scale=1, domain=0:1, smooth, variable=\x, blue, dashed, line width=1pt] plot ({\x+3.5}, {sin(deg(pi*\x))});
			
			% draw (mu-nu)_-
			\draw[scale=1, domain=0:1, smooth, variable=\x, red, dashed, line width=1pt] plot ({\x+1}, {sin(deg(pi*\x))});
			\draw[scale=1, domain=0:1, smooth, variable=\x, red, dashed, line width=1pt] plot ({\x+2.5}, {sin(deg(pi*\x))});

			\draw [->,line width=0.5pt,arrows = {-Stealth[scale=1.5]}] (0.5,0.5) to [out=45,in=135] (1.5,0.5);
			\draw [->,line width=0.5pt,arrows = {-Stealth[scale=1.5]}] (4,0.5) to [out=135,in=45] (3,0.5);
			
			% add legend
			%\matrix [draw, above left] at (8,-2) {
			%	\node {{\draw[blue] (0,0) -- (1,0)}}; & \node {$(\mu-\nu)_+$};
			%};
		
			%\draw[draw=black] (4.5,1.3) rectangle ++(1,0.5);
			\draw[blue,dashed,line width=1pt] (4.6,1.4) -- (4.8,1.4);
			\draw[red,dashed,line width=1pt] (4.6,1.15) -- (4.8,1.15);
			\draw (4.8,1.4) node[right] {\small{$(\mu-\nu)_+$}};
			\draw (4.8,1.15) node[right] {\small{$(\mu-\nu)_-$}};
			
		\end{tikzpicture}
		\caption{Representation of the transport problem of Example \ref{ex:KRshift}. The $L^1$ optimal transport between $(\mu-\nu)_+$ and $(\mu-\nu)_-$ is independent of the shift $t\ge2\pi$. The solution is the shift, of length $\pi$, of the leftmost blue bump of mass to the right and of the rightmost blue bump of mass to the left. Notice that, if $\mu$ has not zero mean, this is not possible and some leftover mass needs to travel from the leftmost to the rightmost side (or viceversa).}
		\label{fig:KRexample}
	\end{figure}
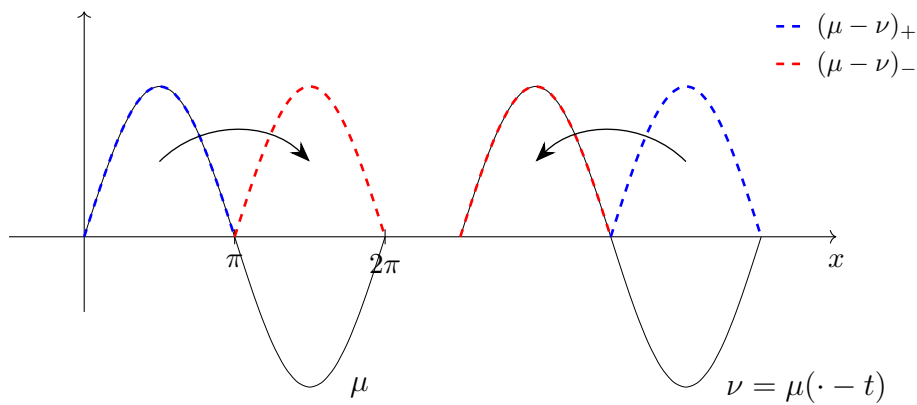

	\subsection{Quadratic penalization}\label{ssec:vectorL1_unbalq2}
	
	As we have seen in the previous section, when using a one-homogeneous total variation norm penalization the parameter $\lambda$ corresponds to half the transport distance up to which mass can be transported. This behavior is independent of the mass considered. In applications it may be more relevant instead to consider a penalization on the transport distance depending also on the mass. This can be achieved by considering a two-homogeneous penalization.
	Another advantage in the specific application of interest for this work, namely seismic imaging, is that the quadratic penalization corresponds to the physical energy of the signal (see Section \ref{sec:numerics}).
	On the contrary, the total variation has no physical interpretation and thus sometimes induces undesirable artifacts, in which some regions of very low physical energy have an excessively strong influence on the transport problem.

	We restrict in this section our attention to measures which are absolutely continuous w.r.t.\ the Lebesgue measure and whose density lies in $L^2(\Omega;\V)$. By a slight abuse of notation, we identify in this case measures with their densities. Since the domain $\Omega$ is compact, $\Mc(\Omega;\V)\cap L^2(\Omega;\V)$ coincides with $L^2(\Omega;\V)$ itself. We define the unbalanced $L^1$ transport cost with quadratic penalization as
	\begin{equation}\label{eq:OT1vect_unbalq2}
	\Tc^\lambda_{1,2}(\mu,\nu) \coloneqq
		\inf_{\delta\in L^2(\Omega;\V)} W_1^V(\mu+\delta,\nu)+\frac{\lambda}{2} \intO |\delta|_\V^2\,,
	\end{equation}
	for any $\mu,\nu\in L^2(\Omega;\V)$.
	The following theorem provides the analogous of the unbalanced formulations \eqref{eq:KRvect} and \eqref{eq:BPvect_unbal} in this case.
	
	\begin{theorem}
		For all $\mu,\nu\in L^2(\Omega;\V)$, problem \eqref{eq:OT1vect_unbalq2} admits a minimizer $\delta\in L^2(\Omega;\V)$ and it holds:
		\begin{align}
			\Tc^\lambda_{1,2}(\mu,\nu) &= \max_{\phi\in\C(\Om;\V)}\left\{ \intO \langle\phi,\nu-\mu\rangle_\V -\frac{1}{2\lambda} \intO |\phi|^2_\V\,, \; |\nabla \phi|_{\Vd,\infty}\le 1 \right\} \label{eq:dualvect_unbalq2} \\
			&=\min_{\substack{\sigma\in\Mc(\Omega;\V^d)\\\delta\in L^2(\Omega;\V)}} \left\{ \int_\Om |\sigma|_{\Vd} + \frac{\lambda}{2} \int_\Om |\delta|^2_{\V}, \; \div (\sigma)=\mu-\nu+\delta \right\} \label{eq:BPvect_unbalq2}.
		\end{align}
		Furthermore, the optimal potential $\phi$ of \eqref{eq:dualvect_unbalq2} is unique.
	\end{theorem}
	\begin{proof}
		Let $\delta_n$ be a minimizing sequence for \eqref{eq:OT1vect_unbalq2}. Since $W_1^V(\mu+\delta,\nu)\ge 0$, we can assume that the sequence is bounded in $L^2(\Omega;\V)$. The $L^2(\Omega;\V)$ space is not closed with respect to the weak-* topology. We can consider however in this case the weak $L^2$ topology to extract a converging subsequence. From definition \eqref{eq:dualvect}, since $\C(\Omega;\V)\subset L^2(\Omega;\V)$, $\Omega$ being compact, the distance $W_1^V$ is clearly lower semi-continuous with respect to this topology. We obtain therefore existence of an optimal $\delta$ for \eqref{eq:OT1vect_unbalq2}. The equivalence with \eqref{eq:BPvect_unbalq2} is trivial.
		
		Let us consider existence for problem \eqref{eq:dualvect_unbalq2}.
		In order to derive the uniform boundedness for a sequence of maximizers, we can use the uniform Lipschitz regularity. The problem is not independent of the addition of a constant to the potential but we can determine a uniform bound for it. Problem \eqref{eq:dualvect_unbalq2} can be rewritten as
		\[
		\sup_{\substack{a\in\Rd,\hat{\phi}\in\C(\Om;\V)\\ \intO \hat{\phi}=0}}\left\{ \intO \langle\hat{\phi},\nu-\mu\rangle_\V+ a \int_{\Om}(\nu-\mu)-\frac{1}{2\lambda} \intO |\hat{\phi}|^2_\V-\frac{a^2}{2\lambda}|\Omega|\,, \; |\nabla \hat{\phi}|_{\Vd,\infty}\le 1 \right\},
		\]
		by making explicit the average $a$ of the potential $\phi$, and introducing the zero-mean difference $\hat\phi \coloneqq \phi - a/|\Omega|$. Consider then a maximizing sequence $(a_n,\hat{\phi}_n)$. The average $a_n$ clearly converges to the optimal value $\lambda \int_{\Om}(\nu-\mu) / |\Omega|$. 
		The sequence of zero mean uniformly Lipschitz potential $\hat{\phi}_n$ is now also uniformly bounded, $|\hat{\phi}_n|\le\diam(\Omega) \coloneqq \max \{|x-y|,\, x,y \in \Omega\}$.
		Indeed, $\hat{\phi}_n$ is Lipschitz because the Frobenius norm is larger than the operator norm. 
		Second, let $v \in V$ with $|v|_V = 1$ be such that $|\hat{\phi}_n|_\infty = |\<\hat{\phi}_n,v\>_V|_\infty$; then $x\mapsto \<\hat{\phi}_n(x),v\>_V$ is $1$-Lipschitz, zero-mean and scalar valued, hence bounded by $\diam(\Omega)$ as announced. Hence, we can apply Ascoli-Arzelà as in the proof of Proposition \eqref{prop:dualvect_existence} to deduce convergence to an optimal $\hat{\phi}$. Finally, the optimal potential in \eqref{eq:dualvect_unbalq2} is $\phi=\lambda \int_{\Om}(\nu-\mu)/|\Omega|+\hat{\phi}$. Uniqueness derives from the strict convexity of the objective functional.
		
		To conclude, the proof of the duality result is analogous to the proof of Theorem \ref{thm:duality_unbal}. We introduce the two convex and lower semi-continuous functionals
		\[
		\Fc(\phi)=\int_\Om \langle\phi,\mu-\nu\rangle_\V + i_{|\nabla \phi|_{\Vd,\infty}\le 1}\,, \quad \Gc(\phi)=\frac{1}{2\lambda}\intO |\phi|_\V^2\,,
		\]
		which satisfies the hypotheses of the Fenchel-Rockafellar duality theorem \cite[Theorem 1.12]{brezis2010functional} (consider again $\phi=0$). It holds again
		\[
		\Fc^*(-\delta)=W_1^V(\mu+\delta,\nu)\,,
		\]
		whereas for the second term
		\[
		\Gc^*(-\delta)=\sup_{\phi\in\C(\Omega;\V)} -\intO \langle\phi,\delta\rangle_\V-\frac{1}{2\lambda}\intO |\phi|_\V^2=
		\begin{cases}
			\frac{\lambda}{2}\intO |\delta|_\V^2 &\text{if $\delta\in L^2(\Omega;\V)$,} \\
			+\infty &\text{else,}
		\end{cases}
		\]
		since $\intO \langle\phi,\delta\rangle_\V$ is not a bounded functional on $L^2(\Omega;\V)$ if $\delta\notin L^2(\Omega;\V)$. The equivalence between \eqref{eq:OT1vect_unbalq2} and \eqref{eq:dualvect_unbalq2} follows.
	\end{proof}
	
	\begin{remark}
		Differently from problems \eqref{eq:OT1vect_unbalq2} and \eqref{eq:BPvect_unbalq2} which require $\mu,\nu \in L^2(\Omega;\V)$, the dual formulation \eqref{eq:dualvect_unbalq2} is well-defined for every $\mu,\nu\in\Mc(\Omega;\V)$.
	\end{remark}

	We can again shed light on the role played by the parameter $\lambda$, in this case with quadratic penalization, by considering an exact solution.
	  
	\begin{example}[Exact solution for two delta measures]\label{ex:deltas_p1q2}
		Consider the case $d=1,n=1$ and two measures of the form $\mu=m_1 \delta_{x_1},\nu=m_2 \delta_{x_2}$ for $m_1,m_2\in \R_+$.
		Following Example \ref{ex:deltas_KRvect}, we want to rewrite problem \eqref{eq:dualvect_unbalq2} in the variables $a=\phi(x_1)$ and $b=-\phi(x_2)$. Clearly, the $L^2$ norm squared of $\phi$ is minimized, under the gradient constraint, by a piecewise linear function which on its support has slope $\pm 1$ almost everywhere, and may only have a local maximum at $x_1$ and a local minimum at $x_2$.	
		Suppose that $m_1>m_2$. We need to distinguish between the cases
		$\lambda\ge \frac{|x_1-x_2|^2}{m_1-m_2}$ and $\lambda\le \frac{|x_1-x_2|^2}{m_1-m_2}$. In the first case the first mass dominates on the second one and we can therefore write the problem as:
		\[
		\sup_{a} \left\{ a m_1 + b m_2  -\frac{1}{3\lambda} a^3 \,, \; a\ge0\,, b=-a+|x_1-x_2| \right\},
		\]
		where $\phi(x) = \max \{0,a-|x-x_1|\}$ is the hat function centered at $x_1$, and thus $\frac{1}{3}a^3=\frac{1}{2}\int\phi^2$. In the second case instead the problem writes:
		\[
		\sup_{a,b} \left\{ a m_1 + b m_2  -\frac{1}{3\lambda} \left(a^3+b^3\right) \,, \; a+b\le |x_1-x_2|, a\ge0,b\ge0 \right\},
		\]
		where now $\phi(x) = \max \{0,a-|x-x_1|\} - \max\{0, b-|x-x_2|\}$ is the difference of two hat functions centered at $x_1$ and $x_2$ and whose supports have disjoint interiors, and thus $\frac{1}{2}\int\phi^2=\frac{1}{3}(a^3+b^3)$.
		Easy computations finally show that the solution has the following form:
		\begin{equation}\label{eq:costdeltas_2}
		\left\{
		\begin{aligned}
			&a=\sqrt{\lambda (m_1-m_2)}, b=-a+|x_1-x_2| &&\text{if $\lambda\ge \frac{|x_1-x_2|^2}{m_1-m_2}$,} \\
			&a=\frac{|x_1-x_2|}{2}+\frac{\lambda(m_1-m_2)}{2|x_1-x_2|}, b=\frac{|x_1-x_2|}{2}-\frac{\lambda(m_1-m_2)}{2|x_1-x_2|} &&\text{if $\lambda\le \frac{|x_1-x_2|^2}{m_1-m_2}$,}\\
			&a=\sqrt{\lambda m_1}, b=\sqrt{\lambda m_2} &&\text{if $\lambda\le\frac{|x_1-x_2|^2}{(\sqrt{m_1}+\sqrt{m_2})^2}$.}
		\end{aligned}
		\right.
		\end{equation}
		
		As in Example \ref{ex:deltas_KRvect}, $\lambda$ represents a threshold for transporting the mass. If the third condition is satisfied in \eqref{eq:costdeltas_2}, mass is not transported. In this case however $\lambda$ represents the distance at which a unit of mass is transported. By adding a quadratic penalization term in \eqref{eq:OT1vect_unbalq2}, particles are transported according to their mass, allowing smaller masses to have a lower influence on the transport problem.
	\end{example}
	
	\begin{remark}[The $H^{-1}$ norm]
		From \eqref{eq:BPvect_unbalq2} arises the natural idea of endowing also $\Mc(\Omega;\Vd)$ with the $L^2$ structure and write
		\begin{equation}\label{eq:H-1_primal}
			\inf_{\substack{\sigma\in L^2(\Omega;\V^d)\\\delta\in L^2(\Omega;\V)}} \left\{ \int_\Om |\sigma|^2_{\Vd} + \frac{\lambda}{2} \int_\Om |\delta|^2_{\V}, \; \div (\sigma)=\mu-\nu+\delta \right\}.
		\end{equation}
		In this case the problem coincides with the $H^{-1}$ norm of $\mu-\nu$. Indeed, the dual problem to \eqref{eq:H-1_primal} is given by
		\begin{equation}\label{eq:H-1}
			\sup_{\phi} \intO \langle \phi, \nu-\mu \rangle_{\V}-\frac{1}{2\lambda}\intO |\phi|_\V^2-\frac12 \intO |\nabla\phi|_\Vd^2
		\end{equation}
		which corresponds to the dual norm
		\[
		\sup_{|\phi|_{H^1}\le1} \intO \langle \phi,\nu-\mu\rangle_{\V} \,, \quad
		|\phi|^2_{H^1} \coloneqq \frac{1}{\lambda}\intO |\phi|_\V^2+ \intO |\nabla\phi|_\Vd^2\,.
		\]
		With this choice the optimal transport structure is lost.
		The solution $\phi$ is provided by the following elliptic equation:
		\begin{equation}\label{eq:H-1_oc}
			\Delta \phi +\frac{1}{\lambda} \phi= \nu-\mu\,.
		\end{equation}
		By writing $\phi,\nu$ and $\mu$ in the common basis of $\V$, \eqref{eq:H-1_oc} is a set of $n$ scalar equations that can be solved separately and \eqref{eq:H-1_primal} is decoupled in the different components of $\mu,\nu$.
		Problem \eqref{eq:H-1_primal} is not interesting for the purpose of this paper but it will be considered for comparison. We denote $\Tc_{2,2}^\lambda(\mu,\nu)$ the functional that associates the two measures $\mu,\nu\in L^2(\Omega;\V)$ with the optimal value in \eqref{eq:H-1}.
	\end{remark}

	\section{Discretization and numerical solution}\label{sec:discrete_model}
	
	In this section, we consider a finite difference discretization of our model and propose the use an SDMM algorithm, a primal-dual proximal splitting optimization approach, for the solution of the discrete problem.
	
	\subsection{Finite difference discretization}
	
	We assume in this section a rectangular domain of the form $\Om=\Pi_{i=1}^d [a_i,b_i]^d$, where $a_i,b_i\in \R$ and $a_i < b_i$ for all $1 \leq i \leq d$, and consider a forward finite difference discretization of our vector valued unbalanced optimal transport problem.
	We restrict to this setting for simplicity, but a straightforward extension to more general polyhedral domains and unstructured partitioning may be obtained by relying on finite volume discretizations.
	
	For simplicity, let $d\coloneqq2$ in the rest of this section (the extension to arbitrary dimension is straightforward). The domain $\Om=[a_x,b_x]\times[a_y,b_y]$ is discretized using a uniform Cartesian grid $\Gc$, with mesh sizes $h_x=\frac{b_x-a_x}{N_x-1}$ and $h_y=\frac{b_y-a_y}{N_y-1}$, where $N_x$ and $N_y$ are the number of grid points, respectively along the first and second direction. A point on the grid is denoted by $(x_i,y_j)=(a_x+(i-1)h_x,\, a_y+(j-1)h_y)\in\Gc$, for $1\le i \le N_x$, $1\le j \le N_y$.
	The set of discrete functions defined on $\Gc$, with values respectively in $\V$ and $\Vd$, are denoted by $\Xc\coloneqq\mathfrak{F}(\Gc;\V)$ and $\Xc^d\coloneqq\mathfrak{F}(\Gc;\Vd)$ (we will use the compact notation $\Xc$ and $\Xc^d$ in subscripts for simplicity). We use the convention $\phi_{i,j}\in\V$ to denote the $(i,j)$-th component for every discrete function $\phi\in \mathfrak{F}(\Gc;\V)$ and similarly for functions $\psi\in \mathfrak{F}(\Gc;\Vd)$. These spaces are endowed with the scalar products
	\[
	\begin{aligned}
		&\langle u, v \rangle_{\Xc} = h_x h_y \sum_{i=1}^{N_x}\sum_{j=1}^{N_y} \langle u_{i,j},v_{i,j}\rangle_{\V}, &&\forall u,v \in \mathfrak{F}(\Gc;\V)\,, \\ 
		&\langle w, z \rangle_{\Xc^d} = h_x h_y \sum_{i=1}^{N_x}\sum_{j=1}^{N_y} \langle w_{i,j},z_{i,j}\rangle_{\Vd}, &&\forall w,z \in \mathfrak{F}(\Gc;\Vd)\,, \\ 
	\end{aligned}
	\]
	and the respective norms $|\cdot|_\Xc, |\cdot|_{\Xc^d}$. 
	
	We also introduce the discrete $l^p$ and $l^\infty$ norms on the discrete spaces $\mathfrak{F}(\Gc;\V)$ and $\mathfrak{F}(\Gc;\Vd)$ as
	\[
	\begin{aligned}
		&|u|^p_{\Xc,p} = h_x h_y   \sum_{i=1}^{N_x}\sum_{j=1}^{N_y} |u_{i,j}|^p_{\V} \,,
		\quad
		&&|w|^p_{\Xc^d,p} = h_x h_y \sum_{i=1}^{N_x}\sum_{j=1}^{N_y} |w_{i,j}|^p_{\Vd}, \\[0.4em]
		&|u|_{\Xc,\infty} = \max_{i,j} |u_{i,j}|_\V\,, \quad &&|w|_{\Xc^d,\infty} = \max_{i,j} |w_{i,j}|_\Vd\,,
	\end{aligned}
	\]
	for $u\in \mathfrak{F}(\Gc;\V), w\in \mathfrak{F}(\Gc;\Vd)$.
	
	The discrete gradient operator $\nabla_{\Xc}: \mathfrak{F}(\Gc;\V) \longrightarrow \mathfrak{F}(\Gc;\Vd)$ (we recall that $d=2$) is defined as
	\[
	\begin{aligned}
		(\nabla_{\Xc}(u))_{i,j} &\coloneqq \left( \frac{u_{i+1,j}-u_{i,j}}{h_x} \,, \frac{u_{i,j+1}-u_{i,j}}{h_y} \right) \quad &&1\le i\le N_x\,, 1\le j\le N_y\,, \\
	\end{aligned}
	\]
	with the convention $u_{N_x+1,j}=u_{N_x,j}$ and $u_{i,N_y+1}=u_{i,N_y}$.
	The discrete divergence operator $\div_{\Xc}:\mathfrak{F}(\Gc;\Vd) \longrightarrow \mathfrak{F}(\Gc;\V)$ is instead defined by duality as
	\[
	\langle \div_{\Xc}(w),u\rangle_{\Xc}=-\langle w,\nabla_{\Xc}(u)\rangle_{\Xc^d} \,.
	\]
	It automatically inherits the discrete Neumann boundary conditions.
	
	The discrete setting allows us to uniformize the notation for the discrete transport problem.
	Given two discrete measures $\mu,\nu\in \mathfrak{F}(\Gc;\V)$, and two natural numbers $p,q\ge 1$, the discrete (generalized) Beckmann's minimal flow problem is
	\begin{equation}\label{eq:BPvect_pq_h}
		(\mathcal{T}_{\Xc})_{p,q}^\lambda (\mu,\nu) \coloneqq\inf_{\substack{\sigma\in \mathfrak{F}(\Gc;\Vd)\\\delta\in \mathfrak{F}(\Gc;\V)}} \Bigg\{ \frac{1}{p} |\sigma|^p_{\Xc^d,p} + \frac{\lambda}{q} |\delta|^q_{\Xc,q} \,, \;\div_{\Xc} (\sigma)=\mu-\nu+\delta \Bigg\}.
	\end{equation}
	The following proposition justifies the existence of a solution and gives the equivalent discrete dual formulation of \eqref{eq:BPvect_pq_h}.
	
	\begin{proposition}\label{prop:duality_discrete}
		For all $\mu,\nu\in \mathfrak{F}(\Gc;\V)$, the following duality formulas hold:
		\begin{gather}
			\label{eq:KRvect_h}
			(\mathcal{T}_{\Xc})_{1,1}^\lambda(\mu,\nu) = \max_{\phi\in \mathfrak{F}(\Gc;\V)}\left\{ \langle\phi,\nu-\mu\rangle_\Xc\,, \; |\nabla_{\Xc} (\phi)|_{\Xc^d,\infty}\le 1, \, |\phi|_{\Xc,\infty} \le\lambda \right\}, \\
			\label{eq:dualvect_unbalq2_h}
			(\mathcal{T}_{\Xc})_{1,2}^\lambda(\mu,\nu) = \max_{\phi\in \mathfrak{F}(\Gc;\V)}\left\{ \langle\phi,\nu-\mu\rangle_\Xc-\frac{1}{2\lambda} |\phi|^2_{\Xc,2}\,, \; |\nabla_{\Xc} (\phi)|_{\Xc^d,\infty}\le 1 \right\}, \\
			\label{eq:H-1_h}
			(\mathcal{T}_{\Xc})_{2,2}^\lambda(\mu,\nu) = \max_{\phi\in \mathfrak{F}(\Gc;\V)}\left\{ \langle\phi,\nu-\mu\rangle_\Xc-\frac{1}{2\lambda} |\phi|^2_{\Xc,2}-\frac12 |\nabla_{\Xc}(\phi)|^2_{\Xc^d,2} \right\}.
		\end{gather}
		Furthermore, the primal problem \eqref{eq:BPvect_pq_h} admits a solution as well.
	\end{proposition}
	\begin{proof}
		The proof of duality can be achieved as a straightforward adaptation of the proof strategies of Section \ref{sec:vectorL1} to the discrete setting, providing at the same time existence for \eqref{eq:BPvect_pq_h}. Existence for problems \eqref{eq:KRvect_h}-\eqref{eq:dualvect_unbalq2_h}-\eqref{eq:H-1_h} follows from discrete compactness arguments.
	\end{proof}
	
	\begin{remark}
	The discretization of the $\KR$ norm proposed in the original works \cite{metivier2016measuring,metivier2016optimal}, on its application to seismic imaging, corresponds to a staggered grid finite difference scheme. 
	As a result, the horizontal and vertical components of the gradient are stored in different locations of the grid, which leads these authors to use an axis-biased discretization of the gradient norm, and thus to implement a Lipschitz regularity constraint w.r.t.\ the Manhattan distance.
	Our discretization corresponds instead to a forward finite difference scheme: although being in general less accurate, this is consistent with the correct isotropic gradient norm $|\cdot|_{\Vd}$, helping to eliminate substantial artifacts in the optimal transport solution.
	\end{remark}
	
	\subsection{SDMM algorithm}
	
	We present now an optimization strategy for solving problem \eqref{eq:BPvect_pq_h}.
	We focus here on the two cases $(p,q)=(1,1)$ and $(p,q)=(1,2)$ since the case $(p,q)=(2,2)$ can be tackled in an easier way.
	Indeed, the latter can be addressed by solving directly the system of optimality conditions, namely the discrete version of the Poisson equation \eqref{eq:H-1_oc}:
	\begin{equation}\label{eq:H-1_h_oc}
		 \Delta_{\Xc}(\phi)+\frac{1}{\lambda}\phi=\nu-\mu\,, \quad \text{where }\, \Delta_{\Xc}\phi \coloneqq\div_{\Xc} (\nabla_{\Xc} \phi)\,.
	\end{equation}	
	
	Problems \eqref{eq:KRvect_h} and \eqref{eq:dualvect_unbalq2_h} can be written as minimization problems of the following form
	\begin{equation}\label{eq:optim_gen}
		\min_\phi \sum_{i=1}^N f_i \left( L_i \phi \right),
	\end{equation}
	for $(f_i)_1^N$ convex functions and $(L_i)_1^N$ full rank operators.
	The SDMM algorithm \cite{combettes2011proximal} is designed to solve problems of this type and is shown in Algorithm \ref{alg:SDMM}.
	\begin{algorithm}
		\SetAlgoLined
		Starting from $\boldsymbol{y}^0,\boldsymbol{z}^0$ and $\tau>0$ \; 
		\For{$n=0,1,\dots$}{
			$\phi^n=\left( \sum_{i=1}^N L_i^T L_i\right)^{-1}\left(\sum_{i=1}^N L_i^T (y^n_i-z^n_i)\right)$ \;
			\For{$i=1,..,N$}{
				$y_i^{n+1}=\prox_{\tau f_i}(L_i \phi^n+z_i^n)$ \;
				$z_i^{n+1}=z_i^n+L_i \phi^n-y_i^{n+1}$\;}
		}
		\caption{SDMM}
		\label{alg:SDMM}
	\end{algorithm}
	More precisely,  the latter solves problem \eqref{eq:optim_gen} in the form
	\begin{equation}\label{eq:optim_gen_aug}
	\min_{\phi,\boldsymbol{y}} \left\{\sum_{i=1}^N f_i \left( y_i \right), \, L_i \phi=y_i, \,1\le i\le N\right\}.
	\end{equation}
	One of the main building blocks is the computation of the proximal operators of $(f_i)_1^N$. For a real valued convex function $f$ and a real parameter $\tau>0$, the proximal operator (with respect to the rescaled Euclidean norm $|\cdot|_{\Xc,2}$) is defined as
	\[
	x^*=\prox_{\tau f}(y)=\argmin_x \frac{1}{2\tau} |x-y|_{\Xc,2}^2 +f(x) \,.
	\]
	The approach can be particularly effective as long as the computation of the proximal operators and the inversion of the operator $\sum_i^N L_i^T L_i$ can be done efficiently.
	
	For problem \eqref{eq:KRvect_h}, the operators $L_i$ and the functions $f_i$ are
	\[
	\begin{gathered}
		L_1=\Id_{\Xc}\,, \quad L_2=\nabla_{\Xc}\,, \quad L_3= \frac{1}{\lambda}\Id_{\Xc}\,, \\
		f_1(u)=-\langle u,\nu-\mu\rangle_\Xc\,, \quad f_2(w)=i_{\Bc_{\Xc^d}}(w)\,, \quad f_3(u)=i_{\Bc_\Xc}(u)\,,
	\end{gathered}
	\]
	for $u\in \mathfrak{F}(\Gc;\V), w\in \mathfrak{F}(\Gc;\Vd)$,
	where $\Id_{\Xc}$ denotes the identity operator on $\mathfrak{F}(\Gc;\V)$ and $i_{\Bc_\Xc}$ (respectively $i_{\Bc_{\Xc^d}}$) is the convex indicator function \eqref{eq:indicator} of the set
	\[
	\Bc_\Xc=\left\{ u\in \mathfrak{F}(\Gc;\V)\,, |u|_{\Xc,\infty} \le 1 \right\} \; \big(\text{respectively } \Bc_{\Xc^d}=\{ w\in \mathfrak{F}(\Gc;\Vd)\,, |w|_{\Xc^d,\infty} \le 1 \}\big).
	\]
	For problem \eqref{eq:dualvect_unbalq2_h} instead
	\[
	\begin{gathered}
		L_1=\Id_{\Xc}\,, \quad L_2=\nabla_{\Xc}\,, \\
		f_1(u)=-\langle u,\nu-\mu\rangle_\Xc+\frac{1}{2\lambda} |u|^2_{\Xc^d,2}, \quad f_2(w)=i_{\Bc_{\Xc^d}}(w) \,.
	\end{gathered}
	\]
	
	The three proximal steps are inexpensive operations as they can be solved component-wise in closed form.
	For $q=1$, the first one amounts to solve the problem
	\[
	\begin{aligned}
	\bar{v}=\prox_{\tau f_1}(u)&=\argmin_{v} \frac{1}{2\tau}|v-u|_{\Xc,2}^2- \langle v,\nu-\mu\rangle_\Xc \\
	&=\Big( \argmin_{v_{i,j}} \frac{1}{2\tau}|v_{i,j}-u_{i,j}|_\V^2- \langle v_{i,j},\nu_{i,j}-\mu_{i,j}\rangle_\V \Big)_{i,j},
	\end{aligned}
	\]
	whose solution is simply
	\[
	\bar{v}_{i,j}=u_{i,j}+\tau(\nu_{i,j}-\mu_{i,j}) \,,\quad \forall (i,j).
	\]
	For the case $q=2$ the proximal operator is analogous, with an additional quadratic term, and the solution becomes
	\[
	\bar{v}_{i,j}=\frac{1}{1+\tau/\lambda}\left(u_{i,j}+\tau(\nu_{i,j}-\mu_{i,j})\right) \,,\quad \forall (i,j).
	\]
	The proximal operator of $i_{\Bc_\Xc}$ and $i_{\Bc_{\Xc^d}}$ are component-wise block soft-thresholding \cite{parikh2014proximal} with respect to the norms $|\cdot|_\V$ and $|\cdot|_\Vd$,
	\[
	\bar{v}_{i,j}=\Big(\prox_{\tau i_{\Bc_\Xc}}(u)\Big)_{i,j}=\Big(\argmin_{v\in i_{\Bc_\Xc}} |v_{i,j}-u_{i,j}|_\V^2\Big)_{i,j}\,,
	\]
	whose solution is
	\begin{equation}
	\label{eq:ball_proj}	
	\bar{v}_{i,j}=
	\begin{cases}
		u_{i,j} &\text{if $|u_{i,j}|_\V\le1$} \\
		\frac{u_{i,j}}{|u_{i,j}|_\V} &\text{else} 
	\end{cases}
	\quad \forall (i,j),
	\end{equation}
	and analogously for $i_{\Bc_{\Xc^d}}$.
	Finally, the operator $\sum_i^N L_i^T L_i$ is
	\[
	\begin{aligned}
		&\sum_{i=1}^N L_i^T L_i	=\Delta_{\Xc}+\left(1+1/\lambda^2\right)\Id_{\Xc}\,, \quad &&\text{for $q=1$}, \\
		&\sum_{i=1}^N L_i^T L_i =\Delta_{\Xc}+\Id_{\Xc}\,, \quad &&\text{for $q=2$}.
	\end{aligned}
	\]

	Algorithm \ref{alg:SDMM} is provably convergent.
	Since $f_i$ are all lower semi-continuous, proper and convex, and thanks furthemore to the duality result established in Proposition \ref{prop:duality_discrete}, for $n\rightarrow +\infty$ it holds \cite{boyd2011distributed}
	\begin{equation}\label{eq:SDMM_convergence}
		\begin{gathered}
			z_i^n\longrightarrow z_i \quad \text{and}\quad  y_i^n-L_i\phi^n\longrightarrow 0\,, \quad \text{for $1\le i\le N$,}\\
			\sum_{i=1}^N f_i \left( y_i^n \right) \longrightarrow \min_\phi \sum_{i=1}^N f_i \left( L_i \phi \right),
		\end{gathered}
	\end{equation}
	where $\boldsymbol{z}=(z_i)_{i=1}^N$ is a solution to the dual problem to \eqref{eq:optim_gen_aug}.
	The following proposition relates this to the solution of the Beckmann problem \eqref{eq:BPvect_pq_h}.
	
	\begin{proposition}\label{prop:sdmm_dual}
		Let $z$ be found through algorithm \eqref{alg:SDMM}, for either problem \eqref{eq:KRvect_h} (q=1) or \eqref{eq:dualvect_unbalq2_h} (q=2). Then:
		\begin{enumerate}[label=(\roman*)]
			\item $z_1=\tau(\mu-\nu)$, $z_2=\tau \sigma$, $z_3=\tau\lambda \delta$ ($q=1$),
			\item $z_1=\tau(\mu-\nu+\delta)$, $z_2=\tau \sigma$ ($q=2$),
		\end{enumerate}
		where $(\sigma,\delta)$ is a solution to \eqref{eq:BPvect_pq_h}.
	\end{proposition}
	\begin{proof}
		We start from the case $q=1$.
		It is sufficient to compute the dual problem to \eqref{eq:optim_gen_aug}: 
		\[
		\begin{aligned}
			\eqref{eq:optim_gen_aug}&=\min_{\phi,\boldsymbol{y}}\max_{\boldsymbol{z}} \sum_{i=1}^N f_i \left( y_i \right) + \sum_{i=1}^N z_i \left(L_i \phi-y_i \right)\\
			&=\max_{\boldsymbol{z}} \min_{\phi,\boldsymbol{y}} \sum_{i=1}^N f_i \left( y_i \right) + \sum_{i=1}^N z_i \left(L_i \phi-y_i \right).
		\end{aligned}
		\]
		The minimization in $y_1$ provides $z_1=\mu-\nu$ and following the same computations exposed in Section \ref{sec:vectorL1}, one readily obtains
		\[
		\begin{aligned}
			\eqref{eq:optim_gen_aug}&=\max_{\boldsymbol{z}}\min_{\phi} - |z_2|_{\Xc^d,1} -|z_3|_{\Xc,1}+\langle \phi, \mu-\nu \rangle_{\Xc}+\langle z_2, \nabla_\Xc \phi \rangle_{\Xc}+\langle z_3, \frac{\phi}{\lambda}\rangle_{\Xc} \\
			&=\max_{\boldsymbol{z}}\Big\{ -|z_2|_{\Xc^d,1} - |z_3|_{\Xc,1} \,, \;\div_{\Xc} (z_2)=\mu-\nu+\frac{z_3}{\lambda} \Big\}.
		\end{aligned}
		\]
		Up to changing sign and rescaling $z_3$ by $\lambda$, this coincides with problem \eqref{eq:BPvect_pq_h}. Finally, since Algorithm \eqref{alg:SDMM} operates a rescaling on the dual variable $\boldsymbol{z}$ by $\frac{1}{\tau}$ \cite[Section 3.1]{boyd2011distributed}, we obtain the claim.
		The case $q=2$ is deduced analogously. 
	\end{proof}	
	
	\begin{remark}[Stopping criterion]\label{rmk:stopping}	
		Thanks to \eqref{eq:SDMM_convergence} and Proposition \ref{prop:sdmm_dual}, we can consider
		\[
		(\sigma^n,\delta^n)=\left(z^n_2/\tau,\div_{\Xc}(z^n_2)/\tau+\nu-\mu\right)
		\]
		as a candidate solution to \eqref{eq:BPvect_pq_h} in order to estimate the duality gap.
		By definition of the SDMM algorithm, $\sum_{i=1}^N f_i \left( y_i^n \right)=f_1(y_1^n)$ and we can therefore control the (relative) duality gap by
		\[
		\Delta^n=\frac{1}{|f_1(y_1^n)|}\left| |\sigma^n|_{\Xc^d,1} + \frac{\lambda}{q} |\delta^n|^q_{\Xc,q} +f_1(y_1^n) \right|.
		\]
		Of course, without a control on the constraints $L_i \phi=y_i$ this information is meaningless. We can evaluate the (relative) error on the constraints as
		\[
		r^n=\frac{\sum_i^N |L_i \phi^n-y_i^n|_{\Xc,2}}{\text{\normalfont max}\big(  \sum_i^N |L_i \phi^n|_{\Xc,2}   \ ,\ \sum_i^N |y_i^n|_{\Xc,2}  \big)}\,.
		\]
		Then, for $\varepsilon \in\R_+$, we can set for simplicity as stopping criterion $\max(\Delta^n,r^n) \le \varepsilon$.
	\end{remark}
	
	\begin{remark}[Computational complexity]\label{rmk:complexity}
		The SDMM algorithm has in general a sub-linear rate of convergence \cite{wang2013online}, without further regularity hypotheses on the objective functions $f_i$, both in the primal objective and the constraints. This implies that $O(\frac{1}{\varepsilon})$ iterations are necessary to solve problem \eqref{eq:optim_gen} up to a tolerance $\varepsilon$. 
		At each iteration, the most expensive operation is the inversion of the operator $\sum_i L_i^TL_i$, which corresponds to the solution of a Poisson equation valued in $\V$.
		As already pointed out in Section \ref{sec:notation}, by expanding the discrete functions of $\mathfrak{F}(\Gc;\V)$ in a basis of $\V$, the operators $\div_{\Xc}$ and $\nabla_{\Xc}$ act component-wise on the coefficients of the basis and so does their composition $\Delta_{\Xc}$, meaning that they can be represented as block-diagonal matrices. The corresponding linear systems can therefore be solved separately in each component, which implies solving $n$ scalar Poisson equations. 
		In our setting, since the problem is discretized on a Cartesian grid using finite differences, this can be done very efficiently and with guaranteed $N_\Gc\log N_\Gc$ complexity via Fast Fourier Transform \cite{swarztrauber1974direct}, where $N_\Gc=N_x \times N_y$ is the total number of degrees of freedom (more generally, for $d>2$, $N_\Gc=\prod_{i}^{d} N_i$ with $N_i$ number of discretization points along dimension $i$), i.e. the number of grid points. Note that similar perfomarnces can be obtained via multigrid solvers \cite{brandt1977multi}, which can deal with more general discretizations on unstructured grids.
		An $\varepsilon$ precise solution can then be found with $O(\frac{n N_\Gc\log N_\Gc}{\varepsilon})$ operations and the method is therefore suited to scale to large dimensions, both with respect to the number of grid points and the dimension of the vector space $\V$.
		Finally, we stress also that, as in the scalar case, requiring $\mu, \nu\in \mathfrak{F}(\Gc;\Co)$ take their values in the cone $\Co \subset V$ does not reflect on the computational complexity of the algorithm, an advantage of the $L^1$ optimal transport problem compared to other formulations.
	\end{remark}

	\begin{remark}
		In order to solve problem \eqref{eq:BPvect_pq_h} for $(p,q)=(1,2)$ we may try to adopt proximal-splitting techniques that take advantage of the higher regularity of the objective functional, such as the accelerated forward-backward algorithm \cite{combettes2011proximal} or the Chambolle-Pock algorithm \cite{chambolle2011first}. Nevertheless, we have not experienced a sensible speed-up by using these algorithms with respect to the SDMM, at least for the regimes of $\lambda$ of interest.
		Some of these methods can avoid the inversion of the Laplacian, which makes individual iterations faster. However, the number of iterations required to reach convergence increases so much that it is not cost effective eventually. 
	\end{remark}

	\begin{remark}[Operator norm]
	Throughout this paper, we chose to formulate the vector-valued $L^1$ optimal transport problem \eqref{eq:dualvect} using the gradient norm $|\nabla \phi|_{\Vd}$, which corresponds to the Frobenius norm in the identification $\Vd = L(\R^d;V)$; the flow norm $|\sigma|_{\Vd}$ also appears in the dual formulation \eqref{eq:BPvect}.
	Alternatively, we may have used the operator norm $|\nabla \phi|_{L(\R^d;V)}$ for the gradient, and thus the dual norm for $|\sigma|_{L(\R^d;V)^*} = \Tr(\sqrt{\sigma^\top \sigma})$ for the flow.
	The operator norm has the advantage of preserving the connection \eqref{eq:lip_iff_grad1} with Lipschitz regularity, which only holds up to a multiplicative constant in the case of the Frobenius norm \eqref{eq:Fr_op_norms}.
	The adaptation of the theoretical results and algorithms presented in the paper is straightforward, except for the proximal operator of $i_{\Bc_{\Xc^d}}$, which is obtained by pointwise orthogonal projection \eqref{eq:ball_proj} at each site $(i,j)$ of the grid onto the unit ball of $\Vd = L(\R^d;V)$ w.r.t.\ the chosen norm. In this specific step, the normalisation procedure \eqref{eq:ball_proj} which is adequate for $|\cdot|_{\Vd}$, should be replaced with the orthogonal projection
	\begin{equation}
	\label{eq:proj_op}	
		\argmin\{ |A-B|_{\Vd},\, |B|_{L(\R^d;V)} \leq 1\}.
	\end{equation}
	Assuming that $V = \R^n$, up to choosing an orthonormal basis, the solution to \eqref{eq:proj_op} is obtained \cite[Theorem 2.1]{cai2010singular} as $B_{\mathrm{opt}} = U \diag(\min(\lambda_1,1),\cdots, \min(\lambda_d,1)) V$, where $A = U \diag(\lambda_1,\cdots,\lambda_d) V$ is the singular value decomposition (with $U \in \R^{d \times d}$ and $V \in \R^{d \times n}$ obeying $ U U^\top = V V^\top = \Id_d$, and $\lambda_1,\cdots,\lambda_d\geq 0$). 
	From the numerical standpoint the added numerical cost is linear in the problem size, since $d$ and $n$ are small and independent of the grid dimensions.
	\end{remark}

	\section{Numerical results and application to FWI}\label{sec:numerics}
	
	In this section we consider the application of the $L^1$ transport for vector valued measures to a specific inverse problem: the full waveform inversion (FWI) problem.
	
	\subsection{The acoustic wave equation and the FWI inverse problem}
	
	The acoustic wave equation with constant density $\rho$ can be written as (where here $\nabla$ and $\div$ are the classical gradient and divergence operator, corresponding to the case $\V=\R$ in Section \ref{sec:notation})
	\begin{equation}\label{eq:acoustic}
		\begin{cases}
			\partial_t v-\frac{1}{\rho}\nabla p = 0 \\
			\partial_t p-\rho c^2 \div(v) = f
		\end{cases}
	\end{equation}
	where $p:D\subset\R^m\rightarrow\R$ is the scalar pressure field, $v:D\rightarrow\R^m$ is the vector field of particles' velocity and $f$ is a given source term.
	The objective of FWI, in the simpler acoustic setting, is to reconstruct the velocity of the medium $c:D\rightarrow\R_+$ from observed data $v^{obs},p^{obs}$, which are signed oscillatory functions.
	
	We restrict here to the two dimensional case for simplicity, $m=2$, and we consider two-dimensional vector valued data $v^{obs}=(v_x^{obs},v_z^{obs})$. In this setting, we can introduce the space  of symmetric matrices $\V=\mathrm{S}^2$ and lift $v^{obs}$ in the self-dual cone of symmetric positive semi-definite matrices $\Co=\mathrm{S}^2_+$ (thus the lifted data dimension is $n=\dim(\mathrm{S}^2)=3$).
	We consider for this purpose Pauli's transformation (in dimension 2) as a lift operator $L$:
	\begin{equation}\label{eq:Pauli}
		v=(v_x,v_z) \;\longmapsto\; L(v_x,v_z)=\begin{bmatrix}
			\alpha-v_x & v_z \\ v_z & \alpha+v_x
		\end{bmatrix}, \quad \text{with $\alpha=\sqrt{v_x^2+v_z^2}$} \,.
	\end{equation}
	The advantage of \eqref{eq:Pauli} is that it is bijective from $\R^2$ to its image, i.e.\ no information is lost in the transformation. Another possibility could be to consider the tensor product
	\[
	v=(v_x,v_z)\;\longmapsto\; v^T v=
	\begin{bmatrix}
		v_x^2 & v_x v_z  \\ v_x v_z & v_z^2
	\end{bmatrix},
	\]
	however this lift is invariant with respect to sign changes of $v$ which limits its effectiveness. 
	This sign invariance amounts to doubling the frequency of oscillating signals, and thus to a greater risk of \emph{cycle skipping} in the application of interest (see Section \ref{ssec:shift} below).
	
	We write the FWI problem as
	\begin{equation}\label{eq:FWI}
		\min_{c} \sum_{s=1}^{N_s} \Tc_{p,q}^\lambda(d_{s}[c],d^{obs}_s), \quad \text{such that \eqref{eq:acoustic} holds,}
	\end{equation}
	where $d_s[c]=L(R v_{x,s},R v_{z,s})$ and $v_{x,s},v_{z,s}$ are the computed particles' velocities given the model $c$ (for simplicity omitting the dependencies), $d^{obs}_s=L(R v^{obs}_{x,s},R v^{obs}_{z,s})$. The operator $R$ evaluates the velocity at specific points in the domain, the receivers locations, for a certain time span. The peculiarity of FWI is that the whole recording, i.e.\ the whole form of the wave, is used in the inversion problem. The index $s$ runs instead over all the possible sources $f_s$: the FWI problem is usually performed by running several simulations in parallel for different sources. Finally, $L$ is the lift function defined in \eqref{eq:Pauli}.
	
	\begin{remark}\label{rmk:lorentz}
		The space of symmetric matrices $\mathrm{S}^{n'}$, endowed with the Frobenius scalar product, is a finite dimensional Euclidean space and is thus isometric to $\R^{n' (n'+1)/2}$ endowed with the usual scalar product. When $n'=2$, such an isometry is 
		$\begin{pmatrix}
			a & b\\
			b & c
		\end{pmatrix}
		\mapsto  2^{-\frac 1 2} (\frac {a-c} 2, b,\frac {a+c} 2),
		$
		and it maps the cone $\mathrm{S}^2_+$ of symmetric positive semi-definite matrices onto the Lorentz cone $\{(x,z,t)\in\R^3, \ \sqrt{x^2+z^2}\le t\}$.
		This isometry eases the numerical implementation.
		In these coordinates, Pauli's lift above reads as $(v_x,v_z) \mapsto (v_x,v_z,\alpha)$, up to a multiplicative constant.	
	\end{remark}
	
	\subsection{Calibration of $\lambda$}\label{ssec:lambda}
	
	The parameter $\lambda$ controls mass transport by tuning the penalization on the divergence constraint.
	On the one hand, if it is too small, mass can be deformed without being transported and the model is intuitively close to an $L^q$ distance on the difference of the two measures: $\Tc_{p,q}^\lambda \sim \frac \lambda q \int_{\Om}|\mu-\nu|_\V^q$ as $\lambda \to 0$.
	On the other hand, it is important to notice that in the seismic imaging application the two measures can have different total mass.
	This mass is in fact related to the energy of the signals if one uses Pauli's lift above and considers $q=2$ (which is our original motivation for considering $q=2$), or the tensor product lift with $q=1$. Observe nevertheless that the total energy of the signals can be quite different, since these are recorded only partially in the domain.
	In the other cases, such as Pauli's lift with $q=1$, the total mass is a non-physical quantity, which has no reason to be the same.
	One can check that $\Tc_{p,q}^\lambda \sim \frac \lambda q |\int_{\Om}\mu -\int_{\Om}\nu|_\V^q$ as $\lambda \to \infty$. The misfit function thus becomes proportional to the $q$-th power of the difference between the masses of $\mu$ and $\nu$, for large $\lambda$, which is not a very relevant objective for the problem of interest.
	The $L^q$ and the difference of masses behaviors lead to poor results in the inversion problem, and the parameter $\lambda$ thus has to be carefully selected in order to avoid these.	
	
	Assume that the physical dimension of the measures is an amplitude, which is denoted $[\mu]=[\nu]=A$. Consequently the physical dimension of the flux is an amplitude times a length, $[\sigma]=A l$, and the physical dimension of $\lambda$ is $[\lambda]=\frac{l^p}{A^{q-p}}$. For $(p,q)=(1,1)$, the parameter $\lambda$ thus represents a length, whereas for $(p,q)=(1,2)$ it represents a length over the amplitude.
	Examples \ref{ex:deltas_KRvect} and \ref{ex:deltas_p1q2} help having a better understanding. For $(p,q)=(1,1)$, $\lambda$ provides half the maximum distance up to which mass can be transported, regardless of the mass involved. By knowing the maximum displacement we want to allow in the transport problem, it can be tuned accordingly.
	For $(p,q)=(1,2)$, $\lambda$ specifies the maximum distance up to which a unit of mass can be transported. This means that large amounts of mass can be transported at large scales, whereas small amounts only at small scales, proportionally to their mass.
	These considerations can be of help for the calibration of this parameter in practice.
	The maximum transport distance can be set by comparing the initial and final measures in order to allow the transport of the main structures (see for example Figure \ref{fig:seismograms}). In general, in the specific case of transporting signals, we may want to allow displacements of the order of the wavelengths involved. Smaller values are not relevant whereas bigger values may be dangerous due to the mass imbalance. The wavelength can be estimated from the frequency of the injected signal. The case $(p,q)=(1,2)$ is more delicate and the parameter $\lambda$ can be estimated taking into account the minimum value of mass allowed to be transported at a certain distance.
	
	We conclude this section by highlighting that the calibration of $\lambda$, following the considerations just made, depends on the scaling of the domain and, in the case $q=2$, also of the measures. A simple computation shows that if $\Omega$ is rescaled by a factor $L$ then $\lambda$ has to be scaled by the same factor, in both cases. If the measures are scaled by a factor $M$ instead, in the case $q=2$, then $\lambda$ has to be rescaled by the same factor again.  
	This scaling behavior, which is consistent with the above physical dimension considerations, has to be taken into account when solving practical instances (see Section \ref{ssec:seismo}).
	
	\subsection{Sensitivity to shifts of $\mathcal{T}_{p,q}^\lambda$}\label{ssec:shift}
	
	Optimal transport is generally used in inverse problems in order to acquire sensitivity to (time) shifts in the data, a feature which is not encoded using $L^p$ distances. In seismic imaging, this issue is referred to as \textit{cycle-skipping}. Due to signed data, optimal transport based misfit functions may fail to provide this sensitivity (see Examples \ref{ex:deltas_KRvect} and \ref{ex:KRshift}). By lifting the data to the space $\Mc(\Om;\Co)$ we aim at solving this issue.
	Here we test for this reason the different behavior of the misfit functions $\Tc_{p,q}^\lambda$, for the different values of the parameters $p,q$ and $\lambda$. We consider the following signal
	\[
	v(t)=(v_x(t),v_z(t))=\Big(\frac{d}{d t} f(t), \frac{d^2}{\d t^2} f(t) \Big),
	\]
	where $f(t)=\frac{1}{\sqrt{2\pi}\alpha}\exp^{-t^2/\alpha^2}$ with $\alpha=\frac{4}{3}$, and we evaluate the misfit between the signal itself and $v(t-t_0)$ for different values of $t_0$. The results are shown in Figure \ref{fig:shift} and are compared with the $L^2$ and the $\KR$ norms.
	
	The $L^2$ norm presents more local minima, explaining the drawback of using it in the FWI problem. The $\KR$ norm \eqref{eq:KRscal} presents several minima as well for small values of the parameter $\lambda$. For a sufficiently big value, there is a single well-defined minimum, but the misfit saturates to a constant value. The time average of $v_x(t)$ and $v_z(t)$ is in fact zero and the $\KR$ norm is therefore not sensitive to big shifts (see Example \ref{ex:KRshift}).
	For signals with non-zero mean the $\KR$ norm recovers sensitivity. However, typically signals in seismic imaging do have zero mean, which is a major issue for its application in this setting.
	The misfits $\mathcal{T}_{p,q}^{\lambda}$ are not convex with respect to the shift either, but for a sufficiently big value of $\lambda$ they have a unique minimum, that is as long as $\lambda$ is sufficiently big to allow sufficient transport of mass. We can also notice, empirically in this example, that $\mathcal{T}_{1,1}^{\lambda}$ is always increasing with respect to the shift, an advantage of our construction with respect to the $\KR$ norm \eqref{eq:KRscal} applied to signed measures. Finally, for $p=1$ the misfit function is more selective around the minimum contrary to the case $p=2$, where it is more flat, suggesting why the $H^{-1}$ norm may not be a good choice.
	
	\begin{figure}
		\centering
		\subfloat[]{\includegraphics[trim={0.5cm 0.3cm 1.2cm 1.1cm},clip,width=0.33\textwidth]{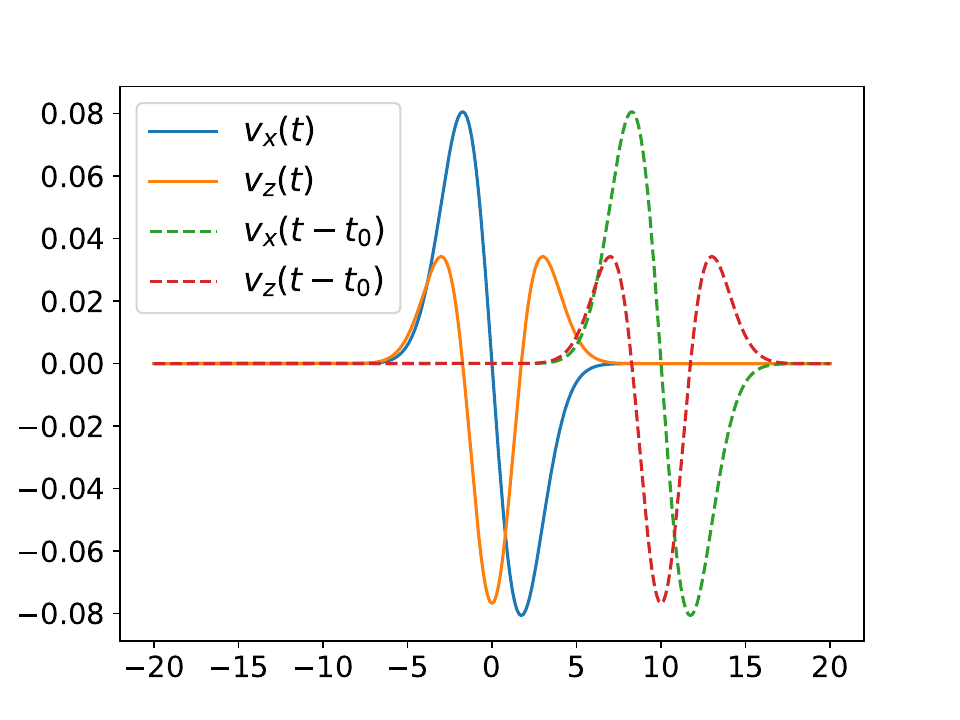}}
		\subfloat[$L^2$ norm]{\includegraphics[trim={0.5cm 0.3cm 1.2cm 1.1cm},clip,width=0.33\textwidth]{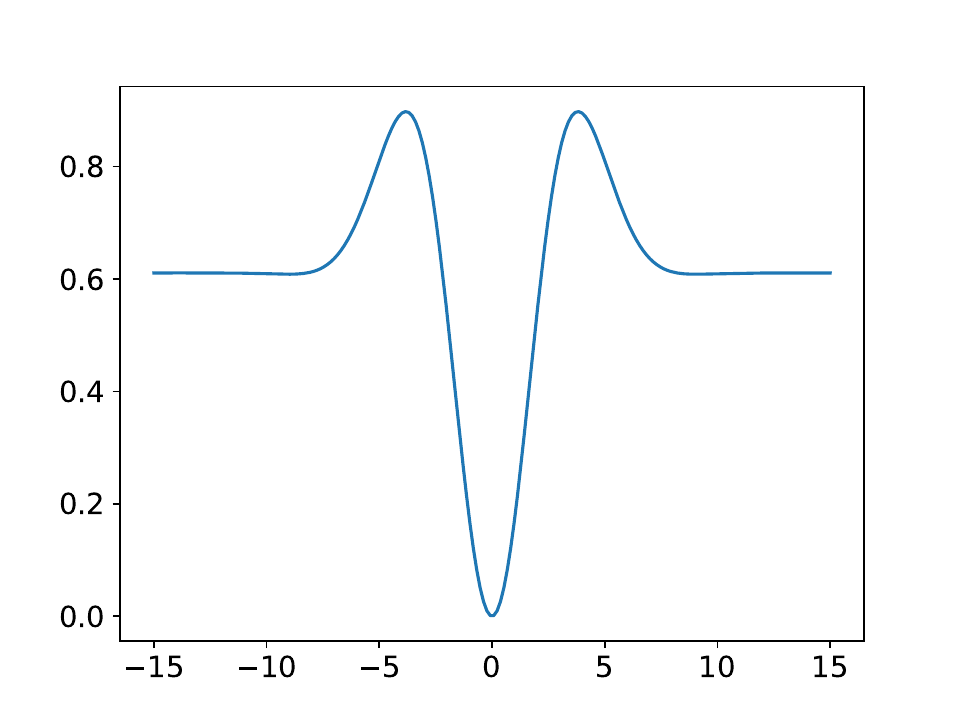}}
		\subfloat[$\KR$ norm]{\includegraphics[trim={0.5cm 0.3cm 1.2cm 1.1cm},clip,width=0.33\textwidth]{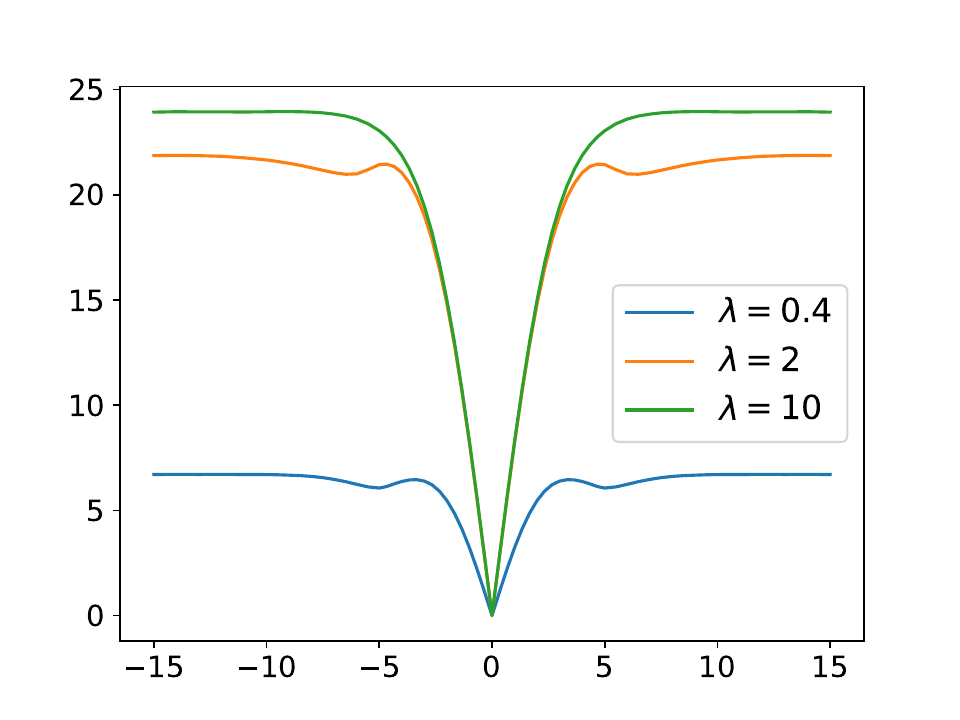}} \\
		\subfloat[$\lambda=0.4$]{\includegraphics[trim={0.5cm 0.3cm 1.2cm 1.1cm},clip,width=0.33\textwidth]{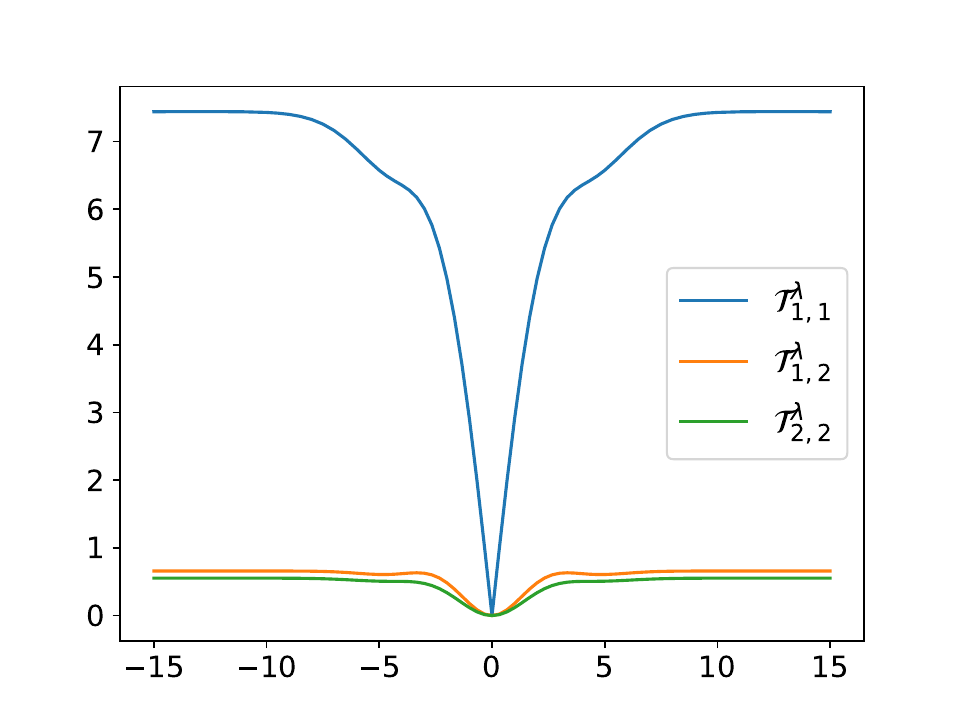}}
		\subfloat[$\lambda=2$]{\includegraphics[trim={0.5cm 0.3cm 1.2cm 1.1cm},clip,width=0.33\textwidth]{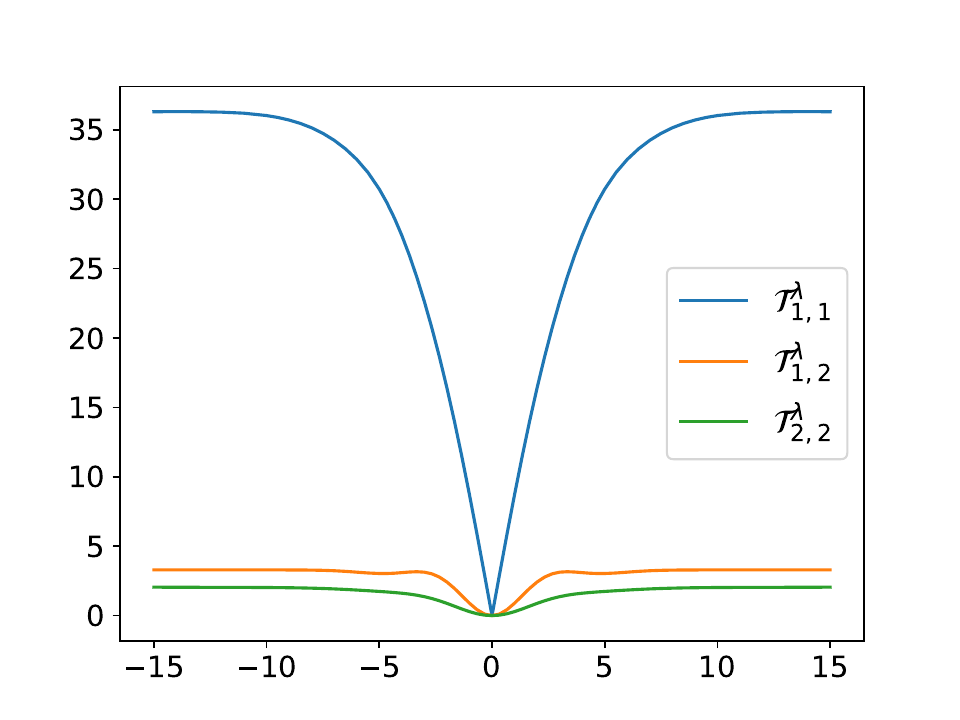}}
		\subfloat[$\lambda=10$]{\includegraphics[trim={0.5cm 0.3cm 1.2cm 1.1cm},clip,width=0.33\textwidth]{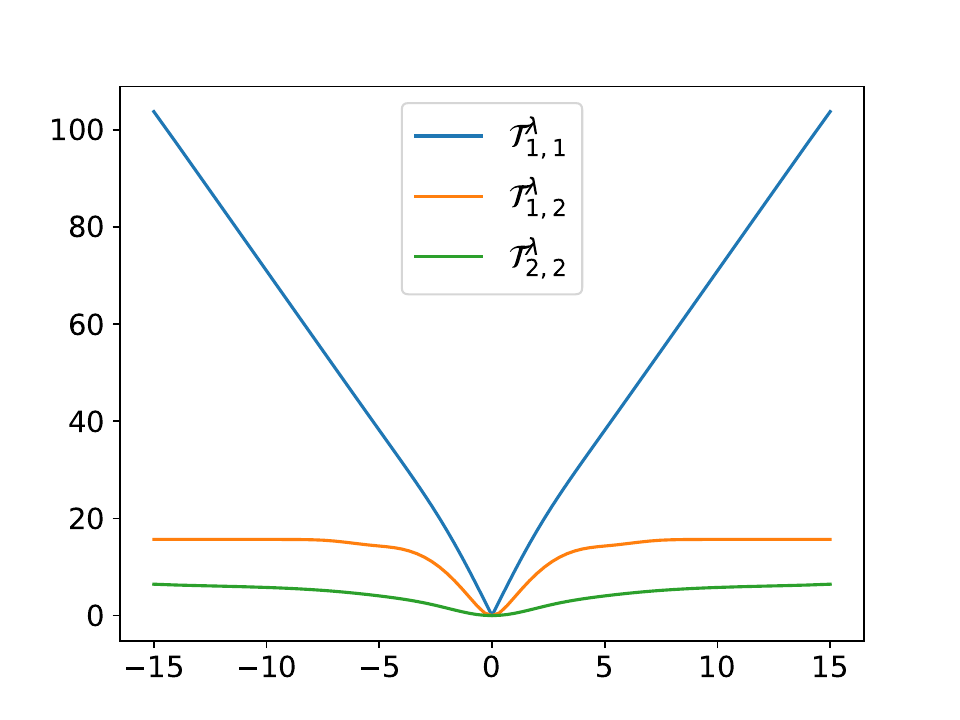}}
		\caption{Sensitivity to cycle-skipping of $\mathcal{T}_{p,q}^{\lambda}$ (d-f), for different values of $\lambda$, and comparison with the $L^2$ distance (b) and the $\KR$ norm \eqref{eq:KRscal}. In each graph (b-f) it is represented the evaluation of the misfit function between the multicomponent signals $(v_x(t),v_z(t))$ and $(v_x(t-t_0),v_z(t-t_0))$ (a), with respect to the time shift $t_0$. For the $L^2$ and $\KR$ norms, the distance is evaluated component-wise and summed.}
		\label{fig:shift}
	\end{figure}

	\subsection{Transport between seismograms}\label{ssec:seismo}
	
	We perform now a test case computing the unbalanced $L^1$ optimal transport problem between two-dimensional vector valued seismograms, in order to further inspect the influence of the parameter $\lambda$ and to compare the two cases $(p,q)=(1,1)$ and $(p,q)=(1,2)$. The setting for the test is as follows.
	We consider a two-dimensional physical domain $D=[a,b]\times[c,d]\subset\R^2$. In $D$, $N_r$ receivers are positioned uniformly along the $x$ axis at a certain constant depth $z$. The data domain, corresponding to the registration time of each receiver, is given by $\Omega=[0,T]\times\{1,..,N_r\}$ and can be considered as a semi-discrete bidimensional domain. This allows to interpret the measurements as semi-discrete bidimensional seismograms. By considering a regular time sampling for $(v_x,v_z)$, this provides a regular Cartesian grid $\Gc$. 
	The components of the lift \eqref{eq:Pauli} are presented in Figure \ref{fig:seismograms} for two examples. 
	They correspond to the lifted particles' velocities generated through the acoustic wave equation \eqref{eq:acoustic} with two different velocity models $c$. In these examples, the traces are registered by $N_r=169$ receivers and regularly sampled $3000$ times. The single point source $f$ is placed in the same location as the $84$-th receiver and injects a Ricker signal.
	
	We solve the problem between the two vector valued measures $\mu$ and $\nu$ represented in Figure \ref{fig:seismograms}. In order to avoid anisotropic effects due to the different dimensions of $\Omega$, and to have a meaningful transport problem, we rescale one of the two directions. Indeed, on the $x$ direction the physical dimension is a length whereas on the $z$ direction it is time. For example, we can rescale the time by the average velocity of the model, $\overline{c}\approx 2000 \,m/s$. Considering the recording time of $\approx 7$ seconds and the horizontal size of the domain $D$ of $16.9$ kilometers, this provides approximately a square domain.
	We consider therefore $\Om=[0,\ell]^2$, with $\ell=N_r-1$. We stress that this scaling has an influence on the choice of the parameter $\lambda$, as highlighted in Section \ref{ssec:lambda}.
	In Figures \ref{fig:delta_x}, \ref{fig:delta_z} and \ref{fig:delta_alpha}, we show the three components of the computed optimal $\delta$ measures for different values of $\lambda$, in the two cases $q=1$ and $q=2$. The value of $\delta$ determines how much mass is injected in the problem and therefore gives a measure of what is or is not transported. It can be compared with the difference of the two measures, i.e.\ the zero transport case (Figure \ref{fig:seismograms}).
	
	\begin{figure}
		\centering
		\setlength{\tabcolsep}{4pt}
		\renewcommand{\arraystretch}{1.5}
		\begin{tabular}{C{0.05\textwidth}C{0.27\textwidth}C{0.27\textwidth}C{0.27\textwidth}C{0.05\textwidth}}
			& \includegraphics[trim={0.5cm 1.5cm 1cm 12.3cm},clip,width=0.28\textwidth]{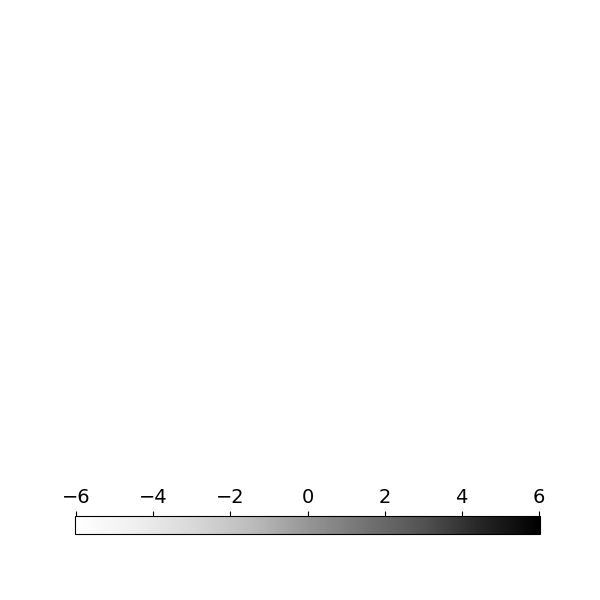} &
			\includegraphics[trim={0.5cm 1.5cm 1cm 12.3cm},clip,width=0.28\textwidth]{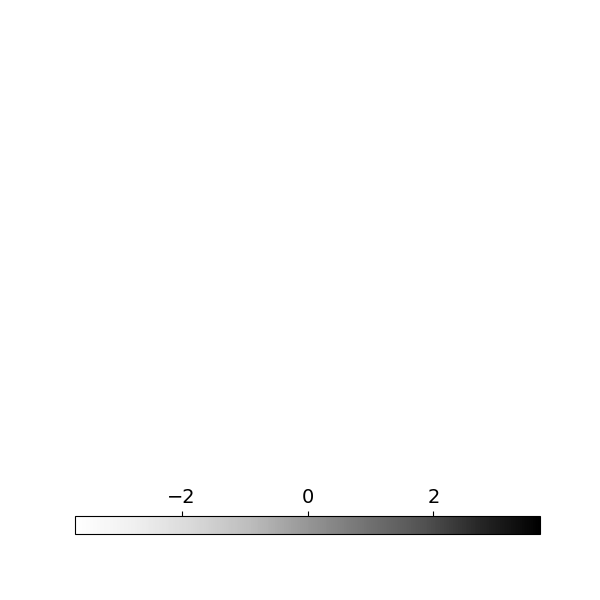} &
			\includegraphics[trim={0.5cm 1.5cm 1cm 12.3cm},clip,width=0.28\textwidth]{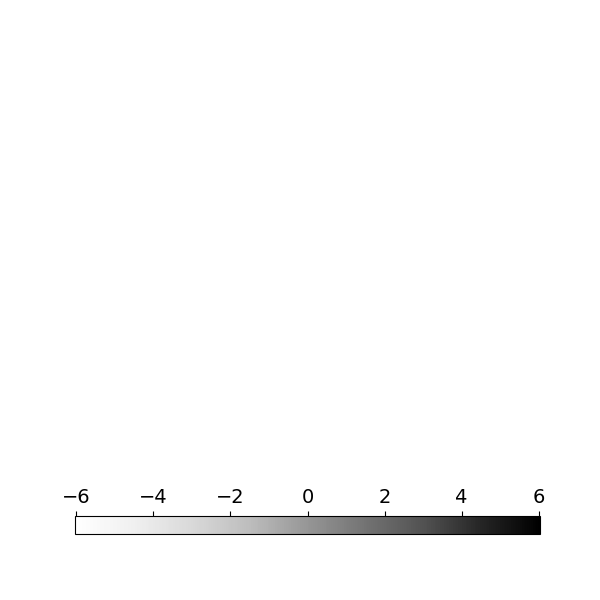} & \\
			{\small $\mu$} &
			\includegraphics[trim={2.5cm 1.3cm 3.4cm 2cm},clip,width=0.27\textwidth]{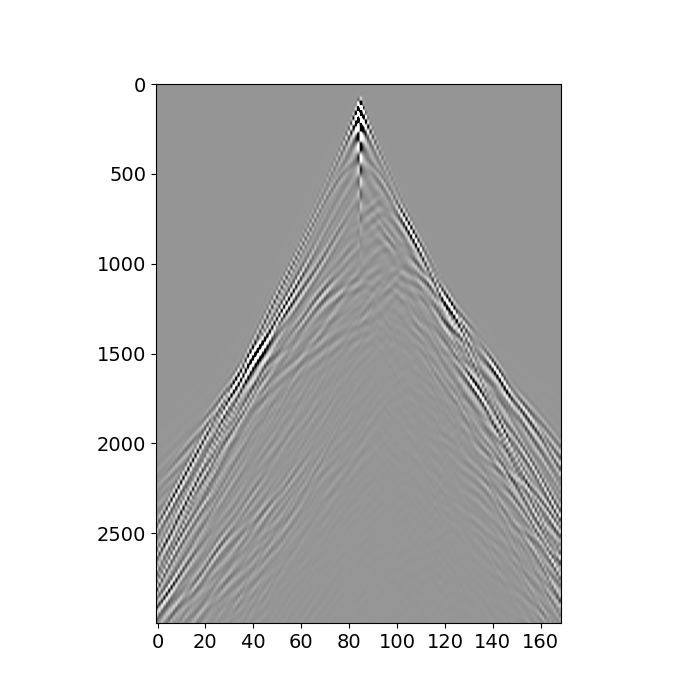} & 
			\includegraphics[trim={2.5cm 1.3cm 3.4cm 2cm},clip,width=0.27\textwidth]{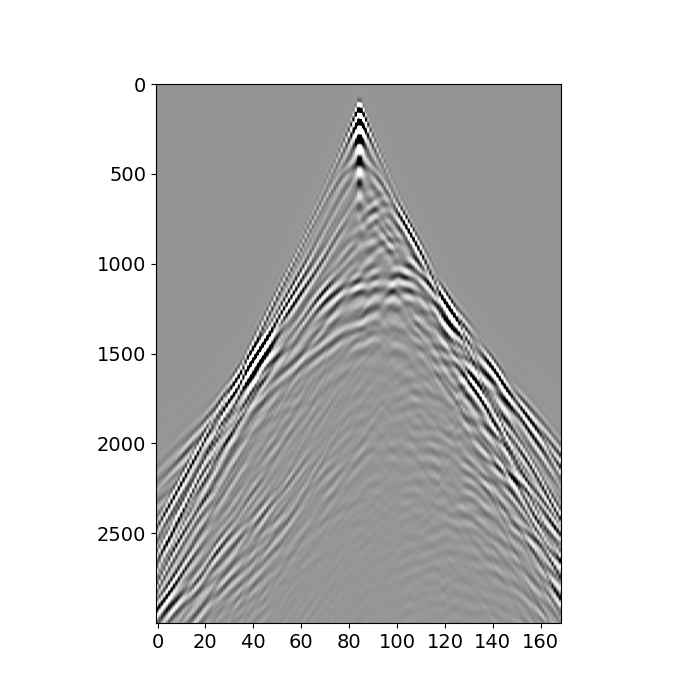} &
			\includegraphics[trim={2.5cm 1.3cm 3.4cm 2cm},clip,width=0.27\textwidth]{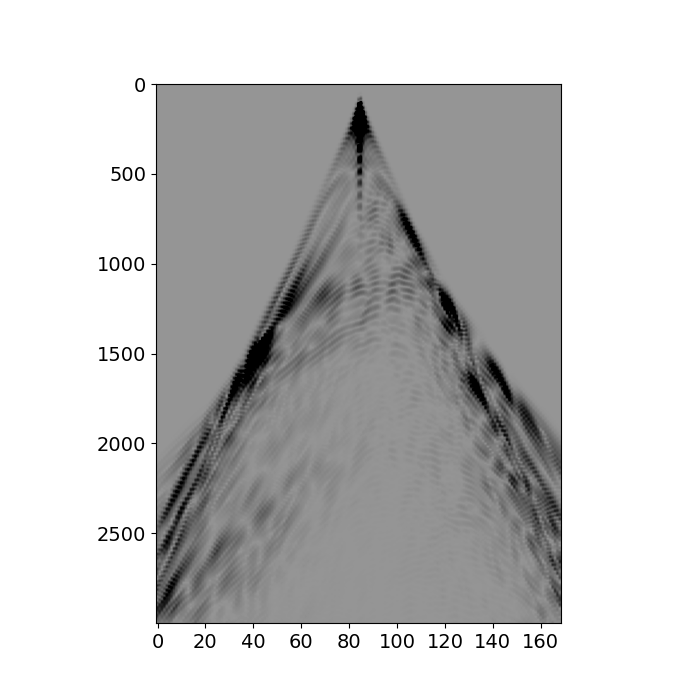} & \\
			{\small $\nu$} &
			\includegraphics[trim={2.5cm 1.3cm 3.4cm 2cm},clip,width=0.27\textwidth]{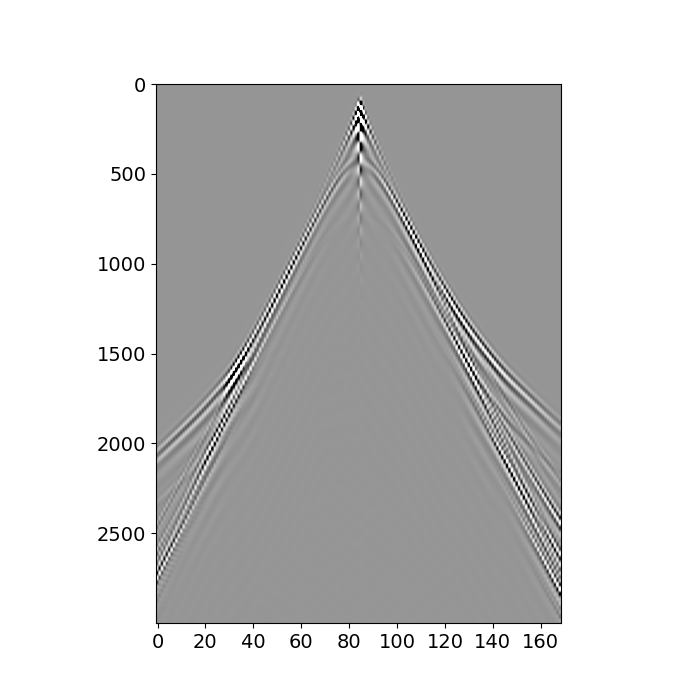} &
			\includegraphics[trim={2.5cm 1.3cm 3.4cm 2cm},clip,width=0.27\textwidth]{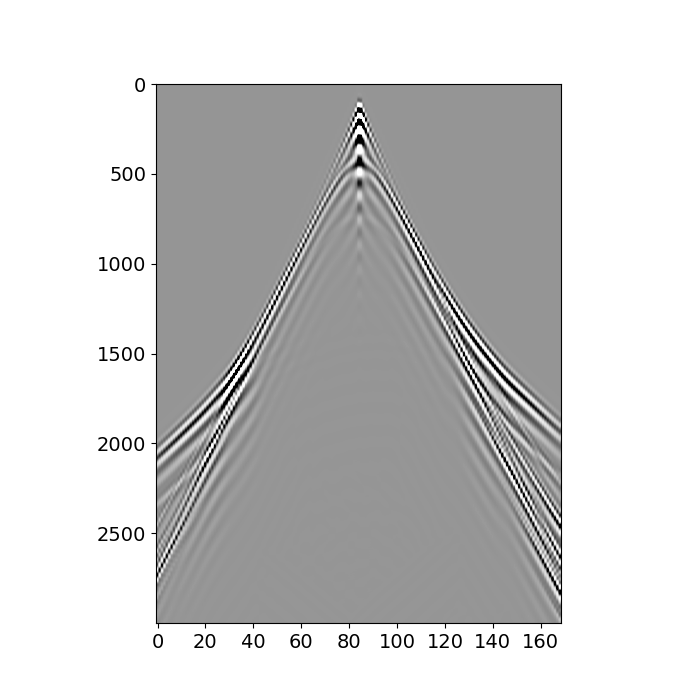} & \includegraphics[trim={2.5cm 1.3cm 3.4cm 2cm},clip,width=0.27\textwidth]{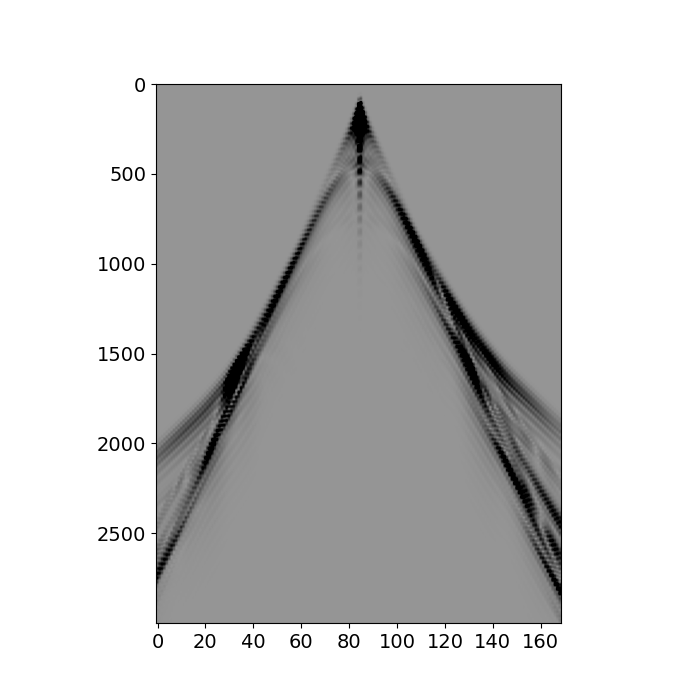} & \\
			{\small $\mu-\nu$} &
			\includegraphics[trim={2.5cm 1.3cm 3.4cm 2cm},clip,width=0.27\textwidth]{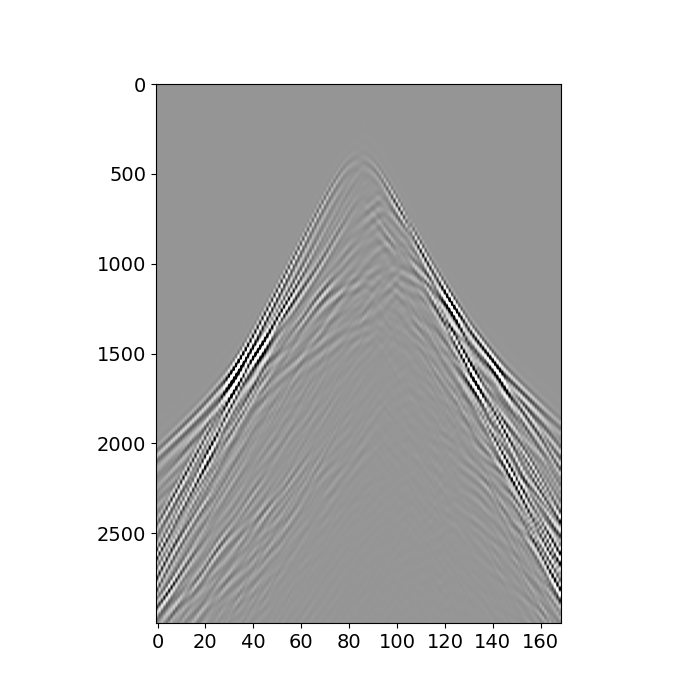} &
			\includegraphics[trim={2.5cm 1.3cm 3.4cm 2cm},clip,width=0.27\textwidth]{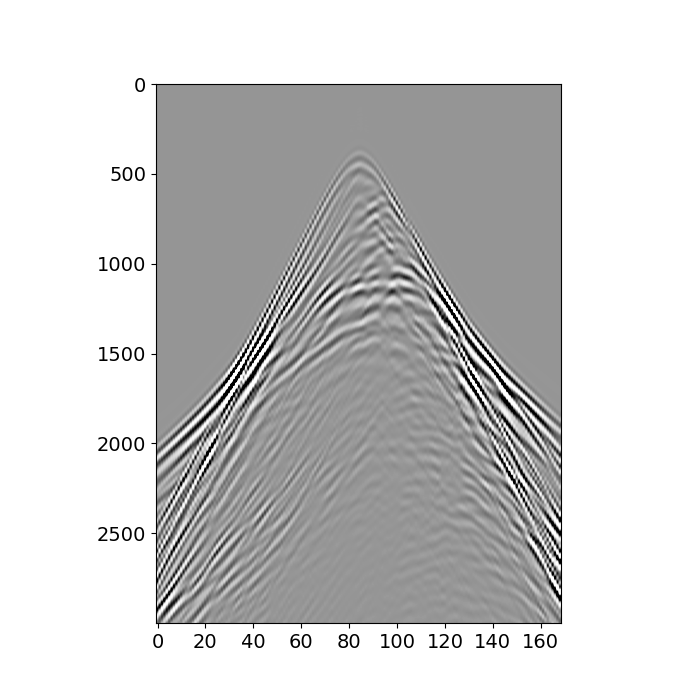} & \includegraphics[trim={2.5cm 1.3cm 3.4cm 2cm},clip,width=0.27\textwidth]{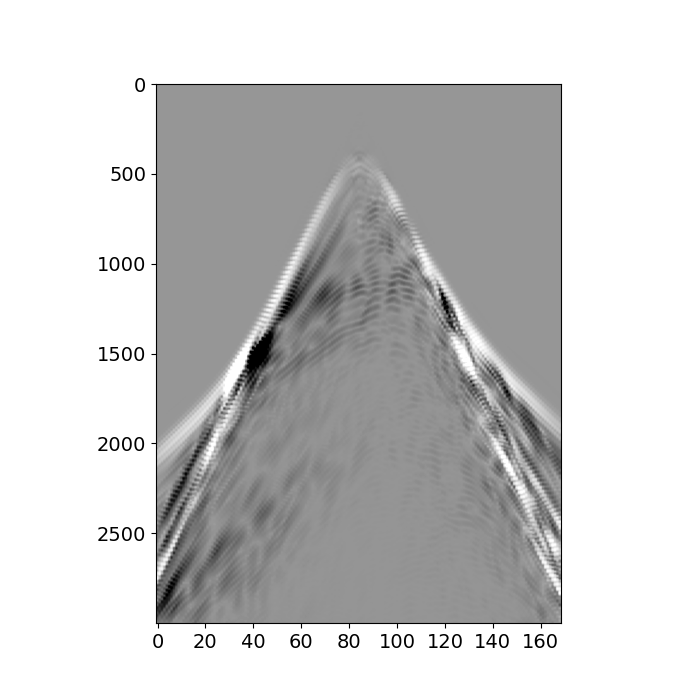} & \\
		    & {\small $v_x$} & {\small $v_z$} & {\small $\sqrt{v_x^2+v_z^2}$} & \\
		\end{tabular}
		\caption{Example of two-dimensional seismograms of particles' velocities generated with two different velocity models $c$, respectively the \textit{true model} and the \textit{initial model 3} represented in Figure \ref{fig:models}, lifted via \eqref{eq:Pauli} (see also Remark \ref{rmk:lorentz}). Recording of $3000$ regular time samples (y axis) from $169$ evenly spaced receivers (x axis). For the sake of the representation, data are clipped to $10\%$ of their maximum value and the color scale varies in each column.}
		\label{fig:seismograms}
	\end{figure}

	First of all notice that, in both case $q=1$ and $q=2$, the order of magnitude of the first two components of $\delta$ (Figures \ref{fig:delta_x} and \ref{fig:delta_z}) decreases for increasing values of $\lambda$. Mass is effectively transported on these components for increasing values of $\lambda$. On the contrary, the third component (Figure \ref{fig:delta_alpha}) does not change noticeably and for big values of $\lambda$ it dominates on the first two by several orders of magnitude. The difference of mass between $\mu$ and $\nu$ concentrates in fact in the third component, since the velocities $v_x$ and $v_z$ have nearly zero average. If $\lambda$ is too big, this component is therefore predominant in the transport problem and the misfit $\Tc_{p,q}^\lambda$ will have an undesired behavior, as mentioned in Section \ref{ssec:lambda}.

	For $q=1$, we can notice as expected that the mass can be transported only at a maximum distance, depending on the value of $\lambda$. For small values ($\lambda=10^{-3}\ell$ or $\lambda=10^{-2}\ell$), mass can be transported at very short distance and $\delta$ resembles the norm of the difference of the two measures, meaning that the model is close to an $L^q$ distance. The main structures are well transported instead for a sufficiently big value ($\lambda=10^{-1}\ell$). For $\lambda=\ell$ the mass is allowed to be transported everywhere in the domain (recall $\Omega=[0,\ell]^2$), however, due to the difference in total mass, $\delta$ is non zero.
	In this case, differently from the previous smaller values of $\lambda$, the third component of $\delta$ is signed and represents the unavoidable addition of mass.
	Increasing $\lambda$ does not change at this point the solution but only makes this latter term more predominant in the overall cost.
	
	On the other hand, the situation changes completely in the case $q=2$. No matter the value of $\lambda$, sufficiently big values of mass can always be transported at every distance, resulting in a rather diffused solution. 
	On the contrary, sufficiently small values of mass will not be transported, which is the reason why the third component of $\delta$ is not signed even for $\lambda=\ell$, differently from the case $q=1$. Transport of the whole amount of available mass can be recovered only in the limit $\lambda\rightarrow\infty$.
	For the small values of $\lambda$, $\delta$ resembles again the difference of the two measures, although more diffused than the case $q=1$ since the bigger values of mass have already started to match. For the bigger value $\lambda=10^{-1}\ell$ we can notice that more or less all the structures of the signals are matched. The biggest value $\lambda=\ell$ provides an extremely diffused solution corresponding to the residual mass (in the third component of $\delta$).
	
	In Figure \ref{fig:profiles} we show the convergence profiles of Algorithm \ref{alg:SDMM} for the different test cases here considered. In general, for the regimes of $\lambda$ of interest, few hundreds iterations are sufficient to reach convergence up to a tolerance $\varepsilon$ of $10^{-4}, 10^{-5}$. As one could expect, the case $(p,q)=(1,2)$ converges faster than the case $(p,q)=(1,1)$ thanks to the uniform convexity of the objective function, with better performance the smaller the parameter $\lambda$. For $(p,q)=(1,1)$ and excessively big values of $\lambda$, the convergence profiles are extremely oscillatory and rather slow.
	This probably depends on the fact that $\Tc_{1,1}^\lambda$ is closer to a total variation norm for these values of $\lambda$, whose minimization is more involved.
	The convergence could probably be improved via a careful tuning of the parameter $\tau$ in Algorithm \ref{alg:SDMM}. In these tests we considered $\tau=0.9$. However, we did not try to tune it since for the values of $\lambda$ of interest the algorithm already performs quite well.
	
	\begin{figure}
		\centering
		\setlength{\tabcolsep}{0pt}
		\renewcommand{\arraystretch}{0}
		\begin{tabular}{C{0.06\textwidth}C{0.22\textwidth}C{0.22\textwidth}C{0.22\textwidth}C{0.22\textwidth}C{0.06\textwidth}}
			&\includegraphics[trim={1.5cm 1.2cm 1cm 12.3cm},clip,width=0.22\textwidth]{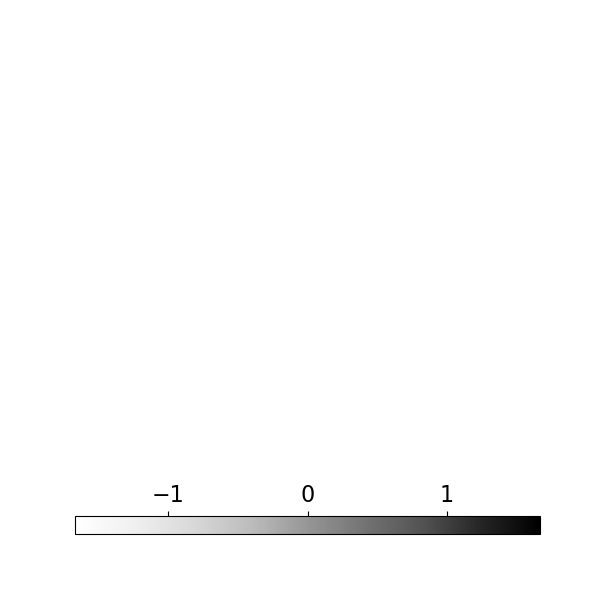} & 
			\includegraphics[trim={1.5cm 1.2cm 1cm 12.3cm},clip,width=0.22\textwidth]{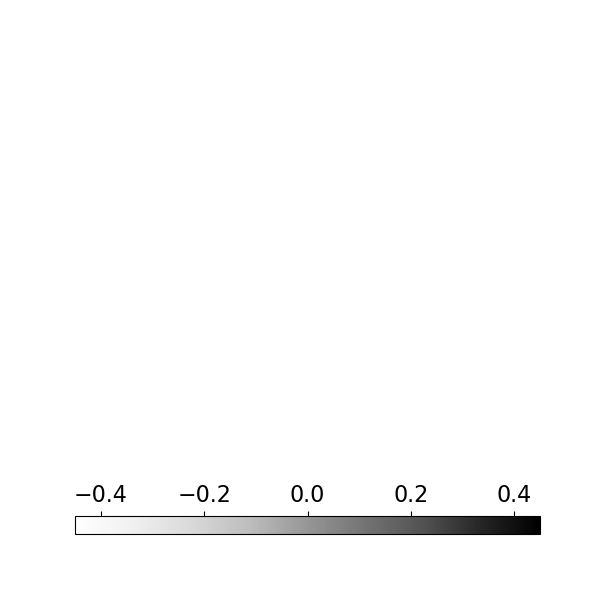} &
			\includegraphics[trim={1.5cm 1.2cm 1cm 12.3cm},clip,width=0.22\textwidth]{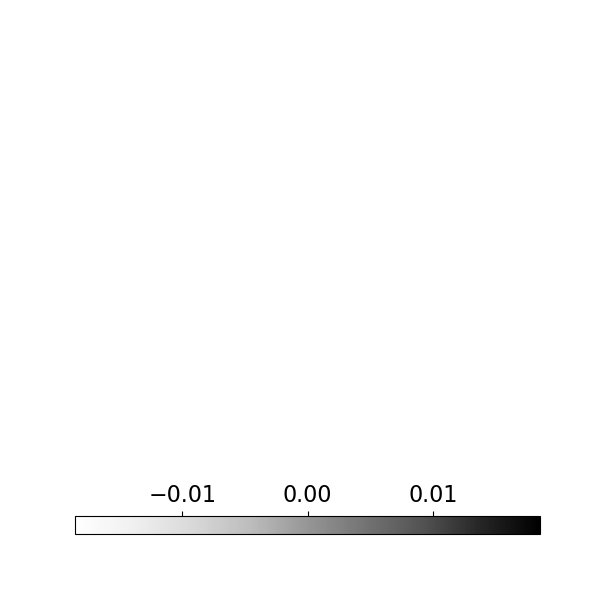} &
			\includegraphics[trim={1.5cm 1.2cm 0.94cm 12.3cm},clip,width=0.22\textwidth]{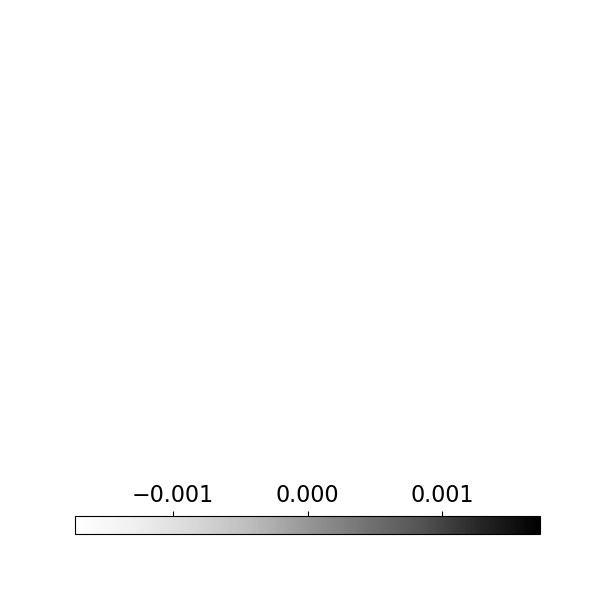} & \\
			{\small $\Tc_{1,1}^\lambda$} & \includegraphics[trim={3.5cm 1.5cm 3.05cm 1.7cm},clip,width=0.22\textwidth]{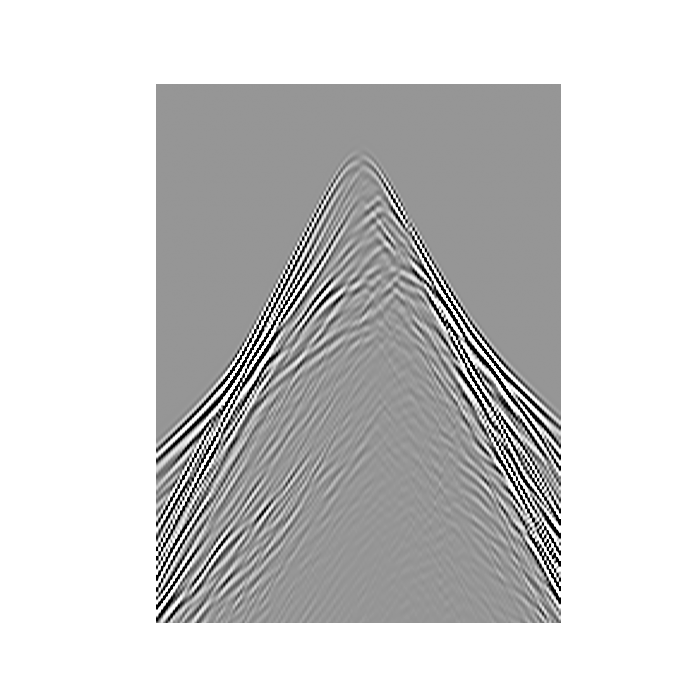} & 
			\includegraphics[trim={3.5cm 1.5cm 3.05cm 1.7cm},clip,width=0.22\textwidth]{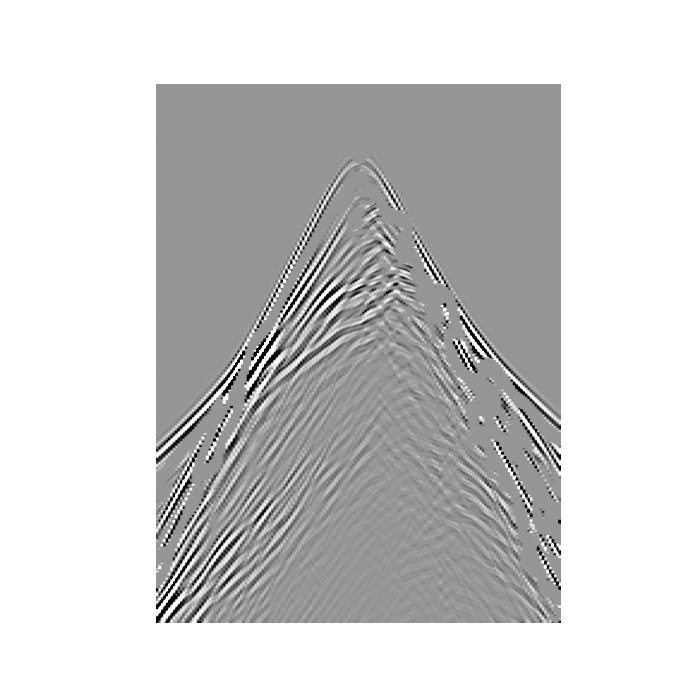} &
			\includegraphics[trim={3.5cm 1.5cm 3.05cm 1.7cm},clip,width=0.22\textwidth]{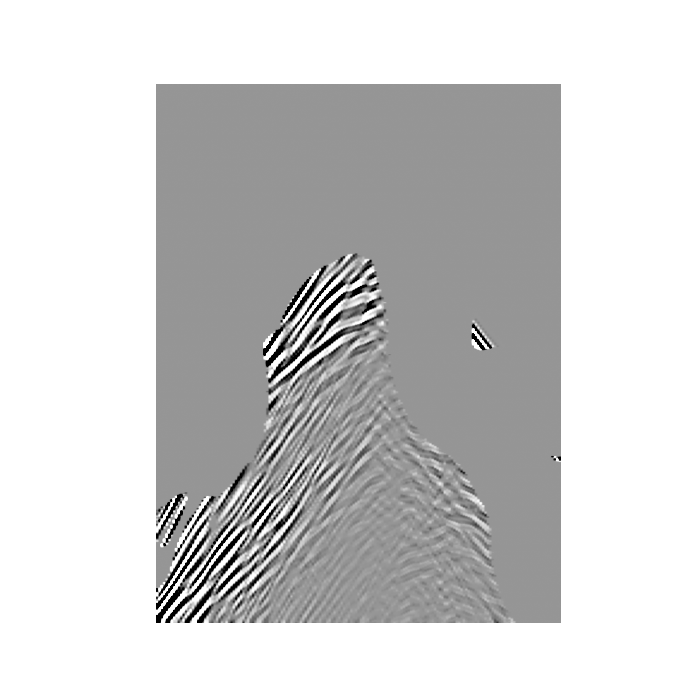} &
			\includegraphics[trim={3.5cm 1.5cm 3.05cm 1.7cm},clip,width=0.22\textwidth]{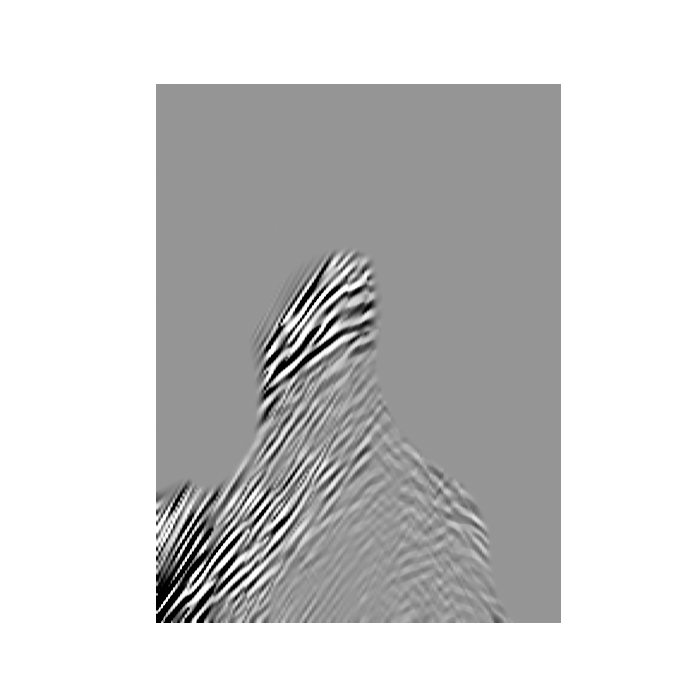} & \\
			{\small $\Tc_{1,2}^\lambda$}&\includegraphics[trim={3.5cm 1.5cm 3.05cm 1.7cm},clip,width=0.22\textwidth]{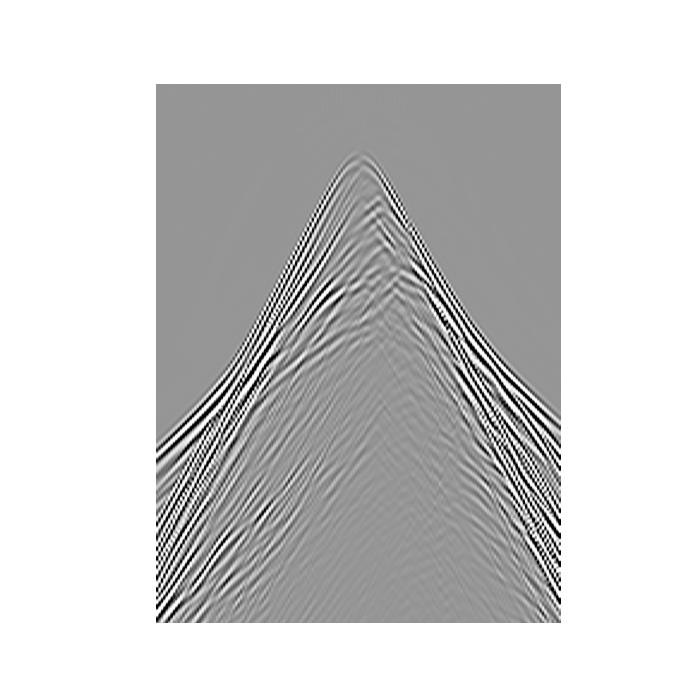} & \includegraphics[trim={3.5cm 1.5cm 3.05cm 1.7cm},clip,width=0.22\textwidth]{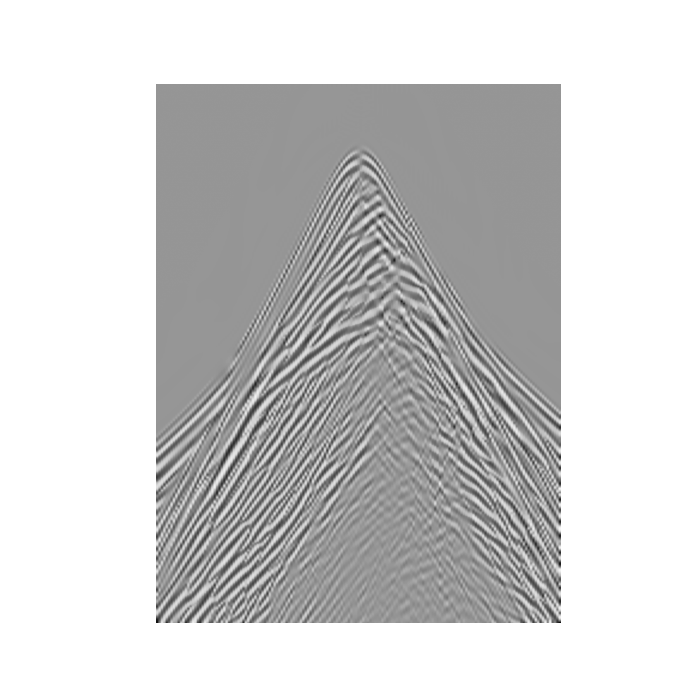} &
			\includegraphics[trim={3.5cm 1.5cm 3.05cm 1.7cm},clip,width=0.22\textwidth]{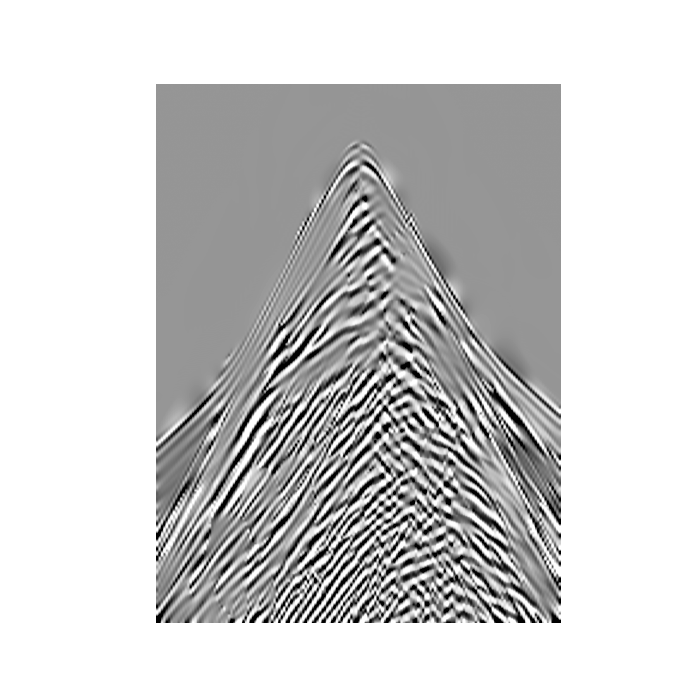} &
			\includegraphics[trim={3.5cm 1.5cm 3.05cm 1.7cm},clip,width=0.22\textwidth]{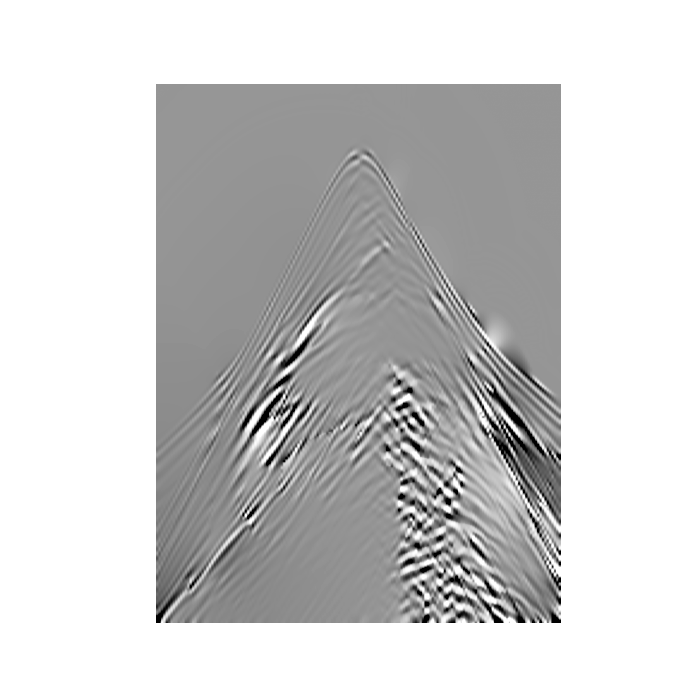} & \\
			& \vphantom{Space} & & & & \\
			& {\small $\lambda=10^{-3}\ell$} & {\small $\lambda=10^{-2}\ell$} & {\small $\lambda=10^{-1}\ell$} & {\small $\lambda=\ell$} & \\
		\end{tabular}
		\caption{Plot of the first component of $\delta$ solution to the unbalanced $L^1$ optimal transport problem between $\mu$ and $\nu$ from Figure \ref{fig:seismograms}, for different values of the parameter $\lambda$ (from first to fourth column), for $(p,q)=(1,1)$ (top row) and $(p,q)=(1,2)$ (bottom row). For the sake of the representation, data are clipped to $10\%$ of their maximum value and the color scale varies in each column.}
		\label{fig:delta_x}
	\end{figure}
	
	\begin{figure}
		\centering
		\setlength{\tabcolsep}{0pt}
		\renewcommand{\arraystretch}{0}
		\begin{tabular}{C{0.06\textwidth}C{0.22\textwidth}C{0.22\textwidth}C{0.22\textwidth}C{0.22\textwidth}C{0.06\textwidth}}
			&\includegraphics[trim={1.5cm 1.2cm 1cm 12.3cm},clip,width=0.22\textwidth]{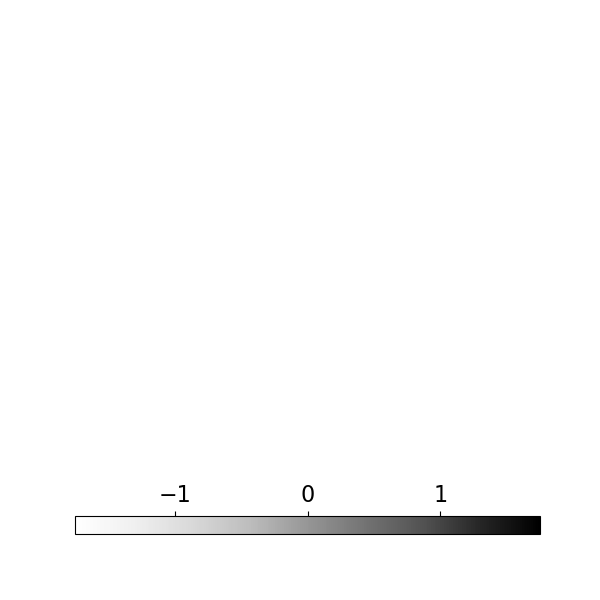} & 
			\includegraphics[trim={1.5cm 1.2cm 1cm 12.3cm},clip,width=0.22\textwidth]{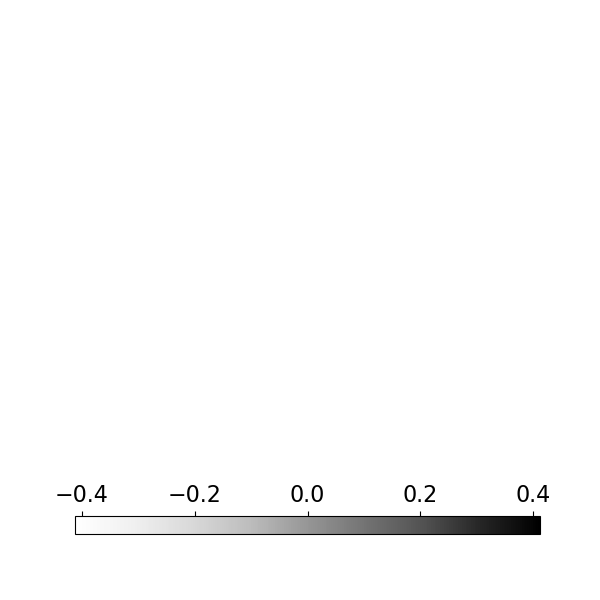} &
			\includegraphics[trim={1.5cm 1.2cm 1cm 12.3cm},clip,width=0.22\textwidth]{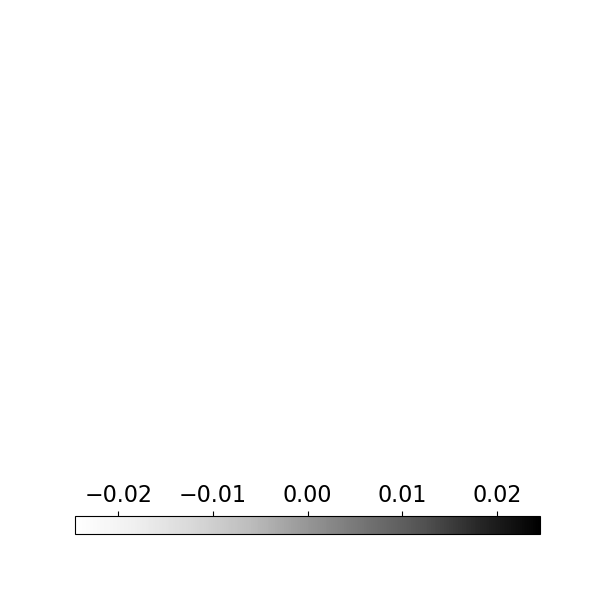} &
			\includegraphics[trim={1.5cm 1.2cm 0.94cm 12.3cm},clip,width=0.22\textwidth]{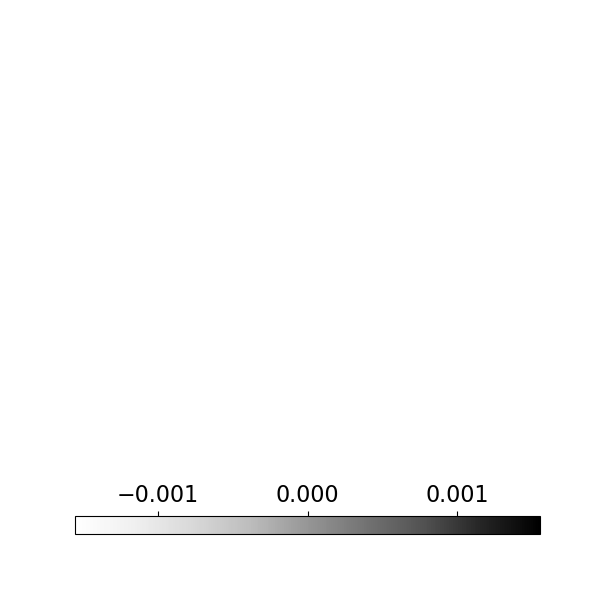} & \\
			{\small $\Tc_{1,1}^\lambda$} & \includegraphics[trim={3.5cm 1.5cm 3.05cm 1.7cm},clip,width=0.22\textwidth]{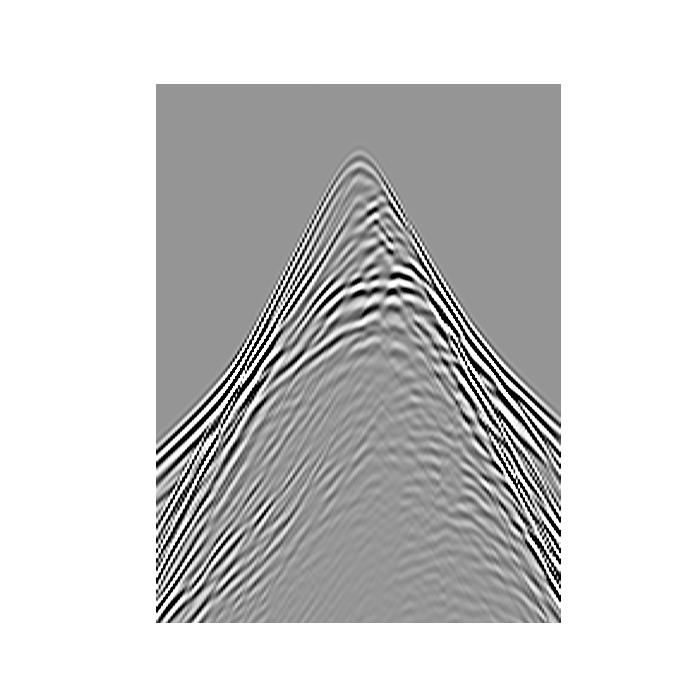} & 
			\includegraphics[trim={3.5cm 1.5cm 3.05cm 1.7cm},clip,width=0.22\textwidth]{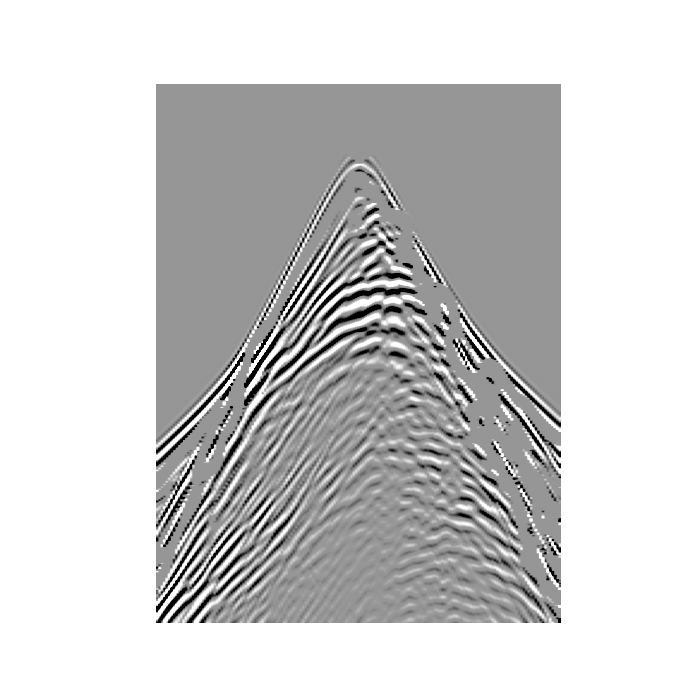} &
			\includegraphics[trim={3.5cm 1.5cm 3.05cm 1.7cm},clip,width=0.22\textwidth]{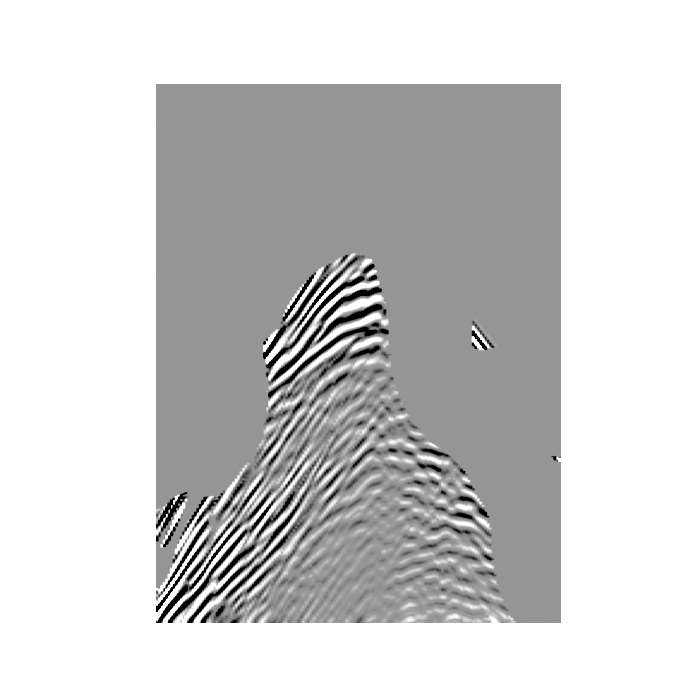} &
			\includegraphics[trim={3.5cm 1.5cm 3.05cm 1.7cm},clip,width=0.22\textwidth]{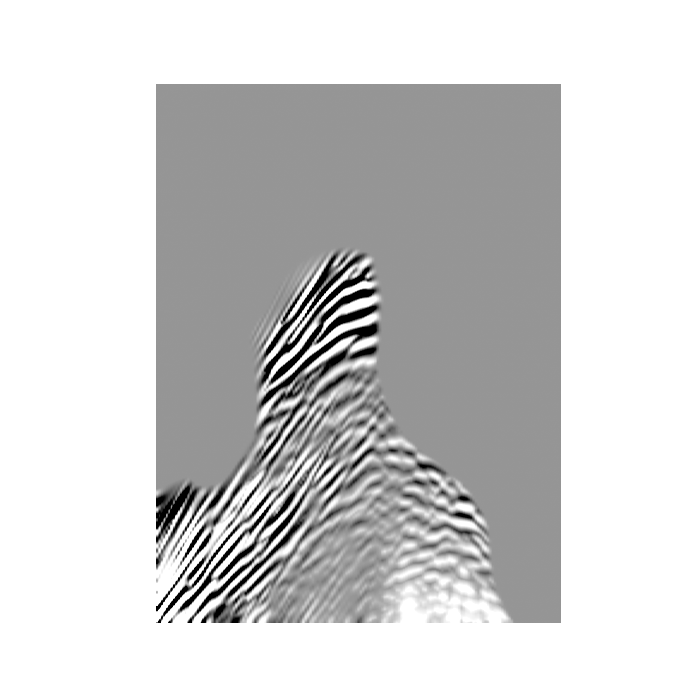} & \\
			{\small $\Tc_{1,2}^\lambda$}&\includegraphics[trim={3.5cm 1.5cm 3.05cm 1.7cm},clip,width=0.22\textwidth]{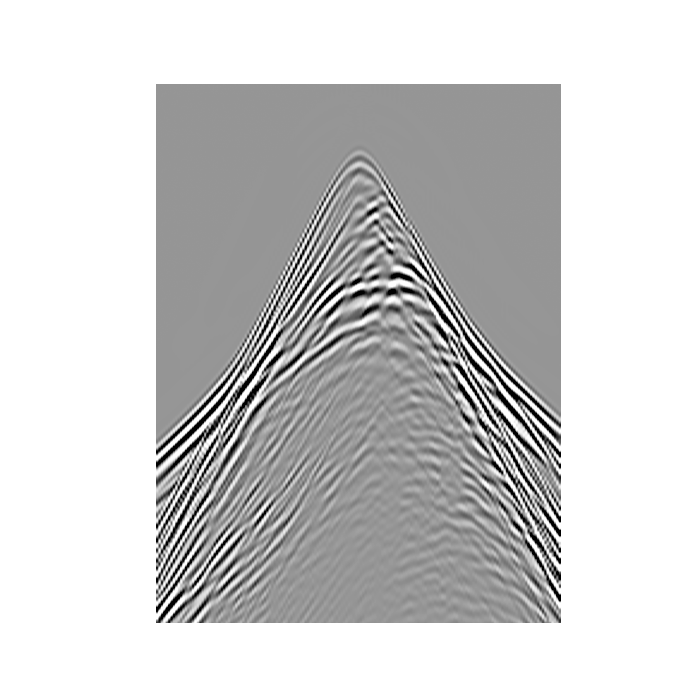} & \includegraphics[trim={3.5cm 1.5cm 3.05cm 1.7cm},clip,width=0.22\textwidth]{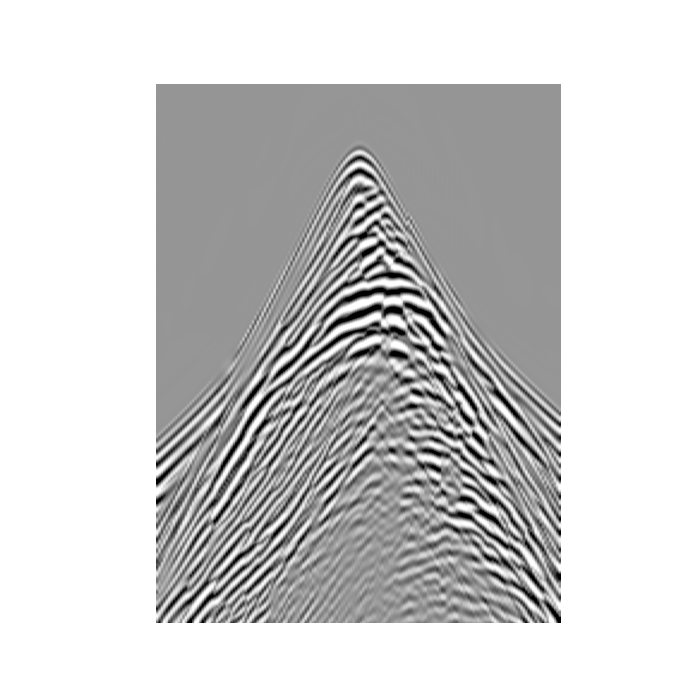} &
			\includegraphics[trim={3.5cm 1.5cm 3.05cm 1.7cm},clip,width=0.22\textwidth]{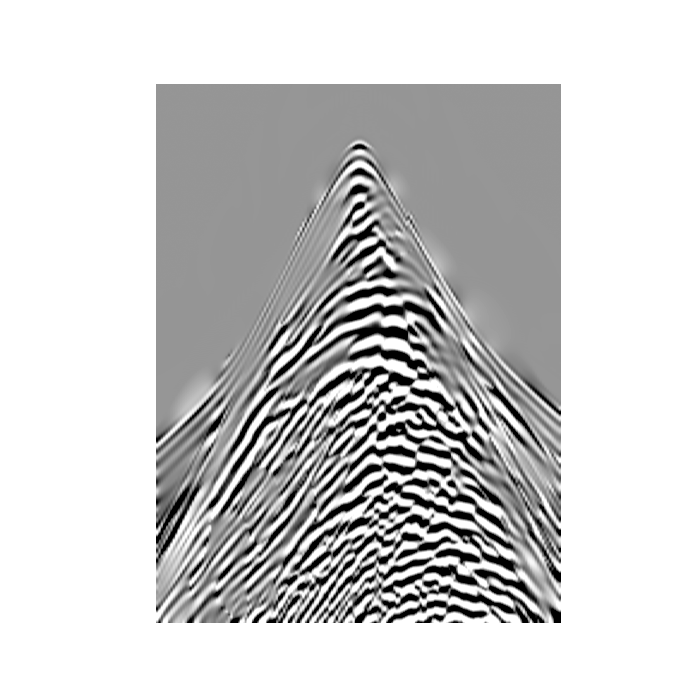} &
			\includegraphics[trim={3.5cm 1.5cm 3.05cm 1.7cm},clip,width=0.22\textwidth]{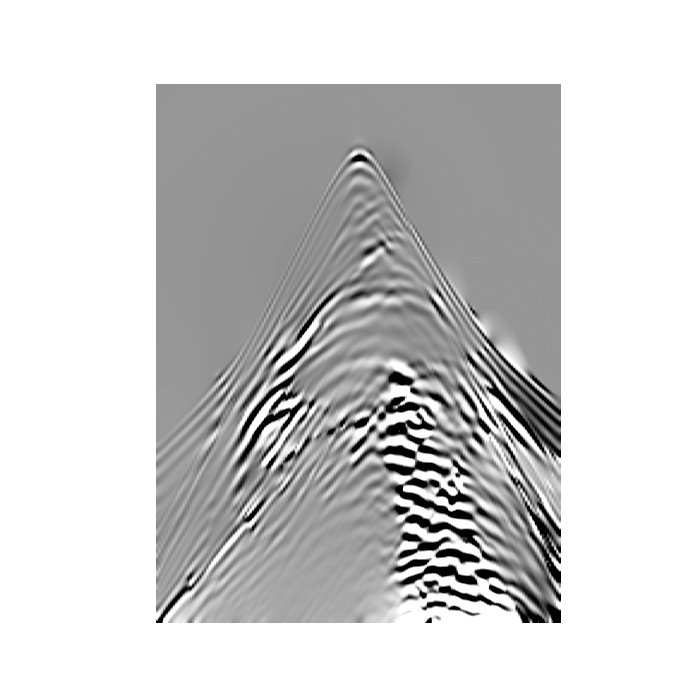} & \\
			& \vphantom{Space} & & & & \\
			& {\small $\lambda=10^{-3}\ell$} & {\small $\lambda=10^{-2}\ell$} & {\small $\lambda=10^{-1}\ell$} & {\small $\lambda=\ell$} & \\
		\end{tabular}
		\caption{Plot of the second component of $\delta$ solution to the unbalanced $L^1$ optimal transport problem between $\mu$ and $\nu$ from Figure \ref{fig:seismograms}, for different values of the parameter $\lambda$ (from first to fourth column), for $(p,q)=(1,1)$ (top row) and $(p,q)=(1,2)$ (bottom row). For the sake of the representation, data are clipped to $10\%$ of their maximum value and the color scale varies in each column.}
		\label{fig:delta_z}
	\end{figure}
	
	\begin{figure}
		\centering
		\setlength{\tabcolsep}{0pt}
		\renewcommand{\arraystretch}{0}
		\begin{tabular}{C{0.06\textwidth}C{0.22\textwidth}C{0.22\textwidth}C{0.22\textwidth}C{0.22\textwidth}C{0.06\textwidth}}
			&\includegraphics[trim={1.5cm 1.2cm 1cm 12.3cm},clip,width=0.22\textwidth]{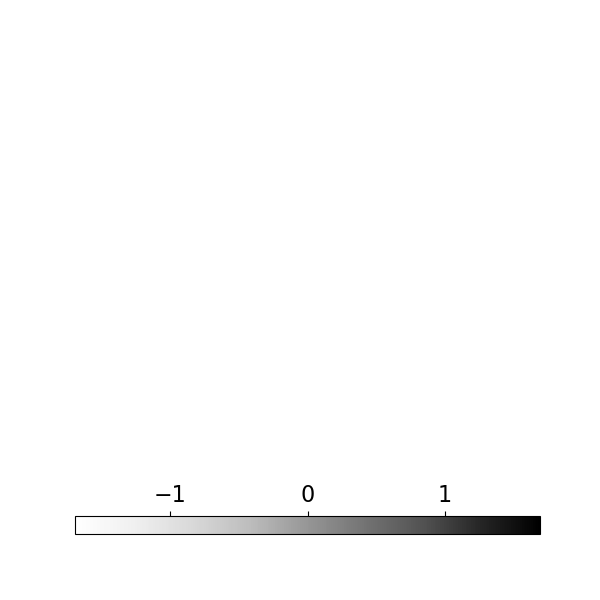} & 
			\includegraphics[trim={1.5cm 1.2cm 1cm 12.3cm},clip,width=0.22\textwidth]{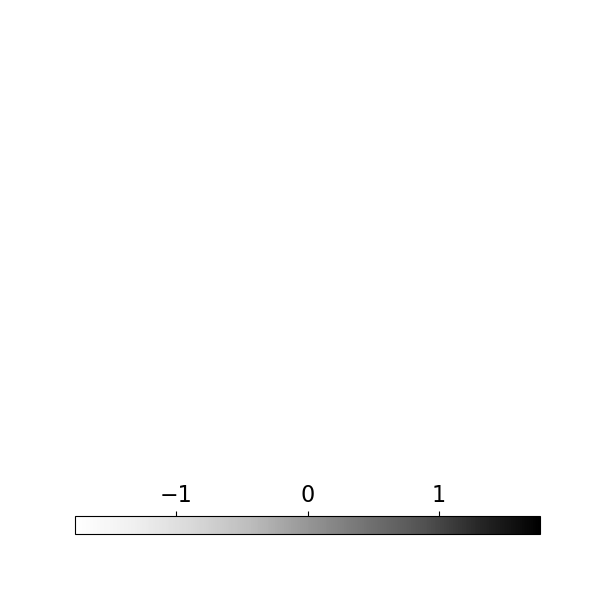} &
			\includegraphics[trim={1.5cm 1.2cm 1cm 12.3cm},clip,width=0.22\textwidth]{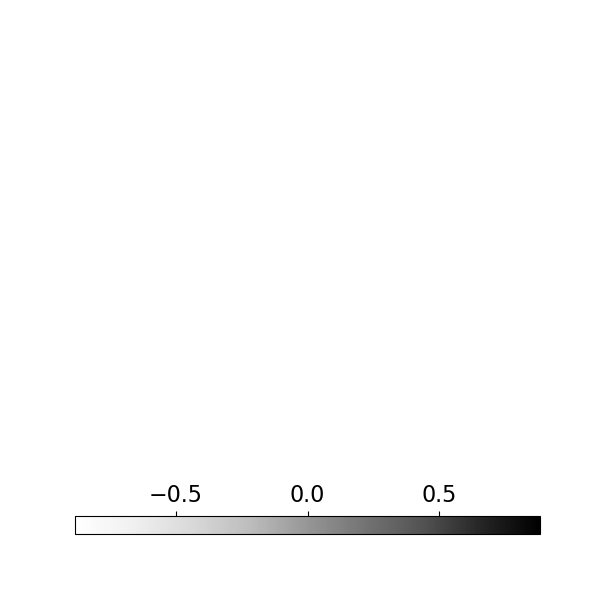} &
			\includegraphics[trim={1.5cm 1.2cm 0.94cm 12.3cm},clip,width=0.22\textwidth]{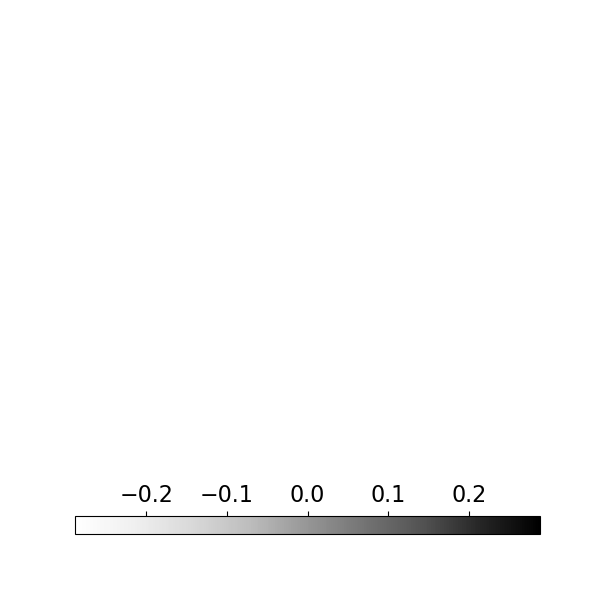} & \\
			{\small $\Tc_{1,1}^\lambda$} & \includegraphics[trim={3.5cm 1.5cm 3.05cm 1.7cm},clip,width=0.22\textwidth]{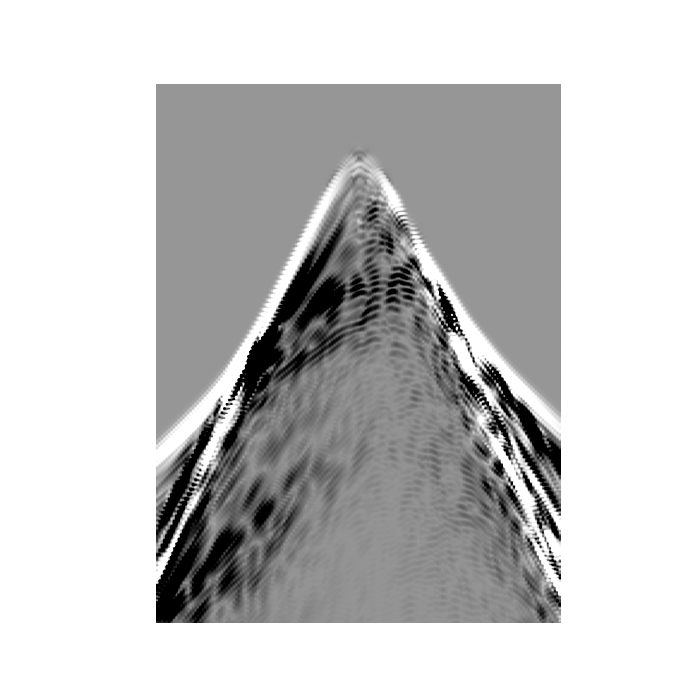} & 
			\includegraphics[trim={3.5cm 1.5cm 3.05cm 1.7cm},clip,width=0.22\textwidth]{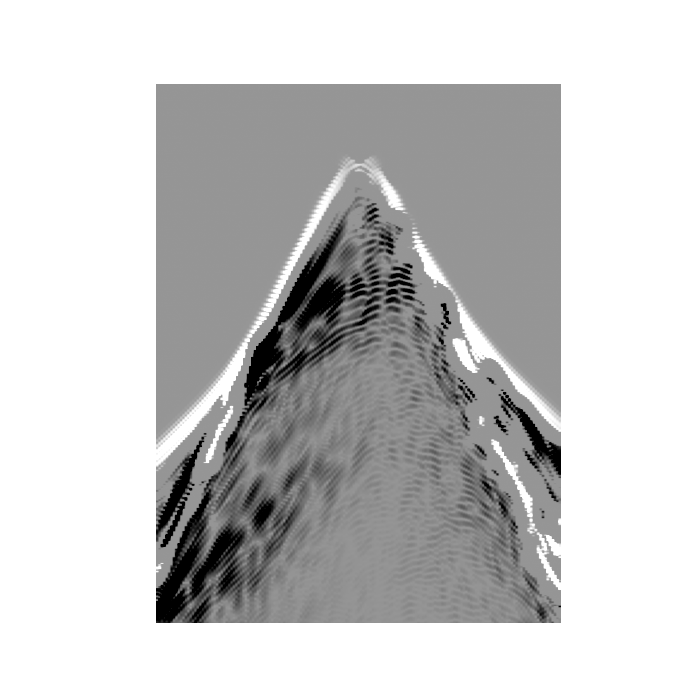} &
			\includegraphics[trim={3.5cm 1.5cm 3.05cm 1.7cm},clip,width=0.22\textwidth]{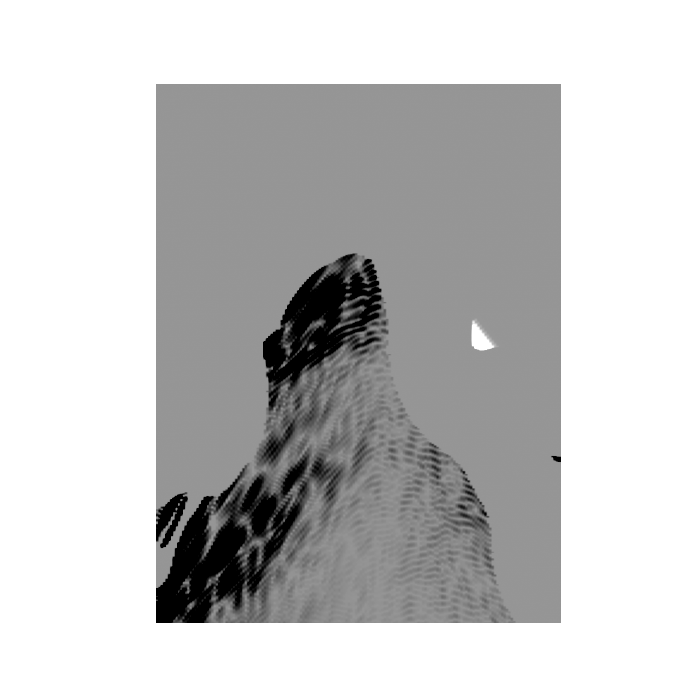} &
			\includegraphics[trim={3.5cm 1.5cm 3.05cm 1.7cm},clip,width=0.22\textwidth]{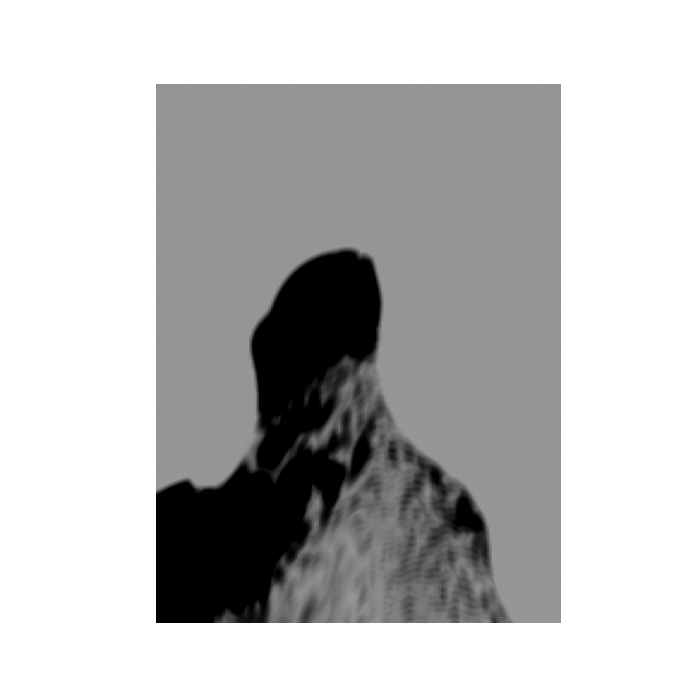} & \\
			{\small $\Tc_{1,2}^\lambda$}&\includegraphics[trim={3.5cm 1.5cm 3.05cm 1.7cm},clip,width=0.22\textwidth]{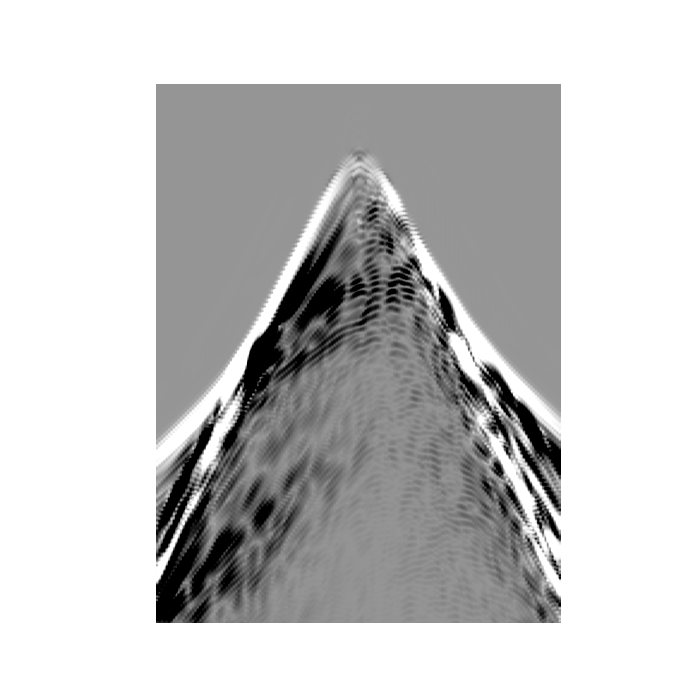} & \includegraphics[trim={3.5cm 1.5cm 3.05cm 1.7cm},clip,width=0.22\textwidth]{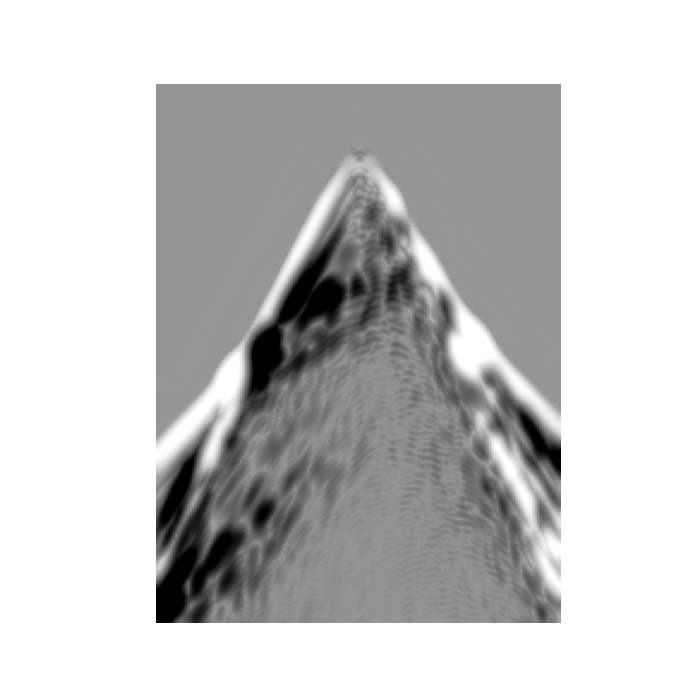} &
			\includegraphics[trim={3.5cm 1.5cm 3.05cm 1.7cm},clip,width=0.22\textwidth]{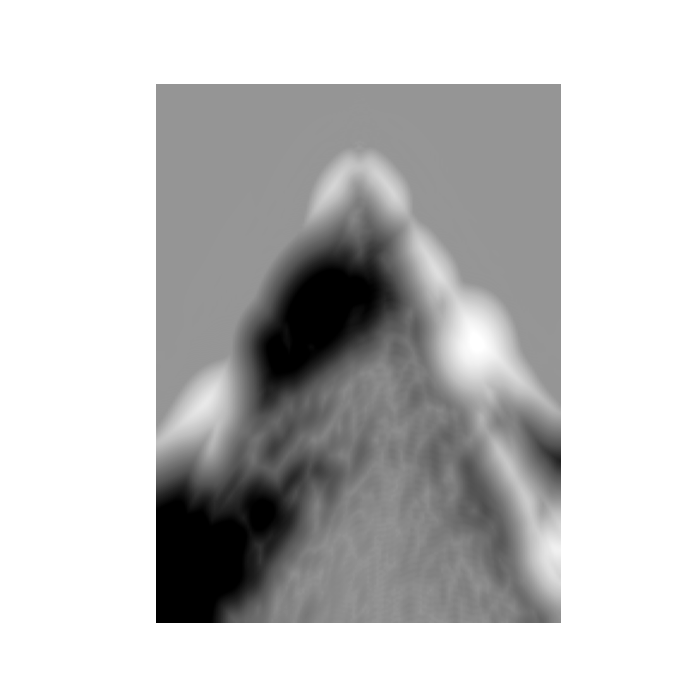} &
			\includegraphics[trim={3.5cm 1.5cm 3.05cm 1.7cm},clip,width=0.22\textwidth]{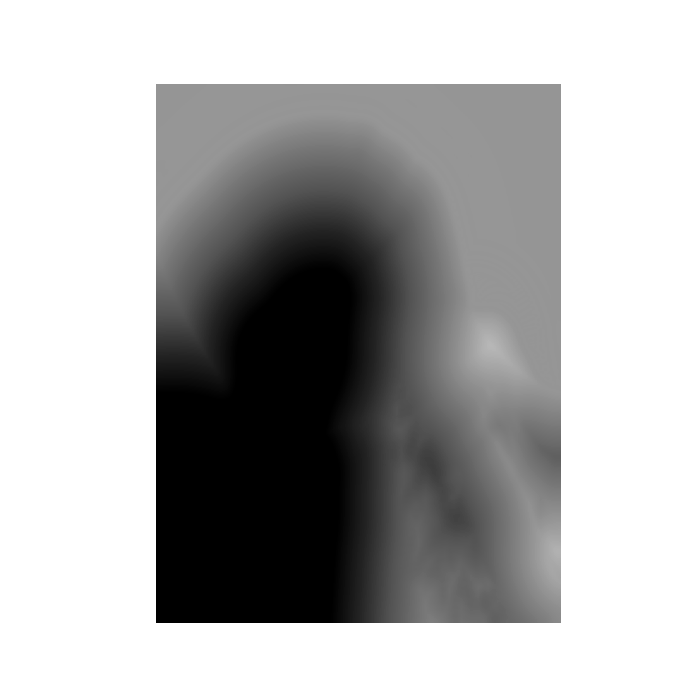} & \\
			& \vphantom{Space} & & & & \\
			& {\small $\lambda=10^{-3}\ell$} & {\small $\lambda=10^{-2}\ell$} & {\small $\lambda=10^{-1}\ell$} & {\small $\lambda=\ell$} & \\
		\end{tabular}
		\caption{Plot of the third component of $\delta$ solution to the unbalanced $L^1$ optimal transport problem between $\mu$ and $\nu$ from Figure \ref{fig:seismograms}, for different values of the parameter $\lambda$ (from first to fourth column), for $(p,q)=(1,1)$ (top row) and $(p,q)=(1,2)$ (bottom row). For the sake of the representation, data are clipped to $10\%$ of their maximum value and the color scale varies in each column.}
		\label{fig:delta_alpha}
	\end{figure}
	
	\begin{figure}
		\centering
		\subfloat[$\Tc_{1,1}^\lambda$]{\includegraphics[trim={0.5cm 0.3cm 1.2cm 1.1cm},clip,width=0.36\textwidth]{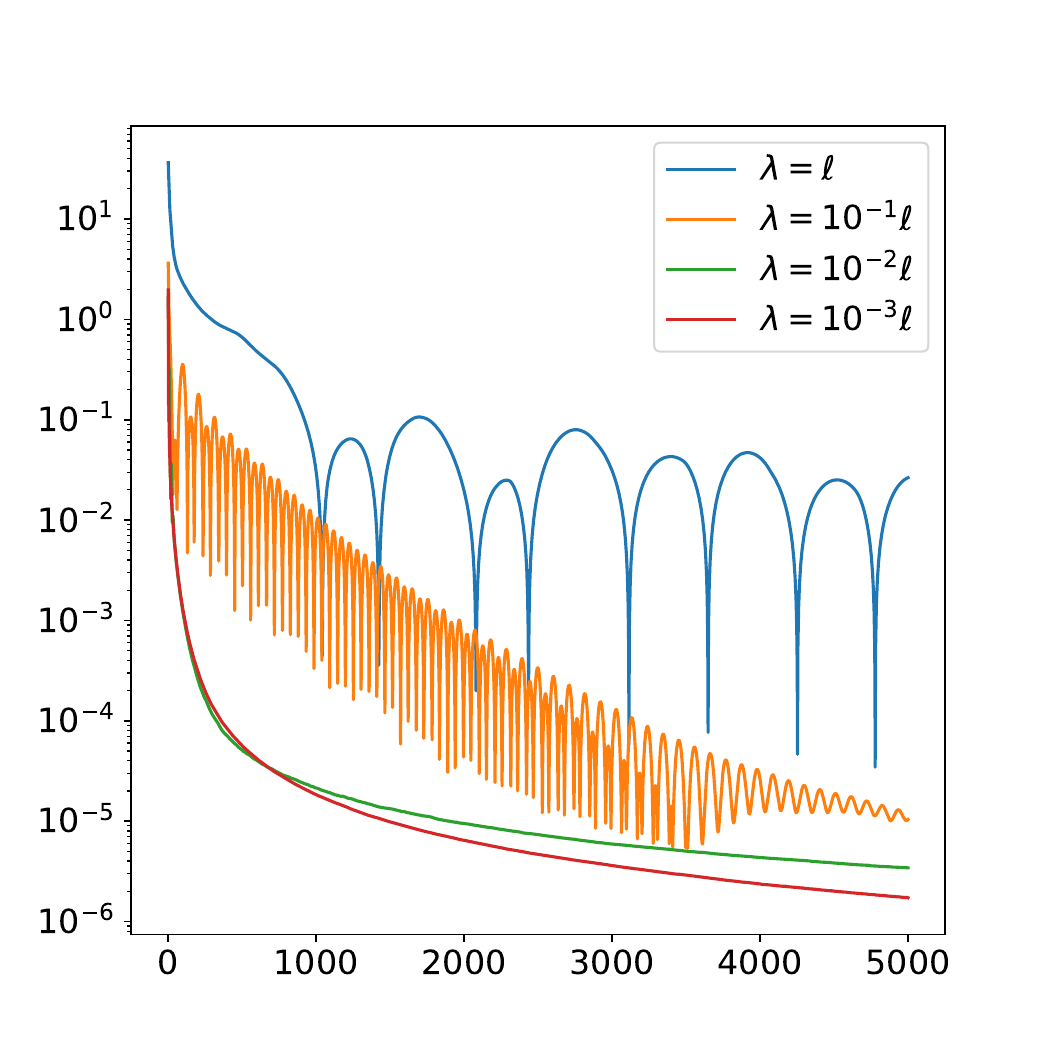}} \qquad
		\subfloat[$\Tc_{1,2}^\lambda$]{\includegraphics[trim={0.5cm 0.3cm 1.2cm 1.1cm},clip,width=0.36\textwidth]{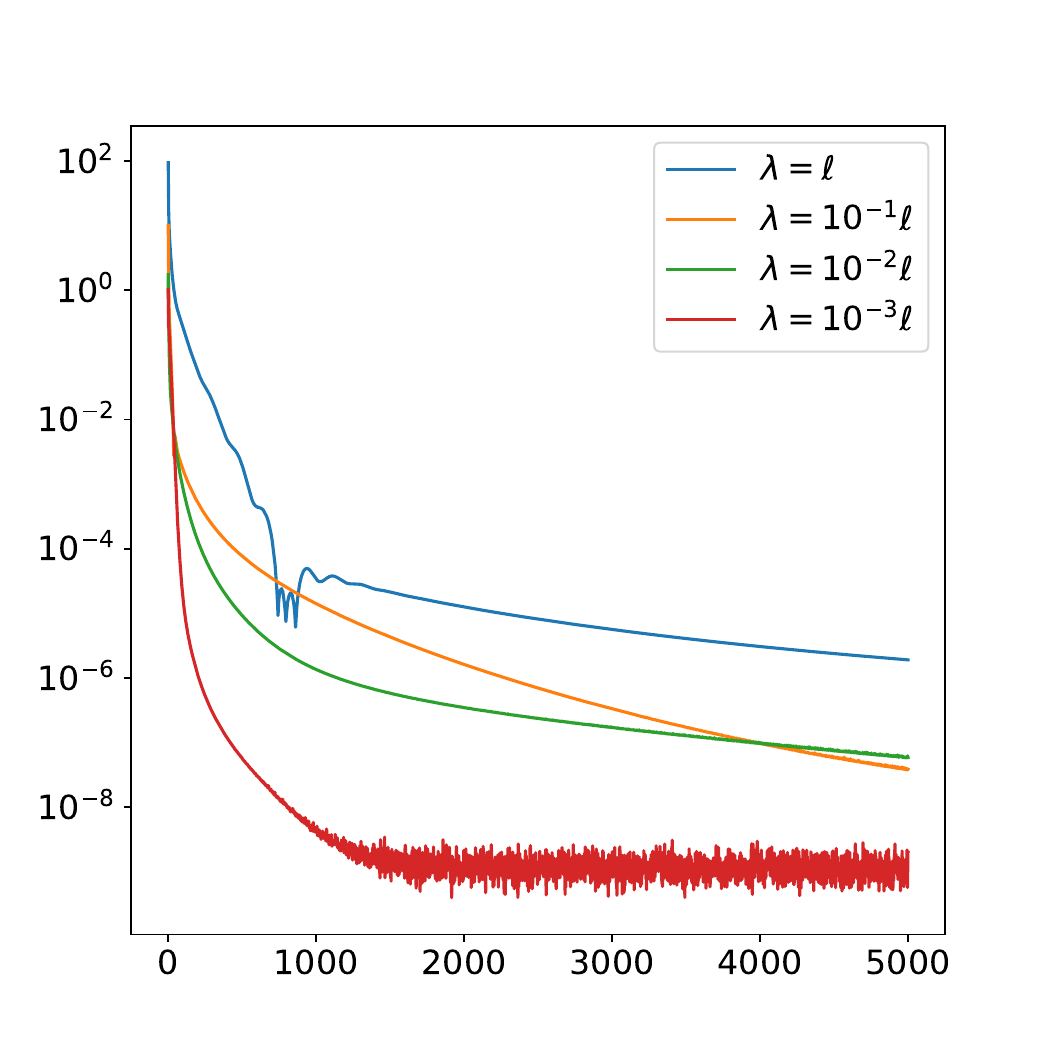}}
		\caption{Convergence profiles of algorithm \ref{alg:SDMM} for the error $\Delta^n+r^n$ in Remark \ref{rmk:stopping}, for the different test cases of Section \ref{ssec:seismo}.}
		\label{fig:profiles}
	\end{figure}
	
	\subsection{Application to a realistic data set}
	
	To validate our approach, we considered a realistic two-dimensional inversion test: the Marmousi case study \cite{martin2006marmousi2}. The setting is the same one as in Section \ref{ssec:seismo}. We consider in this case $N_s=124$ point sources, uniformly distributed along the $x$ axis at the same depth as the receivers, all producing the same Ricker signal. The true velocity model $c:D\rightarrow \R_+$ we aim at reconstructing is shown in Figure \ref{fig:models}.
	
	The tests have been performed using the code \textit{TOYxDAC\_time} \cite{yang2018time} of the \textit{SEISCOPE project}\footnote{https://seiscope2.osug.fr/} where the new misfit functions proposed in this article have been integrated. The acoustic wave equation is discretized using a finite difference approach. The optimization problem \eqref{eq:FWI} is solved with a gradient based l-BFGS approach \cite{nocedal1980updating}. The computation of the gradient is obtained via the \textit{adjoint state method} \cite{plessix2006review}. We considered a tolerance $\varepsilon=10^{-2}$ in Algorithm \ref{alg:SDMM}, according to the stopping criterion presented in Remark \ref{rmk:stopping}. This relatively big value allows us to solve each optimization problem in a small number of iterations and to limit in this way the computational cost for evaluating the misfit function. We have not experienced a sensitivity of the results of the test to considering smaller values of $\varepsilon$.
	
	\subsubsection{Objective gradient via adjoint state method}
	
	We present the method restricting to one source only for simplicity, $f$. We rewrite for the sake of notation the total cost as
	\begin{equation}\label{eq:obj_fwi}
		\Tc_{p,q}^\lambda(L(Rv_x[c],Rv_z[c]),L(v_x^{obs},v_z^{obs}))=g(v_x[c],v_z[c]) \,.
	\end{equation}
	The relation between $c$ and $v_x,v_z$ is given as a forward linear equation and can be seen as the constraint $F(c,v_x,v_z,p)=0$, where
	\begin{equation}\label{eq:sys_forward}
		\begin{cases}
			F_{v_x}(c,v_x,v_z,p)= \partial_t v_x-\frac{1}{\rho}\partial_x p=0 \,,\\
			F_{v_z}(c,v_x,v_z,p)=\partial_t v_z-\frac{1}{\rho}\partial_z p = 0 \,,\\
			F_p(c,v_x,v_z,p)=\partial_t p-\rho c^2 \div(v) -f= 0\,.\\
		\end{cases}
	\end{equation}
	Considering then the velocities $v_x, v_z$ and the pressure $p$ as variables in the optimization problem, we can introduce the Lagrangian functional
	\begin{equation}\label{eq:lagrangian_fwi}
		\mathcal{L}(c,v_x,v_z,p,\lambda)=g(v_x[c],v_z[c]) + \langle F(c,v_x,v_z,p),\lambda\rangle_D \,,
	\end{equation}
	where $\lambda=(\lambda_{v_x},\lambda_{v_z},\lambda_p)$ is the Lagrangian multiplier associated with the constraints and $\langle \cdot, \cdot \rangle_D$ is an appropriate scalar product in the model space. If we denote $v_x[c],v_z[c], p[c]$ the solution to the equation $F(c,v_x,v_z,p)=0$ for a given velocity model $c$, then
	\[
	\mathcal{L}(c,v_x[c],v_z[c],p[c],\lambda)=g(v_x[c],v_z[c])
	\]
	and therefore
	\[
	\begin{aligned}
		\frac{\d g(v_x[c],v_z[c]) }{\d c} =& \partial_{c} \mathcal{L}(c,v_x[c],v_z[c],p[c],\lambda)+\partial_{v_x}\mathcal{L}(c,v_x[c],v_z[c],p[c],\lambda) \frac{\d v_x[c]}{\d c} \\
		&+\partial_{v_z} \mathcal{L}(c,v_x[c],v_z[c],p[c],\lambda)\frac{\d v_z[c]}{\d c} +\partial_{p} \mathcal{L}(c,v_x[c],v_z[c],p[c],\lambda) \frac{\d p[c]}{\d c}\,.
	\end{aligned}
	\]
	The adjoint state method consists in computing a specific multiplier $\bar{\lambda}[c]$ such that the partial derivatives with respect to $v_x, v_z$ and $p$ of the Lagrangian function vanish. In this way one obtains the derivative of the objective function simply as 
	\[
	\frac{\d g(v_x[c],v_z[c]) }{\d c} = \partial_{c} \mathcal{L}(c,v_x[c],v_z[c],p[c],\bar{\lambda})= \partial_{c} \langle F(c,v_x,v_z,p), \bar{\lambda} \rangle_D \,,
	\]
	avoiding the expensive computation of the terms $\frac{\d v_x[c]}{\d c},\frac{\d v_z[c]}{\d c}$ and $\frac{\d p[c]}{\d c}$. From the definition of the Lagrangian function \eqref{eq:lagrangian_fwi} and by the definition of the operator $F(c,v_x,v_z,p)$, $\bar{\lambda}$ can be computed as solution to:
	\begin{equation}\label{eq:sys_adjoint}
		\left\{
		\begin{aligned}
			& \partial_{v_x}\langle F^T(c,v_x,v_z,p), \bar{\lambda}\rangle_D = - \partial_t \bar{\lambda}_{v_x} + c^2\rho \partial_x \bar{\lambda}_p  = -\frac{\partial g(v_x[c],v_z[c]) }{\partial v_x}  \\
			&\partial_{v_z}\langle F^T(c,v_x,v_z,p), \bar{\lambda}\rangle_D = - \partial_t \bar{\lambda}_{v_z} + c^2\rho \partial_z \bar{\lambda}_p = -\frac{\partial g(v_x[c],v_z[c]) }{\partial v_z}  \\
			&\partial_{p}\langle F^T(c,v_x,v_z,p), \bar{\lambda}\rangle_D = -\partial_t \bar{\lambda}_p + \frac{1}{\rho} (\partial_x \bar{\lambda}_{v_x} + \partial_z \bar{\lambda}_{v_z}) = 0\,.
		\end{aligned}
		\right.
	\end{equation}
	System \eqref{eq:sys_adjoint} is the system of adjoint equations with respect to \eqref{eq:sys_forward} and corresponds to a backward acoustic wave equation.
	
	From \eqref{eq:obj_fwi} and our definition of the cost function, taking advantage of the dual formulation (\eqref{eq:KRvect}, \eqref{eq:dualvect_unbalq2} and \eqref{eq:H-1} for the three different cases considered) which implies that $\Tc_{p,q}^{\lambda}$ is a supremum of linear functionals in the measures, hence convex and almost everywhere differentiable, we can write
	\begin{equation}\label{eq:sys_adjoint_sources}
	\begin{gathered}
		\frac{\partial \Tc_{p,q}^{\lambda}(L(Rv_x[c],Rv_z[c]),L(v_x^{obs},v_z^{obs})) }{\partial v_x}=R^T \langle \partial_{v_x} L(Rv_x[c],Rv_z[c])) ,\bar{\phi} \rangle_{\V}\,, \\
		\frac{\partial \Tc_{p,q}^{\lambda}(L(Rv_x[c],Rv_z[c]),L(v_x^{obs},v_z^{obs})) }{\partial v_z}=R^T \langle \partial_{v_z} L(Rv_x[c],Rv_z[c]))^T, \bar{\phi} \rangle_{\V}\,,
	\end{gathered}
	\end{equation}
	where $\bar{\phi}$ is an (the) optimal potential for the dual formulation.
	
	The computation of the gradient of the objective function in \eqref{eq:FWI} requires solving each time the backward acoustic wave equation \eqref{eq:sys_adjoint} with sources \eqref{eq:sys_adjoint_sources}. The forward solution and the optimal potential are already provided by the evaluation of the objective function. 
	
	\subsubsection{Results}
	 
	We solve the problem for different values of the parameter $(p,q)$ in the misfit $\Tc_{p,q}^\lambda$ and also compare with other misfit functions.
	We solve the inverse problem testing several initial models (Figure \ref{fig:models}) of increasing difficulty, that is further and further away from the optimum we want to reconstruct. The objective is to measure the robustness of the inverse problem, that is the possibility to retrieve the minimum with less and less initial information. This is related to the convexity of the misfit function around the minimum. We do not consider in the comparison the computational cost related to the different misfit. The parameter $\lambda$ is tuned manually to obtain the best performance, taking into account the considerations above.
	
	\begin{figure}
		\centering
		\setlength{\tabcolsep}{5pt}
		\renewcommand{\arraystretch}{1.3}
		\begin{tabular}{C{0.33\textwidth}C{0.33\textwidth}}
			{\small True model}& {\small Initial model 1} \\
			\includegraphics[width=0.33\textwidth]{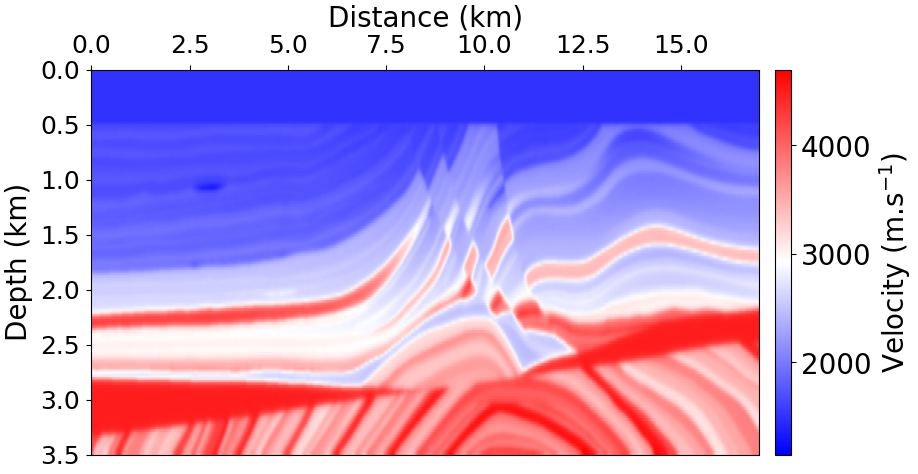} &
			\includegraphics[width=0.33\textwidth]{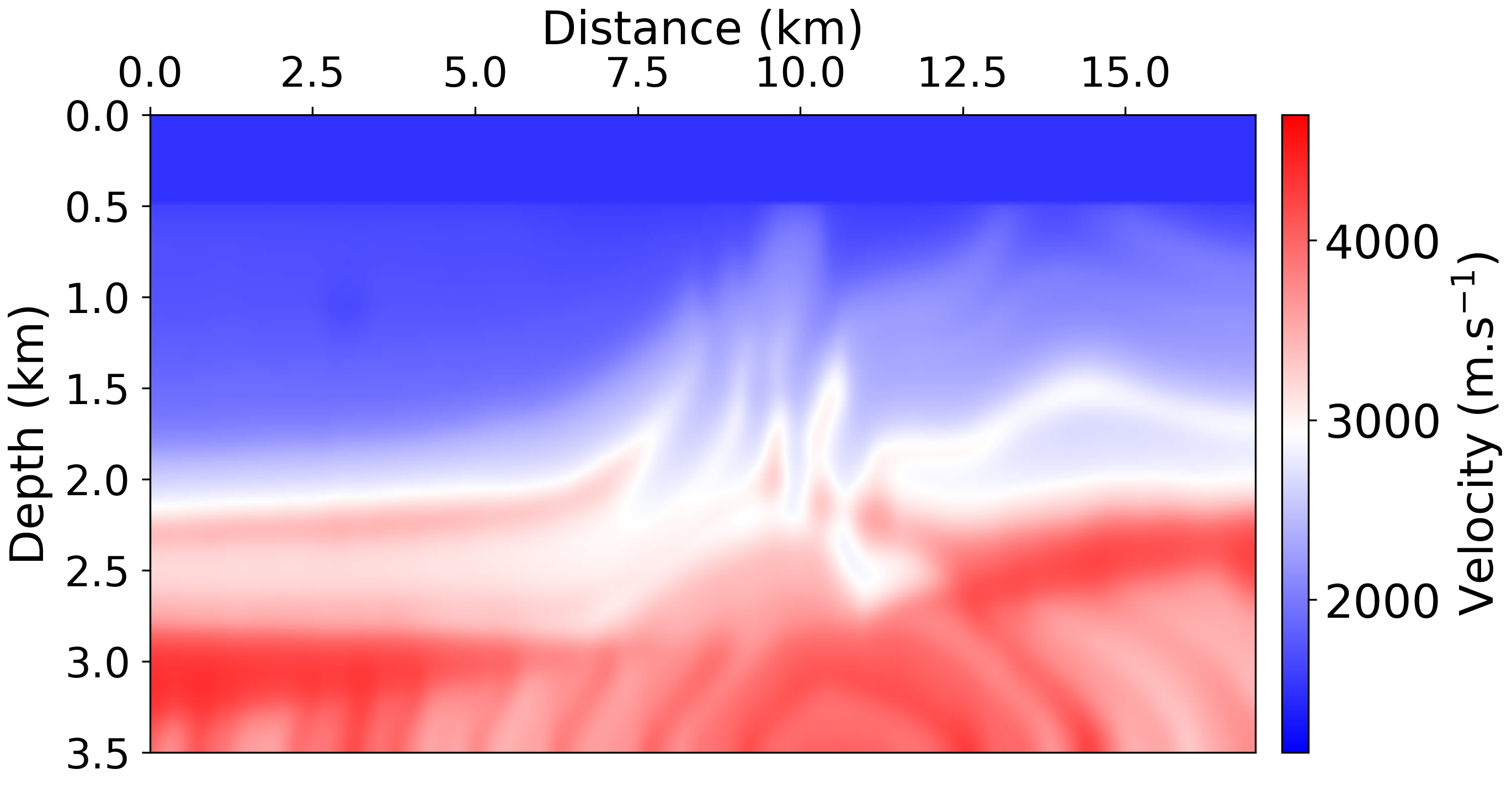} \\
			{\small Initial model 2} & {\small Initial model 3} \\
			\includegraphics[width=0.33\textwidth]{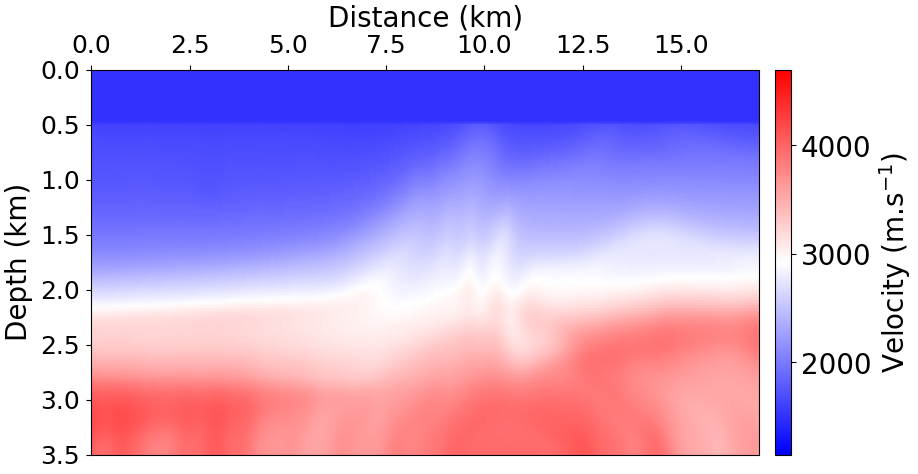}
			&
			\includegraphics[width=0.33\textwidth]{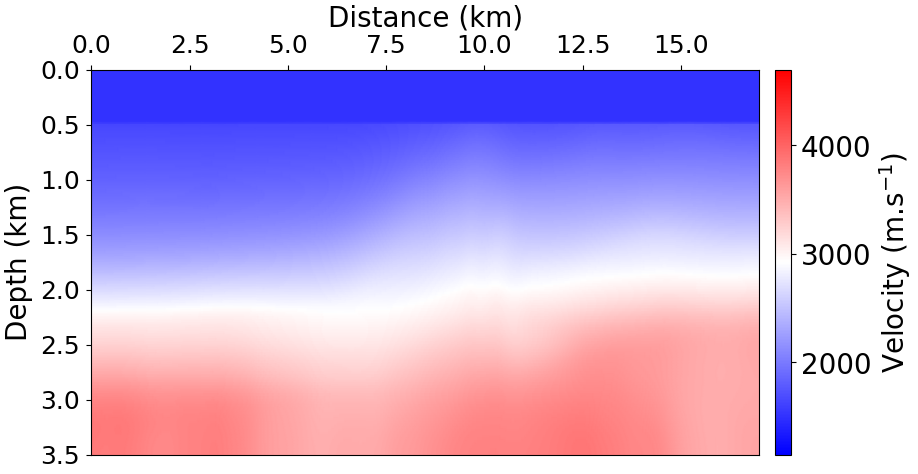}
		\end{tabular}
		\caption{True velocity model of the Marmousi test case and several initial conditions for the inversion problem.}
		\label{fig:models}
	\end{figure}
	
	In Figure \ref{fig:marmousi} we show the results obtained for the three sets of parameters $(p,q)=(1,1),(1,2)$ and $(2,2)$ for $\Tc_{p,q}^\lambda$. These results can be compared further with two other misfit functions: the classical $L^2$ distance and the $\KR$ norm \eqref{eq:KRscal}. These latter are applied componentwise on the $v_x$ and $v_z$ signals. Figure \ref{fig:marmousi} shows the final reconstructed velocity model $c$ obtained from the (local) minimization of the misfit function, starting from several initial conditions (Figure \ref{fig:models}, \textit{initial models} from one to three) obtained as smoothed versions of the true velocity model we aim at recovering (Figure \ref{fig:models}, \textit{true model}).
	
	\begin{figure}
		\centering
		\setlength{\tabcolsep}{0pt}
		\renewcommand{\arraystretch}{2.15}
		\begin{tabular}{C{0.05\textwidth}C{0.3\textwidth}C{0.3\textwidth}C{0.3\textwidth}C{0.05\textwidth}}
			&{\small Initial model 1} & {\small Initial model 2} & {\small Initial model 3} & \\
			{\small $L^2$}
			&\includegraphics[width=0.28\textwidth]{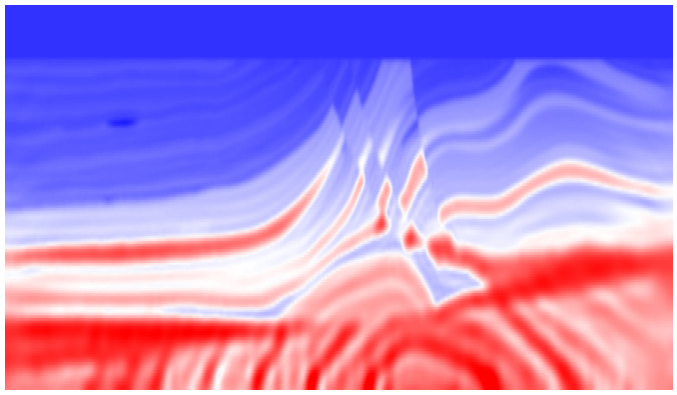} &
			\includegraphics[width=0.28\textwidth]{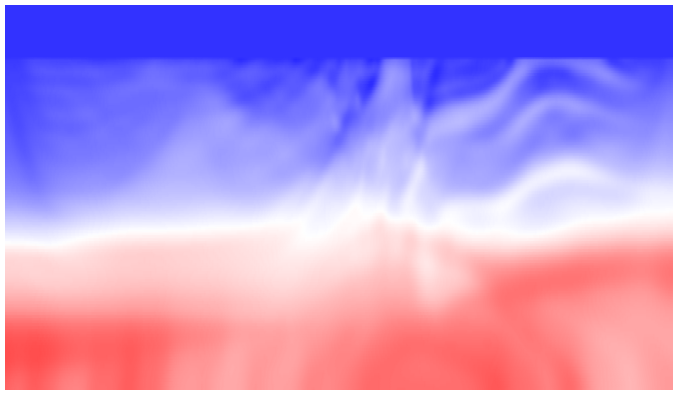} & \includegraphics[width=0.28\textwidth]{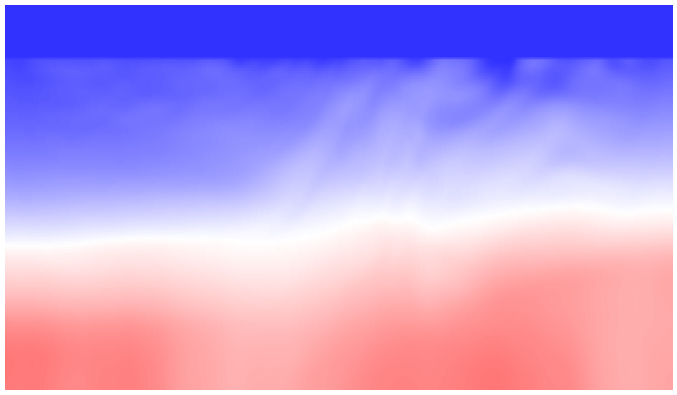} & \\
			{\small $\KR$}
			&\includegraphics[width=0.28\textwidth]{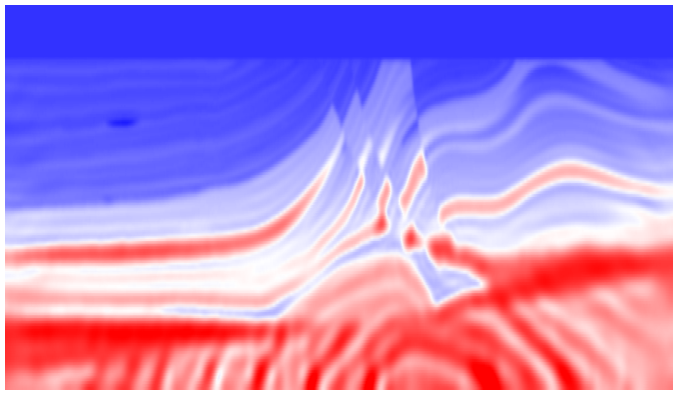} &
			\includegraphics[width=0.28\textwidth]{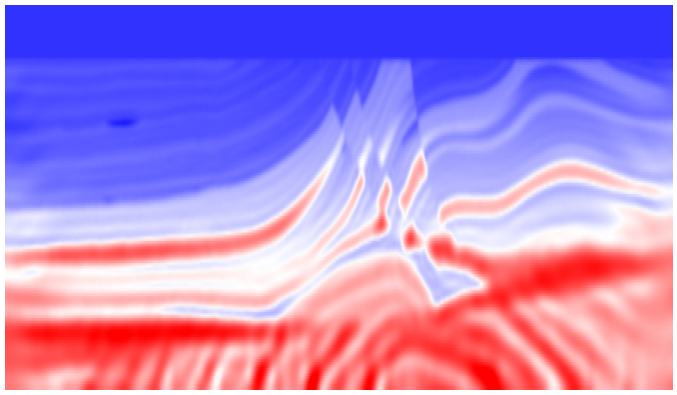} & \includegraphics[width=0.28\textwidth]{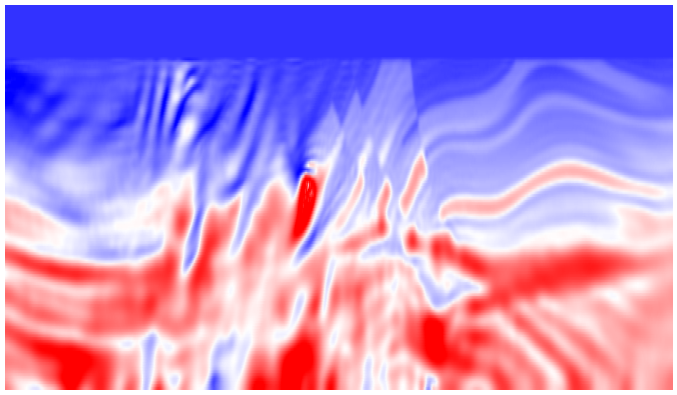} & \\
			{\small $\Tc_{1,1}^\lambda$}
			&\includegraphics[width=0.28\textwidth]{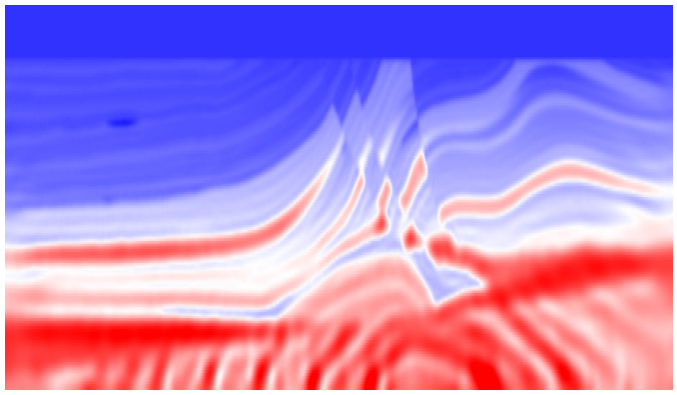} &
			\includegraphics[width=0.28\textwidth]{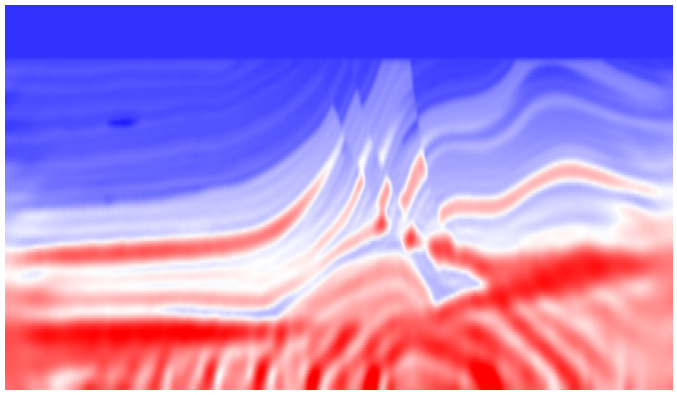} & \includegraphics[width=0.28\textwidth]{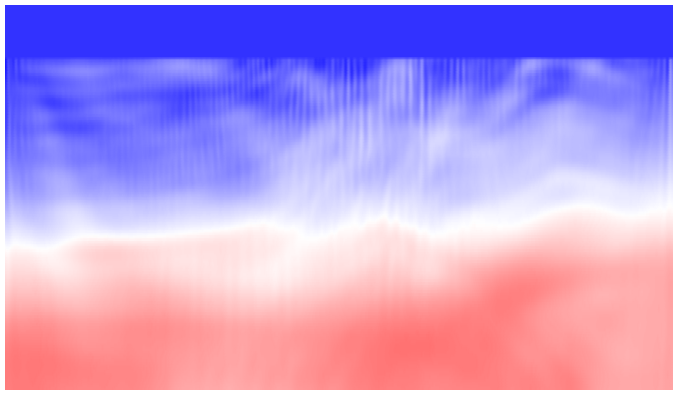} & \\
			{\small $\Tc_{1,2}^\lambda$}
			&\includegraphics[width=0.28\textwidth]{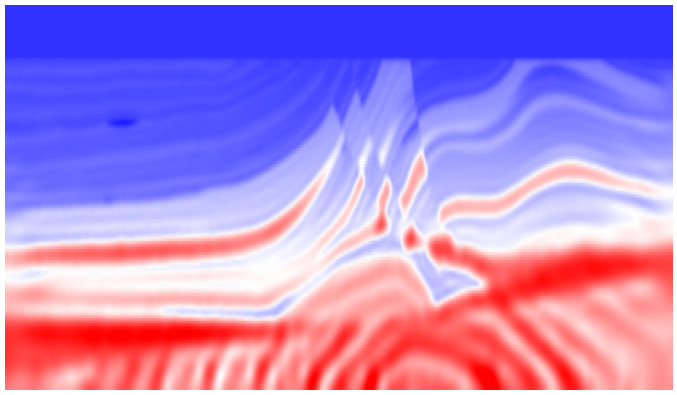} &
			\includegraphics[width=0.28\textwidth]{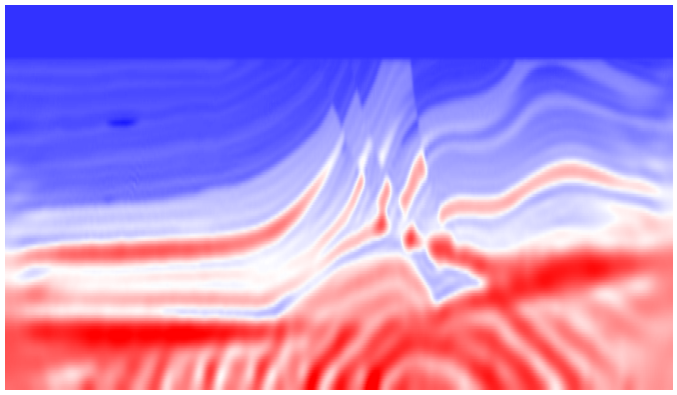} & \includegraphics[width=0.28\textwidth]{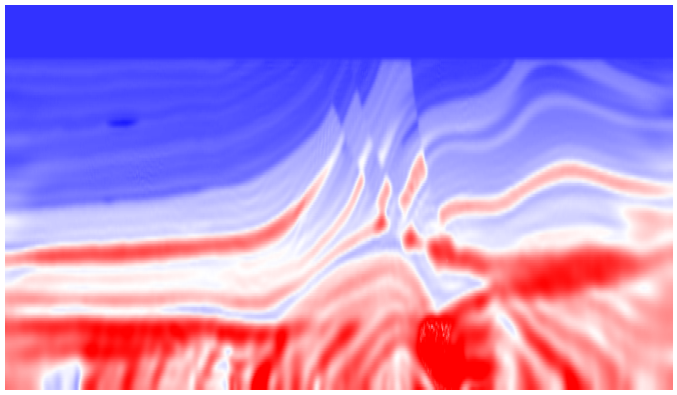} & \\
			{\small $\Tc_{2,2}^\lambda$}
			&\includegraphics[width=0.28\textwidth]{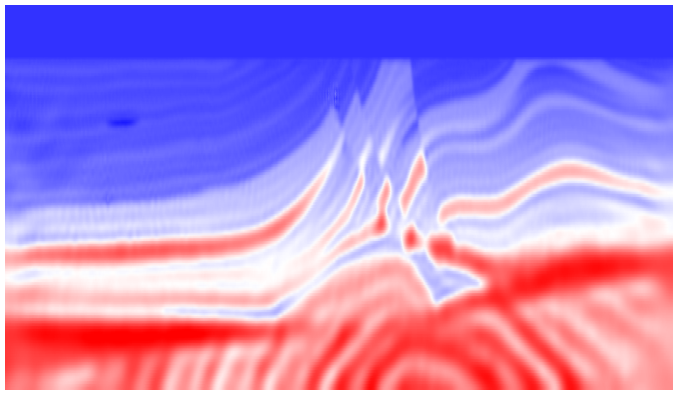} &
			\includegraphics[width=0.28\textwidth]{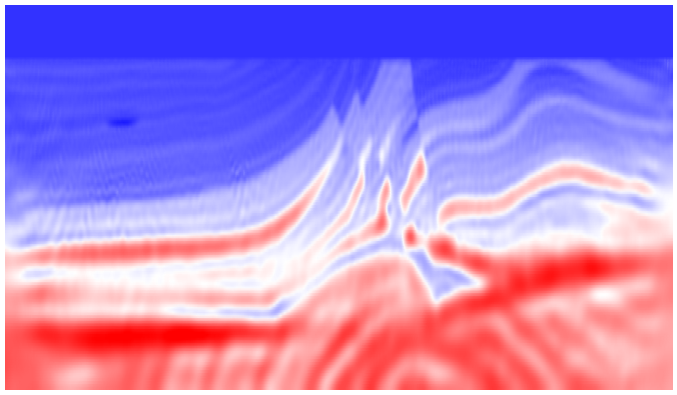} & \includegraphics[width=0.28\textwidth]{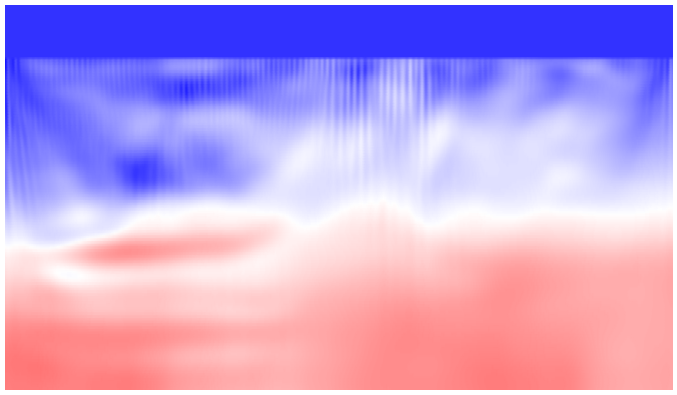} & \\
		\end{tabular}
		\caption{Comparison of the reconstruction of the Marmousi velocity model using the $L^2$ distance (first row), the $\KR$ norm \eqref{eq:KRscal} (second row) and $\mathcal{T}^{\lambda}_{p,q}$ (from third to fifth row), using several initial conditions for the inverse problem (Figure \ref{fig:models}). For the misfit $\Tc_{p,q}^\lambda$, $\lambda$ has been tuned manually to provide the better performance: we used $\lambda\approx 10^{-2}\ell$ for $(p,q)=(1,1)$, $\lambda\approx 0.05 \ell$ for $(p,q)=(1,2)$ and $\lambda\approx 10^{-3}\ell$ for $(p,q)=(2,2)$, with $\ell$ size of the squared domain (see Section \ref{ssec:seismo}).
		}
		\label{fig:marmousi}
	\end{figure}
	
	The $L^2$ distance does not present a big convexity region around the minimizer, as highlighted in Figure \ref{fig:shift}, which is the main drawback of using this simple misfit function. Consequently, the true velocity model can be retrieved only starting from a sufficiently close initial condition, else the l-BFGS algorithm would fall in a meaningless local minimum.
	Using optimal transport based distances may sensibly increase the convexity region, as recently demonstrated, thanks to their horizontal sensitivity. Using the $\KR$ norm as misfit function allows indeed to solve the problem starting from further initial conditions.
	Notice that since the $\KR$ norm is applied separately to zero-mean scalar signals, therefore with no mass imbalance, it does not suffer from a higher bound on $\lambda$ and the inverse problem is robust with respect to the choice of this parameter (as long as it is sufficiently high to allow proper displacement of mass, see Section \ref{ssec:lambda}). In this example we have chosen $\lambda=\ell$ ($\ell$ size of the squared domain, see Section \ref{ssec:seismo}). The horizontal sensitivity is however limited by the fact that the data are signed (see Examples \ref{ex:deltas_KRvect} and \ref{ex:KRshift}). By lifting signed signals into the space of measures with symmetric positive semi-definite values, the misfit $\Tc_{p,q}^\lambda$ allows to go further. Among the three configurations, only the case $(p,q)=(1,2)$ provides considerably better results. The case $(p,q)=(1,1)$, the vectorial extension of the $\KR$ norm, suffers too much from the mass imbalance to allow a meaningful calibration of $\lambda$. The case $(p,q)=(2,2)$, as shown in Section \ref{ssec:shift}, has a limited horizontal sensitivity and does not provide any advantage in terms of robustness with respect to the $\KR$ norm. On the other hand, it can be solved more efficiently with respect to the other two cases and the $\KR$ norm.
	
	\section{Conclusions}
	
	In this work we provided a misfit function for applications to inverse problems as a generalization of the $L^1$ optimal transport problem. By lifting general signed data to non-negative measures defined on a higher dimensional vector space, we recover a meaningful notion of transport leading to a more robust inversion problem.
	By choosing an appropriate lift function no information is lost on the data, differently from other previously proposed approaches.
	We considered the lift \eqref{eq:Pauli}, that is using Pauli matrices to transform a vector with signed coordinates into a positive semi-definite matrix, but other ideas can be followed, our extension providing a general framework.
	The advantage of extending the $L^1$ transport, with respect to other optimal transport problems, is that it can be computed efficiently without a further increase in the computational complexity given by the vector extension. In particular, the quadratic penalization case may offer better optimization strategies with respect to the one considered here.
	On the other hand, the vector extension requires to deal with possible imbalances of mass which make the calibration of the model difficult.
	Our analysis of the role played by the parameter $\lambda$ gives some insights on its calibration, but the model is quite sensitive to it. This in particular depends on the lift function considered and alternatives may be found to soften the problem. Moreover, the lift function may be tailored to the specific problem at hands taking into account also physical considerations in its design.
	
	The application to FWI shows good preliminary results. On the Marmousi case study, a benchmark for FWI, the misfit function we proposed reveals extremely robust being able to recover the true velocity model starting from several initial conditions which are already extremely difficult for standard approaches. Testing our model on even more difficult initial conditions did not provide interesting results. This is propably related to the aforementioned high sensitivity to the calibration of $\lambda$. We will try to tackle this issue in future works. Moreover, to fully assess the potential of our approach, the model needs to be tested on bigger and even real-world seismic data sets. This is another objective we will pursue.

	\section*{Acknowledgements}
	
	This study was partially funded by the SEISCOPE consortium (\textit{http://seiscope2.osug.fr}), sponsored by AKERBP, CGG, DUG, EXXONMOBIL, GEOLINKS, JGI, PETROBRAS, SHELL, SINOPEC and TOTALENERGIES. This study was granted access to the HPC resources provided by the GRICAD infrastructure (https://gricad.univ-grenoble-alpes.fr), which is supported by Grenoble research communities, the HPC resources of Cray Marketing Partner Network (\url{https://partners.cray.com}), and those of CINES/IDRIS/TGCC under the allocation 046091 made by GENCI.
	The work of GT was partly supported by the Bézout Labex (New Monge Problems), funded by ANR, reference ANR-10-LABX-58.

	\bibliographystyle{plain}      
	\bibliography{refs}   

\end{document}